\documentclass[reqno]{amsart}            
\usepackage{amsthm,amssymb,amsfonts,graphicx,cite,color}
\usepackage[bookmarks,bookmarksnumbered,bookmarksopen,colorlinks,backref,linkcolor=blue,citecolor=red]{hyperref}%
\usepackage{xcolor}
\usepackage{graphicx}
\usepackage{mathrsfs}

\newtheorem{theorem}{Theorem}[section]
\newtheorem{definition}{Definition}[section]
\newtheorem{lemma}{Lemma}[section]
\newtheorem{corollary}{Corollary}[section]
\newtheorem{remark}{Remark}[section]
\newtheorem{proposition}{Proposition}[section]
\numberwithin{equation}{section}

\begin{document}
	
	\title[Reactive Transport in Heterogeneously Fractured Porous Media]
	{Multiscale Hyperbolic-Parabolic Models for Nonlinear Reactive Transport in Heterogeneously Fractured  Porous Media}
	
	\author[Taras Mel'nyk \ \& \ Sorin Pop \ \& \ Christian Rohde]{Taras Mel'nyk$^\flat$ \  \& \ Sorin Pop$^\sharp$ \ \& \ Christian Rohde$^\natural$  }
	\address{\hskip-12pt  $^\flat$ 
		1) \, Institute of Applied Analysis and Numerical Simulation,
		Faculty of Mathematics and Physics, Suttgart University\\
		Pfaffenwaldring 57,\ 70569 Suttgart,  \ Germany,\\
		\ \ 2)
		Department of Mathematical Physics, Faculty of Mathematics and Mechanics\\
		Taras Shevchenko National University of Kyiv\\
		Volodymyrska str. 64,\ 01601 Kyiv,  \ Ukraine
	}
	\email{taras.melnyk@mathematik.uni-stuttgart.de}
	
	\address{\hskip-12pt  $^\sharp$ 
		Hasselt University, Faculty of Sciences\\
		Agoralaan Gebouw D, \ Diepenbeek 3590, \ Belgium
	}
	\email{sorin.pop@uhasselt.be }

	\address{\hskip-12pt  $^\natural$ Institute of Applied Analysis and Numerical Simulation,
		Faculty of Mathematics and Physics, Suttgart University\\
		Pfaffenwaldring 57,\ 70569 Suttgart,  \ Germany
	}
	\email{christian.rohde@mathematik.uni-stuttgart.de }

	\begin{abstract}
		We study nonlinear reactive transport in a layered porous medium that is separated into two bulk domains by  a thin, geometrically highly heterogeneous fracture. Precisely, the fracture is formed by  an 
		$\varepsilon\!\!$-thin channel  whose aperture and  internal obstacle pattern vary periodically on the same microscale. 
		The reactive transport of the  species  in the two bulk domains is modelled by a system of  parabolic reaction-diffusion equations. It couples to a convection-diffusion-reaction problem in the thin fracture  accounting for nonlinear chemical reactions on the    fracture's  wall and obstacle 
		boundaries.  Notably,  the  species transport inside the fracture  is convection dominated with Péclet number of order $\varepsilon^{-1}$.\\
		We perform a multiscale asymptotic analysis as $\varepsilon \to 0$ and are led, in the limit, to 
		a new type of homogenized model in which the thin layer collapses to a flat interface. The homogenized limit consists of classical diffusion–reaction equations in the bulk domains coupled through homogenized nonlinear interface conditions and a first-order semilinear hyperbolic system posed on the flat interface. The latter encodes the effective fracture dynamics produced by the  strong advection and the microscale geometry. The following aspects are addressed in this study: firstly, the well-posedness of the homogenized system is proven, including the derivation of regularity results; secondly, an explicit multiscale approximation incorporating boundary-layer correctors is constructed; and thirdly, quantitative error estimates in appropriate energy norms are established. These three aspects rigorously justify the reduced model and quantify the approximation accuracy with respect to the small parameter $\varepsilon.$
	\end{abstract}
	
	\maketitle
	\tableofcontents

	\section{Introduction}
	\newcommand{\eps}{\varepsilon}
	We consider  reactive species transport in fractured porous media. As a fracture we understand   structures that have a  small aperture relative to their length. Fractures are ubiquitous in natural and engineered porous media and strongly influence flow and transport of dissolved species. In many applications, such as hydrogeology, contaminant remediation, and subsurface reactive flows, fractures act as preferential pathways with much larger convective transport than the surrounding matrix. In reality, fractures are often geometrically complex; for example, their aperture may vary rapidly and they may  contain embedded obstacles, possibly inducing several spatial scales. 

	These microscale heterogeneities complicate direct numerical simulation and motivate rigorous upscaling  of the microscale models to derive effective macroscopic models, in which the thin fracture is reduced to a lower-dimensional interface.
	
	In this contribution we are interested in chemically reactive transport of species in a two-dimensional porous medium that 
	is separated by a horizontal channel-like fracture into two \textit{bulk domains}. For some  small parameter $\varepsilon,$ the  fracture has ${\mathcal O}(\varepsilon)\!$-aperture. As geometrical heterogeneities we account for the fracture's  wall roughness (differently on the fracture's  upper and lower side) and some  internal obstacle pattern which are assumed to vary periodically on the same microscale. For the sake of illustration we refer to Fig.~\ref{fig1}.
	\begin{figure}[htbp]  
		\vspace*{-0.3cm}
		\begin{center}
			\includegraphics[width=7cm]{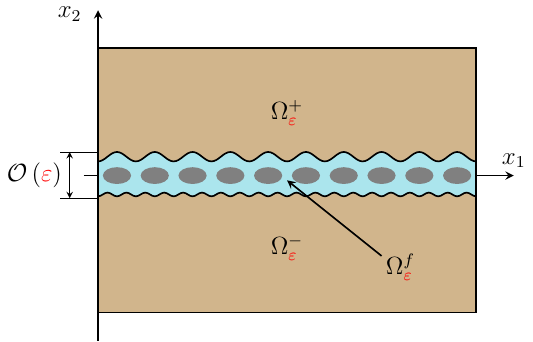}
		\end{center}
		\vspace*{-0.4cm}
		\caption{Sketch of the two bulk domains and the thin fracture.}\label{fig1}
	\end{figure}
	
	The reactive transport of the species in the two bulk domains is modelled by a system of parabolic 
	reaction-diffusion equations which  couples to a convection-diffusion-reaction problem in the thin fracture accounting for
	nonlinear chemical reactions on the fracture's boundaries.  We assume that the 
	species transport  in the  fracture is  convection-dominated, i.e., the Péclet number is proportional to 
	the inverse of $\eps$. The coupling between the bulk domains and the fracture is done by  nonlinear side‑dependent flux conditions on the rough boundaries whereas nonlinear Robin conditions are employed 
	on the obstacle boundaries. The complete microscale model is given in  \eqref{probl} below.
	Up to our knowledge such a setting has not been analyzed before with respect to the limit $\varepsilon \to 0$.

	Due to the nonlinear interface laws and convection-dominated transport in the fracture, the asymptotic analysis leads to qualitatively different homogenized behaviour compared to previous works. 
	The homogenized limit (see \eqref{homo-problem}) consists of nonlinear diffusion–reaction equations in the bulk domains, coupled through homogenized nonlinear interface conditions, together with a first‑order semilinear hyperbolic system posed on the flat interface that captures the effective fracture dynamics generated by strong advection and the microscale geometry. The well‑posedness of this homogenized system is established in Theorem~\ref{Th-4-1}, and the $C^3$-regularity of its solution, required for proving the error estimates, is derived in Lemma~\ref{Th-4-4}.
	
	Another novel contribution 	is the  complete approximation for the microscopic solution, i.e., 	we construct an approximation (see \eqref{app-functions}) that incorporates boundary‑layer correctors near the oscillating boundaries in the bulk domains (\S~\ref{Par-4-2}) as well as additional correctors near the vertical boundaries of the fracture  (\S~\ref{BLPs}) to satisfy the prescribed Dirichlet condition.

	Finally, we obtain error estimates of order $\mathcal{O}(\sqrt{\varepsilon })$ in the energy norm (see Theorem~\ref{main theorem}), thereby rigorously justifying the adequacy of the homogenized model. These estimates clearly reveal the influence of the interface microstructure through the second‑order terms of the asymptotic approximation, including the boundary‑layer correctors, and thus provide deeper insight into how microscopic features and convection‑dominated transport shape the macroscopic behavior. 
	
	\medskip 
	
	We next review the existing literature on homogenization in fractured porous media.
	 The literature pertaining to such issues has expanded considerably over the past two decades such that we restrict ourselves to the most closely related works and the differences between them and our own. For broader context, we  refer to works on related classes of multiscale problems, including thin‑domain models \cite{Arrieta-2025,Macrone-2024,MvD11,Mel-Roh_Non-Diff-2024}, models in junctions of thin structures \cite{Bun-Gau-2026,Gau-Pan-2026,Mel-Roh_AsAn-2024,Mel-Roh_JMAA-2024}, and homogenization problems in porous media \cite{AMP10,Bha-Gah-Rad-2022,Gahn-Radu-Pop-2021,Gahn-Pop-2023,Donato-2025,Melnyk-2026,RdBCM25}.  Among those, the papers \cite{AMP10,MvD11,Mel-Roh_AsAn-2024,Mel-Roh_JMAA-2024,Mel-Roh_Non-Diff-2024,RdBCM25} address convection-dominated transport, but not in a fractured porous-media geometry like in Fig.~\ref{fig1}.
	It is worth noting that only the results in \cite{MvD11,Mel-Roh_AsAn-2024,Mel-Roh_JMAA-2024,Mel-Roh_Non-Diff-2024} yield hyperbolic dispersion equations in the limit.
	
	A significant body of research has been dedicated to the study of problems involving the  separation of two bulk domains by  an  $\varepsilon\!$-thin fracture. These problems have been examined from the perspective of effective interface laws for the limit $\eps\to 0.$ We  restrict ourselves  to the most recent contributions \cite{Gahn-Radu-2018,Gahn-Jag-Radu-2022,Gahn-Radu-2025,Gahn-2025,Freud-Eden-2024,Freud-Eden-2025,Lis-Kum-Pop-Rad-2020,Pop-Bog-Kum-2017} that provide rigorous asymptotic justification of dimensionally-reduced models in the the limit as $\varepsilon \to 0$ (see also the fairly comprehensive reviews contained in these papers).
	
	The works \cite{Gahn-Radu-2018,Gahn-Jag-Radu-2022,Gahn-Radu-2025,Gahn-2025,Freud-Eden-2024,Freud-Eden-2025} derive macroscopic models equipped with effective interface conditions and establish weak and strong (two‑scale) convergence of the microscopic solutions. 
	The results under consideration are obtained by means of a priori estimates in conjunction with the concept of two-scale convergence for thin heterogeneous layers, as introduced in \cite{Neuss-Jaeger-2007}. 
	In \cite{Pop-Bog-Kum-2017}, by averaging in the transversal direction of the thin rectangular fracture  and using a formal asymptotic approach, a reduced model is obtained for reactive transport with prescribed shear convection in the thin rectangle and with nonlinear constitutive transmission conditions on the interface, and convergence is rigorously established  for the moderate P\'eclet regime of order $\mathcal{O}(1).$ The same analytical techniques are used in \cite{Lis-Kum-Pop-Rad-2020} to derive effective models for flows governed by Richards equations with different scalings of fracture‑to‑matrix porosity and fracture‑to‑matrix permeability, together with continuous transmission conditions.
	
	Fractures with more complex microscale structure are addressed  in \cite{Hoerl-Rohde-2024}, however restricted to  linear Darcy flow. The work  \cite{Gahn-Jag-Radu-2021}  analyses   nonlinear reactive transport in  porous media that involve a  thin heterogeneous layer. Explicit approximations of the microscopic solution to a nonlinear reaction-diffusion equation are constructed and error estimates in the energy norm are established. Aside from \cite{Gahn-Jag-Radu-2021} and up to our knowledge, only the older paper \cite{Pan-1981} offers this level of quantitative justification. In fact, our methodological approach is 
	close to that of \cite{Gahn-Jag-Radu-2021,Pan-1981}.
	 
	The authors of  \cite{Gahn-Radu-2018} consider  purely reactive–diffusive transport through a thin rectilinear heterogeneous layer for different scalings  in the diffusion matrix of the  transport system in the fracture, together with nonlinear side-dependent flux conditions on the interfaces.
	Similar to the analysis in \cite{Gahn-Radu-2018}, macroscopic transport equations for the solute concentration are obtained in \cite{Gahn-Radu-2025}, where a coupled model for fluid flow and reactive solute transport in a three‑dimensional domain is considered.
	
	None of the above‑mentioned works consider convection‑driven transport in a fractured porous medium, which leads to a hyperbolic limit problem as $\varepsilon \rightarrow 0.$ Moreover, the geometric configuration, transport mechanisms, and analytical framework in our setting differ significantly from those in the existing literature.
	
	\medskip

	The paper is organized as follows. Section~\ref{Sec-2} introduces the geometric configuration of the domain, formulates the assumptions on the data together with their mathematical and physical motivation, and presents the microscopic problem in precise form.
	Section~\ref{Sect-3} carries out the formal asymptotic analysis. We first analyze the thin heterogeneous fracture (\S~\ref{Par-4-1}) and then the bulk domains (\S~\ref{Par-4-2}), establishing in each case the well‑posedness of the associated model problems. These results allow us to identify the subsequent terms in the asymptotic ansatzes for the microscopic solution, obtained as solutions of cell boundary‑value problems generated by the periodic fracture structure and as solutions of boundary‑layer problems near the oscillatory interfaces in the bulk domains.
	Section~\ref{Sect-4} collects the relations derived in Section~\ref{Sect-3} into the homogenized problem, for which we prove existence and uniqueness of a weak solution and then establish its  smoothness properties in \S~\ref{par-exist}.
	In \S~\ref{BLPs}, we construct boundary‑layer asymptotics near the right side of the thin fracture to enforce the Dirichlet condition on that part of the boundary. In Section~\ref{Sect-5}, using these asymptotic ansatzes together with the method of matching asymptotic expansions, we construct a global approximation of the microscopic solution, compute the residuals it leaves in the original problem, and establish asymptotic estimates for the difference between the microscopic solution and the approximation in various norms, including the energy norm. Finally, Section~\ref{Sect-6} discusses results and how the proposed approach can be applied to other problems, and it also outlines directions for future research.

	\section{Problem statement}\label{Sec-2}
	
	\paragraph*{{\bf Domain structure}}
	Consider the rectangle $\Omega := (0, \ell)\times(-\mathfrak{h}^-, \, \mathfrak{h}^+),$ where $\mathfrak{h}^\pm >0,$ inside which a thin perforated domain (fracture) $\Omega^f_\varepsilon$ with rapidly varying boundaries is embedded. This fracture divides $\Omega$ into two bulk subdomains $\Omega^+_\varepsilon$ and $\Omega^-_\varepsilon$ (see e.g. Fig.~\ref{fig1}).
	We begin by describing the thin domain $\Omega^f_\varepsilon.$  
	
	Let $h_+$ and $h_-$ be $C^{2,\mu}$-smooth, positive, $1$-periodic, and even functions on $\mathbb{R};$ here $\mu\in (0,1).$ We  additionally assume that $h_\pm = 1$ in a small neighborhood of $0.$   
		Define the domain
	\[
	Y := \left\{ \xi = (\xi_1, \xi_2) \in \mathbb{R}^2 \colon\ 0 < \xi_1 < 1,\ -h_-(\xi_1) < \xi_2 < h_+(\xi_1) \right\}.
	\]
	
	Let $T_0$ be a finite union of $C^{2,\mu}$-smooth, disjoint, non-tangent subdomains strictly contained in $Y.$  
	We then define the periodicity cell as $Y_0 := Y \setminus \overline{T_0}.$  
	
	To describe parts of the boundary $\partial Y_0,$ we introduce the following notation: the sets
	\[
	S^\pm := \left\{ \xi \in \mathbb{R}^2 \colon\ \xi_2 = \pm h_\pm(\xi_1),\ 0 < \xi_1 < 1 \right\}
	\]
	denotes either the upper ($+$) or lower ($-$)  part of the boundary, depending on the sign;   
	the symbol $\partial T_0$ refers to the boundaries of the holes inside $Y_0$; and
	\[
	\Gamma := \partial Y_0 \setminus \left( S^\pm \cup \partial T_0 \right)
	\]
	represents the union of the vertical segments of $\partial Y_0.$
	
	Introducing a small parameter $\varepsilon := \frac{\ell}{N},$ where $N$ is a large positive integer, and taking into account the $1$-periodicity of $h_\pm,$ the functions $h_+(\frac{x_1}{\varepsilon})$ and $h_-(\frac{x_1}{\varepsilon})$ $(x_1\in \Bbb R)$  are  $\varepsilon$-periodic. Define the thin domain
	\begin{equation}\label{cylindr}
		Q_\varepsilon :=\Big\{x=(x_1, x_2) \in \Bbb R^2\colon \ \ x_1 \in (0, \ell), \quad - \varepsilon\, h_-\Big(\frac{x_1}{\varepsilon}\Big) < x_2 < \varepsilon \, h_+\Big(\frac{x_1}{\varepsilon}\Big)\Big\},
	\end{equation}
	which has thickness of order $\mathcal{O}(\varepsilon)$ and $\varepsilon$-periodic oscillations of the upper and lower parts of the boundary. 
	
	Let
	$$
	T := \bigcup_{k\in \Bbb N_0} \big( k \, \vec{\boldsymbol{e}}_1 + T_0\big),
	$$
	where $\vec{\boldsymbol{e}}_1= (1, 0).$ We denote  by ${T}_\varepsilon := \varepsilon \, T$ the homothetic transformation of $T$ with  the coefficient $\varepsilon.$ Then, the thin perforated domain is defined by 
	$$
	\Omega^f_\varepsilon := Q_\varepsilon \setminus  \overline{T_\varepsilon}.
	$$
	
	\vspace{-0.5cm}
	\begin{figure}[htbp]
		\begin{center}
			\includegraphics[width=12cm]{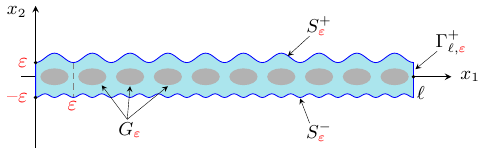}
		\end{center}
		
		\vspace{-0.5cm}
		\caption{ Schematic illustration of the thin fracture $\Omega^f_\varepsilon$}\label{fig-2}
	\end{figure}

	The boundary $\partial \Omega^f_\varepsilon$ consists of the following parts (see Fig.\ref{fig-2}):
	\begin{itemize}
		\item $G_\varepsilon := \partial T_\varepsilon \cap Q_\varepsilon$ is the boundaries of holes ${T}_\varepsilon,$
		\item $S^\pm_\varepsilon := \big\{x\in \Bbb R^2\colon \ \  x_2 =  \pm \varepsilon\,  h_\pm\big(\frac{x_1}{\varepsilon}\big), \ \  x_1 \in (0, \ell) \big\}$ are upper/lower oscillating parts, respectively,
		\item $\Gamma_\varepsilon := \partial\Omega_\varepsilon \setminus \big(G_\varepsilon \cup S^+_\varepsilon \cup S^-_\varepsilon\big)$ is the union of two vertical equal intervals  $\Gamma_{0,\varepsilon}$ and $\Gamma_{\ell,\varepsilon}.$
	\end{itemize}
	
	Now we define the bulk domains $\Omega^+_\varepsilon$ and $\Omega^-_\varepsilon:$
	$$
	\Omega^\pm_\varepsilon := \{x\colon \ \pm x_2 > 0\} \bigcap \big(\Omega \setminus \Omega^f_\varepsilon\big).
	$$
	
	In the paper, the index  "$+$"  at the top will always indicate the connection with the upper part of $\Omega,$ the index "$-$" with  the lower part,
	and the index "$f$" with the thin domain (fracture) $\Omega^f_\varepsilon$ or the  interval
	$$
	\mathcal{I}^f := \{x\colon\ x_1 \in (0, \ell), \ \ x_2 =0\},
	$$
	to which the domain $\Omega^f_\varepsilon$ shrinks as $\varepsilon \to 0.$
	
	\begin{remark}
	Instead of the rectangle $\Omega,$ we may consider a planar domain with smooth boundary; the only requirement is that its intersection with $\{ x\colon |x_2|<\delta \},$ for some $\delta >0,$ coincides with the rectangle $\{x\colon x_1 \in (0, \ell), \ |x_2| < \delta\}.$
	We choose the rectangle $\Omega$ to simplify the description of the problem and to reduce some of the technical arguments in the proof of Lemma~\ref{Th-4-4}.
	\end{remark}
	
	\paragraph*{{\bf Problem}} We consider the following parabolic  problem:
	\begin{equation}\label{probl}
		\left\{\begin{array}{rcll}
			\partial_t u^\pm_{k,\varepsilon} -  D^\pm_k  \Delta u^\pm_{k,\varepsilon}
			& = & F^\pm_k(\mathbf{u}^\pm_\varepsilon, x, t) + \mathfrak{f}_k^\pm(x,t) &
			\text{in} \ \ \Omega^{\pm, T}_\varepsilon,
			\\[2mm]
			u^\pm_{k,\varepsilon}
			& = & 0, & \text{on} \ \partial\Omega^{\pm, T}_\varepsilon \setminus S^{\pm, T}_\varepsilon,
			\\[2mm]
			D^\pm_k \, \nabla u^\pm_{k,\varepsilon} \cdot \boldsymbol{\nu}_\varepsilon &=&  \Upsilon^\pm_k\big( \mathbf{u}^\pm_\varepsilon, \mathbf{u}^f_\varepsilon, \tfrac{x}{\varepsilon}, x_1, t \big) &  \text{on} \ \   S^{\pm, T}_\varepsilon,
			\\
			
			\\
			\partial_t u^f_{k,\varepsilon} -  \varepsilon\, \nabla_x \cdot\big( \mathbb{D}_k(\tfrac{x}{\varepsilon}) \, \nabla_x u^f_{k,\varepsilon}\big) +
			\nabla_x \cdot \big( \overrightarrow{V_\varepsilon}(x) \, u^f_{k,\varepsilon}\big) &=& 0  & \text{in} \ \ \Omega^{f, T}_\varepsilon,
			\\[2mm]
			-  \varepsilon \,  \mathbb{D}_k(\tfrac{x}{\varepsilon})\nabla_x u^f_{k,\varepsilon}  \cdot \boldsymbol{\nu}_\varepsilon &  = & \varepsilon^\alpha \, \Phi^\pm_k\big(\mathbf{u}^\pm_\varepsilon, \mathbf{u}^f_\varepsilon, \tfrac{x}{\varepsilon}, x_1, t\big) &
			\text{on} \ S^{\pm, T}_\varepsilon,
			\\[2mm]
			-  \varepsilon \,  \mathbb{D}_k(\tfrac{x}{\varepsilon})\nabla_x u^f_{k,\varepsilon} \cdot \boldsymbol{\nu}_\varepsilon &  = & \varepsilon^\beta \, \Psi_k\big(\mathbf{u}^f_\varepsilon, \tfrac{x}{\varepsilon}, x_1, t\big) &
			\text{on} \ G^T_\varepsilon,
			\\
			
			\\
			u^f_{k,\varepsilon}\big|_{x_1= 0} = 0 \ \ \text{on} \ \Gamma^T_{0,\varepsilon},
			& & 
			u^f_{k,\varepsilon}\big|_{x_1=\ell}
			=  q^\ell_k(t), \ \  \text{on} \ \Gamma^T_{\ell,\varepsilon}, &k\in \{1,\ldots, \mathcal{M}\},
			\\[4mm]
			\mathbf{u}^\pm_\varepsilon\big|_{t=0} \ = \ \mathbf{u}^f_\varepsilon\big|_{t=0}&=& \mathbf{0} & \text{on} \ \Omega_{\varepsilon},
		\end{array}\right.
	\end{equation}
	where the vector-functions $\mathbf{u}^\pm_\varepsilon = \big(u^\pm_{1,\varepsilon},\ldots,u^\pm_{\mathcal{M},\varepsilon}\big)$ and
	$\mathbf{u}^f_\varepsilon = \big(u^f_{1,\varepsilon},\ldots,u^f_{\mathcal{M},\varepsilon}\big)$ are the unknown concentration densities of various chemicals, parameters $\alpha$  and $\beta$ are greater than or equal to $1,$  $\boldsymbol{\nu}_\varepsilon$ is the outward unit normal to the boundary of $\Omega^f_\varepsilon,$ $\frac{x}{\varepsilon} =(\frac{x_1}{\varepsilon}, \frac{x_2}{\varepsilon}),$ and 
	\begin{gather*}
		\Omega^{\pm, T}_\varepsilon := \Omega^\pm_\varepsilon \times (0, T],\quad 
		\partial\Omega^{\pm, T}_\varepsilon := \partial\Omega^\pm_\varepsilon \times (0, T], \quad
		S^{\pm, T}_\varepsilon := S^\pm_\varepsilon \times (0, T],
		\\
		\Omega^{f, T}_\varepsilon := \Omega^{f}_\varepsilon \times (0, T],\quad
		G_\varepsilon^T := G_\varepsilon \times (0, T], \quad \Gamma^T_{0,\varepsilon} := \Gamma_{0,\varepsilon}\times (0, T], \quad 
		\Gamma^T_{\ell,\varepsilon} := \Gamma_{\ell,\varepsilon}\times (0, T].
	\end{gather*}
		
	\paragraph*{{\bf Assumptions on the data}}
	\begin{description}
		\item[A1] 
		For each \(k = 1,\dots,\mathcal{M}\), the bulk diffusion coefficients \(D_k^+\) and \(D_k^-\) are strictly positive, and the microscopic diffusion tensor
		\[
		\mathbb{D}_k\!\bigl(\tfrac{x}{\varepsilon}\bigr)
		= \bigl\{d^k_{ij}(\tfrac{x}{\varepsilon})\bigr\}_{i,j=1}^2
		\]
		is symmetric and uniformly positive definite.  Moreover, each component \(d^k_{ij}(\xi)\) is smooth and \(1\)-periodic in $\xi_1$  on \(\overline{Y}_0\), and $\mathbb{D}_k$ is a diagonal constant matrix in a small neighborhood of the interval $\{\xi\colon \xi_1 =0, \ \xi_2 \in [-1,1]\}.$
	\end{description}
	
	\begin{remark}
		In the bulk regions the diffusion coefficients are constant.  Extending these constants to smooth, positive-definite tensors poses no additional difficulty.
	\end{remark}

	\begin{description}
		\item[A2] The given convective velocity field, $\overrightarrow{V_\varepsilon} = \left(v_1(\frac{{x}}{\varepsilon}), \  v_2(\tfrac{{x}}{\varepsilon})
		\right)$
		is conservative  and
		its potential $p$ is a solution to the boundary-value problem
		\begin{equation}\label{potential0}
			\left\{\begin{array}{rcll}
				\Delta_\xi p(\xi)  & = & 0, & \quad
				\xi= (\xi_1, \xi_2) \in Y_0,
				\\[3pt]
				\partial_{\boldsymbol{\nu}_\xi} p(\xi)  &=&  0, & \quad
				\xi \in S^\pm \cup  \partial T_0,
				\\[4pt]
				\partial_{\xi_1} p(\xi)\big|_{\xi_1=0}   &=& \partial_{\xi_1} p(\xi)\big|_{\xi_1=1} \  = \ \mathrm{v}_0(\xi_2), & \quad
				\xi_2  \in [-1, \, 1],
				\\[4pt]
				\langle p \rangle_{Y_0} &=& 0, &
			\end{array}\right.
		\end{equation}
		where $\mathrm{v}_0$ is a smooth,  positive function on the closed interval $ [-1, 1]$ and $\mathrm{v}'_0(\pm 1) = 0,$
		$\Delta_\xi$ is  the Laplace operator, $\partial_{\xi_1} = \frac{\partial}{\partial \xi_1},$   $\partial_{\boldsymbol{\nu}_\xi}$ is the derivative along the  outward unit normal $\boldsymbol{\nu}_\xi$ to $\partial Y_0,$ and
		$$
		\langle p \rangle_{Y_0} := \frac{1}{\upharpoonleft\!\! Y_0 \!\!\upharpoonright_2} \int_{Y_0} p(\xi)\, d\xi.
		$$
		Hereinafter, the symbol $\upharpoonleft\!\! \mathcal{S} \!\!\upharpoonright_n$ denotes  the $n$-dimensional Lebesgue measure of a set $\mathcal{S}$ $(n\in \{1, 2\}).$
		It is easy to verify that  problem \eqref{potential0} has a unique smooth solution. Thus,
		$$
		\overrightarrow{V}(\xi) = \nabla_\xi p(\xi), \quad \xi \in \overline{Y}_0, \quad \text{and} \quad
		\overrightarrow{V_\varepsilon} = \overrightarrow{V}(\xi)\big|_{\xi=\frac{x}{\varepsilon}}, \ \ x \in \Omega^f_\varepsilon.
		$$
	\end{description}
	\begin{remark}
		If $Y_0 = (0, 1)\times (-1, 1)$ and $\mathrm{v}_0$ is a constant,  then the potential $p(\xi) = \mathrm{v}_0 \, \xi_1$ and the velocity field $\overrightarrow{V} = (\mathrm{v}_0, 0)$ represents a shear flow.
		Shear flows naturally arise in various physical contexts (see, e.g., \cite{Maj_Kra_1999}). As noted in that work, a particularly valuable extension of the shear flow model is the inclusion of a nonzero transverse sweep.
		
		The roughness and perforated structure of the fracture $\Omega^f_\varepsilon$ clearly influence the convective flow. To account for these irregularities, we consider conservative convective flows that incorporate nonzero transverse sweeps.
		
		It is evident that both shear flows and the flow under consideration are incompressible $(\nabla_\xi \boldsymbol{\cdot} \overrightarrow{V}(\xi) = 0).$
		In our case, the inflow of $\overrightarrow{V}$ into $Y_0$ from the left is equal to the outflow from $Y_0$ on the right.
		
	\end{remark}
	
	\begin{description}
		\item[A3] For each $k\in \{1,\ldots,\mathcal{M}\}$, the function $\mathfrak{f}_k^\pm$ and the nonlinear functions 
		$F^\pm_k$, $\Upsilon^\pm_k$, $\Phi^\pm_k$, and $\Psi_k$, which describe the chemical processes in our problem, 
		are $C^3$-smooth in their respective domains of definition and are uniformly bounded together with all their derivatives. 
		In addition, 
		$
		\mathfrak{f}_k^\pm(x,0)=0,
		$
		and
		\begin{equation}\label{init-conds}
			F_k^\pm(\mathbf{0},x,t)
			= \Upsilon_k^\pm(\mathbf{0},\mathbf{0},\xi,x_1,t)
			= \Phi_k^\pm(\mathbf{0},\mathbf{0},\xi,x_1,t)
			= \Psi_k(\mathbf{0},\xi,x_1,t)
			= 0.
		\end{equation}
		Moreover, $\mathfrak{f}_k^\pm$ and $F^\pm_k$ have compact support in $x\in\Omega^\pm$, while 
		$\Upsilon^\pm_k$, $\Phi^\pm_k$, and $\Psi_k$ have compact support in $x_1\in \mathcal{I}^f$, independent of the other variables.
		
		Finally, there exists a small $\delta>0$ such that, for all $t\in[0,\delta]$,
		\begin{equation}\label{init-conds+}
			\Upsilon_k^\pm(\cdot,\cdot,\cdot,\cdot,t)
			= \Phi_k^\pm(\cdot,\cdot,\cdot,\cdot,t)
			= \Psi_k(\cdot,\cdot,\cdot,t)
			= 0.
		\end{equation}
	\end{description}
	
	\begin{description}
		\item[A4]
		Functions  $\{q^\ell_k(t), \ t\in [0, T]\}_{k=1}^\mathcal{M} $ are nonnegative and $C^2$-smooth, and satisfy 
		\begin{equation}\label{match_conditions}
			q^\ell_k(0) =  \frac{d q^\ell_k}{dt}(0) = 0, \quad k\in \{1,\ldots,M\}.
		\end{equation}
	\end{description}
	\begin{remark}\label{rem-2-3}
		As will be seen later, the $C^3$-smoothness of the nonlinear functions is required in order to determine 
		higher-order terms in the asymptotic expansion of the solution to problem~\eqref{probl}. 
		In particular, this regularity guarantees Lipschitz continuity; for example,
		\begin{equation}\label{Lip-ineq}
			\big|\Upsilon^+_k(\mathbf{u}^+, \mathbf{u}^f, \xi, x_1,t) 
			- \Upsilon^+_k(\mathbf{v}^+, \mathbf{v}^f, \xi, x_1,t)\big|
			\le C\big(|\mathbf{u}^+ - \mathbf{v}^+| + |\mathbf{u}^f - \mathbf{v}^f|\big),
		\end{equation}
		with a constant $C$ independent of the arguments. 
		The compact-support assumptions model localized reactions (e.g., solute penetration from the solid into the fracture) 
		and automatically enforce the zero- and first-order matching conditions in~\eqref{probl}. 
		The relations~\eqref{init-conds+} mean that chemical activity in the fracture begins only after a delay, 
		which is natural in applications, and they are also important  in establishing the regularity of the solution 
		to the homogenized problem.
	\end{remark}

	From a physical point of view, problem \eqref{probl} can be interpreted as the evolution of the densities of various chemical species
	$\mathbf{u}^\pm_\varepsilon = \big(u^\pm_{1,\varepsilon},\ldots,u^\pm_{\mathcal{M},\varepsilon}\big)$ and
	$\mathbf{u}^f_\varepsilon = \big(u^f_{1,\varepsilon},\ldots,u^f_{\mathcal{M},\varepsilon}\big)$
	in the bulk domains $\Omega_\varepsilon^\pm$ and in the fracture $\Omega_\varepsilon^f$, respectively.
	The process begins in the bulk domains, where each species is subject to the source term
	 $\boldsymbol{\mathfrak{F}}^\pm= (\mathfrak{f}_1^\pm,\ldots,\mathfrak{f}_\mathcal{M}^\pm)$ and the reaction term $\mathbf{F}^+(\mathbf{u}^\pm_\varepsilon,x,t)=\big(F^+_1(\mathbf{u}^\pm_\varepsilon,x,t),\ldots,F^+_\mathcal{M}(\mathbf{u}^\pm_\varepsilon,x,t)\big)$. 
	As diffusion proceeds, the species reach the interfaces $S_\varepsilon^\pm$ with the fracture, where 
	additional interfacial reactions occur. These reactions describe exchange between the bulk and the fracture, and are 
	governed by the boundary flux terms  
	\begin{gather*}
		\mathbf{\Upsilon}^\pm\big(\mathbf{u}^\pm_\varepsilon,\mathbf{u}^f_\varepsilon, \tfrac{x}{\varepsilon}, x_1, t\big) =
		\big(\Upsilon^\pm_1(\mathbf{u}^\pm_\varepsilon, \mathbf{u}^f_\varepsilon, \tfrac{x}{\varepsilon}, x_1, t),\ldots,\Upsilon^\pm_\mathcal{M}(\mathbf{u}^\pm_\varepsilon, \mathbf{u}^f_\varepsilon, \tfrac{x}{\varepsilon}, x_1, t)\big),    
		\\
		\mathbf{\Phi}^\pm\big(\mathbf{u}^\pm_\varepsilon,\mathbf{u}^f_\varepsilon, \tfrac{x}{\varepsilon}, x_1, t\big) =
		\big(\Phi^\pm_1(\mathbf{u}^\pm_\varepsilon, \mathbf{u}^f_\varepsilon, \tfrac{x}{\varepsilon}, x_1, t),\ldots,\Phi^\pm_\mathcal{M}(\mathbf{u}^\pm_\varepsilon, \mathbf{u}^f_\varepsilon, \tfrac{x}{\varepsilon}, x_1, t)\big).    
	\end{gather*}
	
	Inside the fracture, the transport is more complex: diffusion is scaled by $\varepsilon,$ and the flow field $\overrightarrow{V_\varepsilon}$ induces convection along the fracture from left to right. The 
	geometry is perforated, so the species also interact with the boundaries of the micro-holes. These additional 
	reactions are described by $\mathbf{\Psi}\big(\mathbf{u}^f_\varepsilon, \tfrac{x}{\varepsilon}, x_1, t\big) =
	\big(\Psi_1(\mathbf{u}^f_\varepsilon, \tfrac{x}{\varepsilon}, x_1, t),\ldots,\Psi_\mathcal{M}(\mathbf{u}^f_\varepsilon, \tfrac{x}{\varepsilon}, x_1, t)\big).$  Finally, boundary conditions at $x_1=0$ and $x_1=\ell$ prescribe inflow and outflow of species along the fracture.
	
	\subsubsection*{Weak formulation}	
		Denote by $\boldsymbol{\mathfrak{H}}^{\pm, *}_\varepsilon$ the dual space to the Sobolev vector space 
	$$
	\boldsymbol{\mathfrak{H}}^{\pm}_\varepsilon := \big\{ \mathbf{u}\in H^1(\Omega^\pm_\varepsilon; \mathbb{R}^\mathcal{M})\colon \ \mathbf{u}\big|_{\partial\Omega^\pm_\varepsilon \setminus S^\pm_\varepsilon} = \mathbf{0} \big\},
	$$ 
	and by 
	$\boldsymbol{\mathfrak{H}}^{f, *}_\varepsilon$ the dual space to  
	$\boldsymbol{\mathfrak{H}}^{f}_\varepsilon := \big\{ \mathbf{u}\in H^1(\Omega^f_\varepsilon; \mathbb{R}^\mathcal{M})\colon \ \mathbf{u}\big|_{\Gamma_{0,\varepsilon}} = \mathbf{0} \big\}.$ The brackets $\langle\cdot,\cdot\rangle^\pm_\varepsilon$ and $\langle\cdot,\cdot\rangle^f_\varepsilon$ denote the pairing of the corresponding spaces.

	\begin{definition}\label{Def-weak-sol}
		We say that a vector-function $\mathbf{u}_\varepsilon := \big(\mathbf{u}^+_\varepsilon, \mathbf{u}^-_\varepsilon, \mathbf{u}^f_\varepsilon \big)$ is a \emph{weak solution} to problem \eqref{probl} if
		$$
		\mathbf{u}^\pm_\varepsilon \in L^2(0,T; \boldsymbol{\mathfrak{H}}^{\pm}_\varepsilon), \quad \partial_t \mathbf{u}^\pm_\varepsilon \in L^2(0,T; \boldsymbol{\mathfrak{H}}^{\pm, *}_\varepsilon) \quad \text{and} \quad 
		\mathbf{u}^f_\varepsilon \in L^2(0,T; \boldsymbol{\mathfrak{H}}^f_\varepsilon), \quad \partial_t\mathbf{u}^f_\varepsilon \in L^2(0,T; \boldsymbol{\mathfrak{H}}^{f, *}_\varepsilon),
		$$
		and the following identities and conditions hold.
		\begin{itemize}
			\item For almost every \(t\in(0,T)\) and for all $\mathbf{v}^\pm \in \boldsymbol{\mathfrak{H}}^{\pm}_\varepsilon$ there holds
			\begin{multline}\label{identity-1}
				\langle\partial_t \mathbf{u}^\pm_\varepsilon, \mathbf{v}^\pm\rangle^\pm_\varepsilon 
				\, + \,  \int_{\Omega^\pm_\varepsilon} \mathbb{D}^\pm\nabla_x\mathbf{u}^\pm :  \nabla_x \mathbf{v}^\pm \, dx 
				\\
				= \int_{\Omega^\pm_\varepsilon} \left(\mathbf{F}^\pm(\mathbf{u}^\pm_\varepsilon, x, t) + \boldsymbol{\mathfrak{F}}^\pm\right) \cdot \mathbf{v}^\pm\, dx  \, - 
				\int_{S^\pm_\varepsilon} \mathbf{\Upsilon}^\pm\big(\mathbf{u}^\pm_\varepsilon,\mathbf{u}^f_\varepsilon, \tfrac{x}{\varepsilon}, x_1, t\big) \cdot \mathbf{v}^\pm\, dl_{x}.
			\end{multline}
			\item For almost every \(t\in(0,T)\) and for all
			\(\mathbf{v}^f\in H^1(\Omega^f_\varepsilon;\mathbb{R}^{\mathcal M})\) with
			\(\mathbf{v}^f|_{\Gamma_{0,\varepsilon}\cup\Gamma_{\ell,\varepsilon}}=\mathbf{0}\) there holds 
			\begin{multline}\label{identity-f}
				\langle\partial_t \mathbf{u}^f_\varepsilon, \mathbf{v}^f\rangle^f_\varepsilon 
				\, + \,  \varepsilon \int_{\Omega^f_\varepsilon}\mathbb{D}_\varepsilon(x)\nabla_x\mathbf{u}^f_\varepsilon :\nabla_x\mathbf{v}^f \, dx = - 
				\varepsilon^\beta \int_{G_\varepsilon} \mathbf{\Psi}\big(\mathbf{u}^f_\varepsilon, \tfrac{x}{\varepsilon}, x_1, t\big) \cdot \mathbf{v}^f\, dl_{x}
				\\
				- 	\varepsilon^\alpha \int_{S^+_\varepsilon} \mathbf{\Phi}^+\big(\mathbf{u}^+_\varepsilon,\mathbf{u}^f_\varepsilon, \tfrac{x}{\varepsilon}, x_1, t\big) \cdot \mathbf{v}^f \, dl_{x}
				- 	\varepsilon^\alpha \int_{S^-_\varepsilon} \mathbf{\Phi}^-\big(\mathbf{u}^-_\varepsilon,\mathbf{u}^f_\varepsilon, \tfrac{x}{\varepsilon}, x_1, t\big) \cdot \mathbf{v}^f \, dl_{x}.
			\end{multline}
			\item The trace condition on \(\Gamma_{\ell,\varepsilon}\) and the initial condition hold: 
			$$
			\mathbf{u}^f_\varepsilon|_{\Gamma_{\ell,\varepsilon}} = \bigl(q_1^\ell(t), \ldots, q_{\mathcal{M}}^\ell(t)\bigr), \qquad
			\mathbf{u}_\varepsilon|_{t=0}=\mathbf{0}.
			$$ 
		\end{itemize}
	\end{definition}
	Here,  $\mathbb{D}^\pm := \mathrm{diag}(D^\pm_1,\ldots,D^\pm_\mathcal{M})$ is a diagonal matrix,
	$\mathbb{D}_\varepsilon(x) :=\mathrm{diag}\bigl(\mathbb{D}_1(\tfrac{x}{\varepsilon}),\dots,\mathbb{D}_{\mathcal{M}}(\tfrac{x}{\varepsilon})\bigr)
	$
	is a block diagonal matrix, and the corresponding Frobenius-type contractions:
	$$
	\mathbb{D}^\pm\nabla_x\mathbf{u} : \nabla_x\mathbf{v} := \sum_{k=1}^{\mathcal{M}} D^\pm_k\nabla_x u_k \cdot \nabla_x v_k ,
	\qquad
	\mathbb{D}_\varepsilon(x)\nabla_x\mathbf{u}:\nabla_x\mathbf{v} := \sum_{k=1}^{\mathcal{M}} \bigl(\mathbb{D}_k(\tfrac{x}{\varepsilon})\nabla_x u_k\bigr) \cdot \nabla_x v_k . 	
	$$
	\begin{remark}
		It is known that if 
		$\mathbf{u}^\pm_\varepsilon \in L^2(0,T; \boldsymbol{\mathfrak{H}}^{\pm}_\varepsilon)$ and $\partial_t \mathbf{u}^\pm_\varepsilon \in L^2(0,T; \boldsymbol{\mathfrak{H}}^{\pm, *}_\varepsilon)$, then the function $\mathbf{u}^\pm_\varepsilon \in C([0,T]; L^2(\Omega^{\pm}_\varepsilon)^\mathcal{M}).$ Therefore,
		the equality $\mathbf{u}^\pm_\varepsilon|_{t=0}=\mathbf{0}$ is meaningful. The same statement holds for $\mathbf{u}^f_\varepsilon.$
	\end{remark}
	
	The existence and uniqueness of a weak solution to problem \eqref{probl} for each fixed $\varepsilon>0$, under the above Lipschitz assumptions on the nonlinear terms, is standard; one may obtain it, for example, by a Galerkin approximation (cf.\ \cite[Sect.~3]{Mel-Roh_Non-Diff-2024}) or by Sch\"afer's fixed point theorem (cf.\ \cite[Prop.~1]{Gahn-Radu-2018}).
	
	\medskip
	
	Our goal is to  study the asymptotic behavior of the solution $\mathbf{u}_\varepsilon = \big(\mathbf{u}^+_\varepsilon, \mathbf{u}^-_\varepsilon, \mathbf{u}^f_\varepsilon \big)$ as
	$\varepsilon \to 0,$ corresponding to the regime where the thin perforated layer $\Omega^f_\varepsilon$ collapses to the one-dimensional interval
	$\mathcal{I}^f := \{x\colon x_1\in (0, \ell),\ x_2 =0\},$ and
	the bulk domains $\Omega^+_\varepsilon$ and $\Omega^-_\varepsilon$ are transformed into
	$$
	\Omega^+ := (0, \ell)\times(0, \, \mathfrak{h}^+) \quad \text{and} \quad \Omega^- := (0, \ell)\times(-\mathfrak{h}^-, \, 0).
	$$
	Specifically, we  will  
	\begin{itemize}
		\item
		derive an effective limit model $(\varepsilon = 0)$ that 
		captures the impact of microscale features such as variable thickness, perforations and oscillating boundaries of the fracture,
		as well as chemical reactions and convection-dominated transport within it, on macroscopic behavior of diffusion-reaction-convection transport of dissolved substances;
		\item
		establish the well-posedness of the homogenized problem;
		\item
		construct the asymptotic approximation for the solution to  problem \eqref{probl} and prove the corresponding asymptotic estimates.
	\end{itemize}
	
	\begin{remark}\label{remark-2-6}
		Problem \eqref{probl} depends on three parameters: the small geometric parameter \(\varepsilon\), the intensity parameter \(\alpha\) appearing in the boundary conditions on the oscillating interfaces \(S^\pm_\varepsilon\), and the intensity parameter \(\beta\) appearing in the boundary conditions on the perforated boundaries \(G_\varepsilon\). In general (see \cite{Mel-Roh_AsAn-2024,Mel-Roh_JMAA-2024}), the construction of an asymptotic approximation requires adapting the ansatz and the asymptotic scaling to these parameters and to the detailed geometry of the problem. In this work we focus on the regime \(\alpha=\beta=1\), since this represents a critical scenario in which nonlinear boundary interactions act at the same asymptotic order as the bulk and fracture transport; consequently they contribute to the leading terms of the expansion and directly affect the homogenized model. 
		Other regimes (e.g. \(\alpha,\beta>1\) or \(\alpha,\beta<1\)) give rise to different asymptotic hierarchies and require adapted ansätze; they are not considered here and will be addressed in subsequent work.
	\end{remark}
	
	\begin{remark}\label{rem-2-4}
		In all statements in the paper,  $\varepsilon= \frac{\ell}{N}$ is a discrete parameter $(N\in \Bbb N).$  Therefore, $\varepsilon \to 0$
		means that $N \to +\infty.$
	\end{remark}

	
	
	\section{Formal asymptotic analysis}\label{Sect-3}
	
	\subsection{Analysis in the thin perforated domain}\label{Par-4-1}
	
	In  $\Omega^f_\varepsilon,$ we propose the following  asymptotic ansatz for each component $u_{k, \varepsilon}^f$ of the 
	solution $\mathbf{u}_\varepsilon = \big(\mathbf{u}^\pm_\varepsilon, \mathbf{u}^f_\varepsilon \big)$ to problem~\eqref{probl}:
	\begin{equation}\label{regul}
		\mathcal{U}^{(k)}_\varepsilon(x,t) :=  w_{0,k}(x_1,t) + \varepsilon\, N^{(k)}_1\big(\tfrac{x}{\varepsilon}\big) \, \partial_{x_1}w_{0,k}(x_1,t)
		+ \varepsilon^2 N^{(k)}_{2}\big(\tfrac{x}{\varepsilon}\big) \, \partial^2_{x_1^2}w_{0,k}(x_1,t),
	\end{equation}
	where $\partial_{x_1} := \frac{\partial}{\partial x_1},$ $\partial^2_{x_1} := \frac{\partial^2}{\partial x^2_1},$ and coefficients $N^{(k)}_1(\xi)$ and $N^{(k)}_{2}(\xi)$ are $1$-periodic in $\xi_1.$
	To keep the notation simple, we will drop the index $k$ for all coefficients both in \eqref{regul} and throughout this paragraph.
	
	Substituting $\mathcal{U}_\varepsilon$ in the corresponding differential equation of  problem \eqref{probl} and considering the incompressibility of the vector field $\overrightarrow{V_\varepsilon},$ we obtain
	\begin{multline}\label{eq1}
		\partial_t \mathcal{U}_\varepsilon -  \varepsilon\, \nabla_x \cdot\big( \mathbb{D}(\tfrac{x}{\varepsilon}) \nabla_x \mathcal{U}_\varepsilon \big) +
		\overrightarrow{V_\varepsilon}(x) \cdot \nabla_x \mathcal{U}_\varepsilon =
		\partial_t w_0 + \varepsilon\, N_1\, \partial^2_{t x_1}w_0
		\\
		- \Big(L_{\xi\xi}(N_1(\xi)) + \sum_{i=1}^{2} \partial_{\xi_i} d_{1 i}(\xi) - \overrightarrow{V}(\xi) \cdot \nabla_{\xi}\big(\xi_1 + N_1(\xi)\big)\Big)\Big|_{\xi=\frac{x}{\varepsilon}} \, \partial_{x_1} w_0
		\\
		- \varepsilon \bigg(\Big( L_{\xi\xi}(N_{2}(\xi)) + \sum_{i=1}^{2} \big(d_{1 i}\, \partial_{\xi_i}N_1  +
		\partial_{\xi_i}(d_{1 i} \, N_1)\big) + d_{11}(\xi) - \overrightarrow{V} \cdot \nabla_{\xi}N_{2}\Big)\Big|_{\xi=\frac{x}{\varepsilon}}\,  \partial^2_{x^2_1} w_0
		+ \mathcal{O}(\varepsilon^2),
	\end{multline}
	where $\partial_{t}w = \frac{\partial w}{\partial t},$ $\partial_{\xi_i}N = \frac{\partial N}{\partial \xi_i},$ $L_{\xi\xi}(N) = \sum_{i, j =1}^{2} \partial_{\xi_i}\big(d_{i j}(\xi) \partial_{\xi_j} N\big).$ Further, for convenience, we usually omit arguments in function notations (e.g., $d_{i j} = d_{i j}(\xi),$ $N = N(\xi),$ $w = w(x_1,t))$ unless this leads to misunderstandings. The symbol $\mathcal{O}(\varepsilon^2)$ in \eqref{eq1} denotes  the sum of terms of order $\varepsilon^2,$ namely
	$\varepsilon^2 \big(d_{11} N_{1} \, \partial^3_{x^3_1}w_0  + N_2\, \partial^3_{t x^2_1}w_0 + \varepsilon\, d_{11} N_{2} \, \partial^4_{x^4_1}w_0\big).$
	
	We require that the sum of the terms of order $\varepsilon^0$ in \eqref{eq1} vanishes. 
	To achieve this, we first eliminate the dependence on the microvariables $\xi = \frac{x}{\varepsilon}$ by imposing the conditions that
	\begin{gather}\label{eq2}
		L_{\xi\xi}(N_1(\xi)) + \sum_{i=1}^{2} \partial_{\xi_i} d_{1 i}(\xi) - \overrightarrow{V}(\xi) \cdot \nabla_{\xi}N_1(\xi) - v_1(\xi) =  \widehat{v}_1,
		\\ \label{eq3}
		L_{\xi\xi}(N_{2}(\xi)) + \sum_{i=1}^{2} \big(d_{1 i}\, \partial_{\xi_i}N_1  +
		\partial_{\xi_i}(d_{1 i} \, N_1)\big) + d_{11}(\xi) - \overrightarrow{V} \cdot \nabla_{\xi}N_{2} = \widehat{d}_{11} ,
	\end{gather}
	where $\widehat{v}_1$ and $\widehat{d}_{11}$ are constants  determined below.
	
	Substituting \eqref{regul} in the boundary conditions on $S_\varepsilon^\pm$ and $G_\varepsilon$ and using Taylor's formula for functions $\Phi^\pm$ and $\Psi,$ we obtain the following relations:
	\begin{equation}\label{eq5}
		\mathcal{B}_\xi(N_1(\xi)) + \sum_{i=1}^{2} d_{1 i}(\xi)\, \nu_i(\xi) = - \frac{\Phi^\pm(\mathbf{u}^\pm_0, \mathbf{w}^f_0, x_1, \xi, t)}{\partial_{x_1} w_0(x_1,t)} \quad \text{on} \ \ S^\pm,
	\end{equation}
	\begin{equation}\label{eq6}
		\mathcal{B}_\xi(N_1(\xi)) + \sum_{i=1}^{2} d_{1 i}(\xi)\, \nu_i(\xi) = - \frac{\Psi(\mathbf{w}^f_0, x_1, \xi, t)}{\partial_{x_1} w_0(x_1,t)} \quad \text{on} \ \ \partial T_0,
	\end{equation}
	\begin{equation}\label{eq7}
		\mathcal{B}_\xi(N_{2}(\xi)) + \sum_{i=1}^{2} d_{1 i}(\xi)\, N_1(\xi)\, \nu_i(\xi) = 0 \quad \text{on} \ \ S^\pm \cup \partial T_0,
	\end{equation}
	where $\mathcal{B}_\xi(N) := \sum_{i, j =1}^{2} d_{i j}(\xi)\, \partial_{\xi_j}N(\xi)\, \nu_i(\xi)$ is the conormal derivative, $\nu_i$ is the $i$-th component of the outward unit normal $\boldsymbol{\nu}_\xi$ to the corresponding parts of the boundary of $Y_0,$  and  the vector-functions $\mathbf{u}^\pm_0 = \big(u^\pm_{0,1},\ldots,u^\pm_{0, \mathcal{M}}\big)$ and
	$\mathbf{w}^f_0 = \big(w_{0,1},\ldots,w_{0, \mathcal{M}}\big)$ represent the first terms of the asymptotics of  the solution $\mathbf{u}_\varepsilon = \big(\mathbf{u}^\pm_\varepsilon, \mathbf{u}^f_\varepsilon \big).$ 
	
	\begin{remark}\label{rem-4-1}
		Equalities \eqref{eq5} and \eqref{eq6} are formal relations, as the right-hand sides depend on the variables $x_1$ and $t,$ while the left-hand sides do not.  These equalities will then be adjusted accordingly.
	\end{remark}
	
	Relations \eqref{eq2}, \eqref{eq5} and \eqref{eq6} define the Neumann boundary-value problem for the coefficient $N_1,$ while \eqref{eq3} and \eqref{eq6} define it for the coefficient$N_{2}.$  To find the solvability conditions for them, we consider a model cell problem.
	
	\subsubsection{Model cell problem in $Y_0$}
	
	In the periodicity cell $Y_0,$ we examine the following boundary-value problem:
	\begin{equation}\label{cell-problem}
		\left\{
		\begin{array}{ll}
			L_{\xi\xi}(N(\xi))  - \overrightarrow{V}(\xi) \cdot \nabla_{\xi}N(\xi) = F_0(\xi) + \sum_{i=1}^{2}\partial_{\xi_i} F_i(\xi), &
			\xi\in Y_0,
			\\[3pt]
			\mathcal{B}_{\xi}(N(\xi)) =  \Phi^\pm(\xi) + \sum_{i=1}^{2} F_i(\xi)\, \nu_i(\xi),
			&  \xi \in S^{\pm},
			\\[3pt]
			\mathcal{B}_{\xi}(N(\xi)) =  \Psi(\xi) + \sum_{i=1}^{2} F_i(\xi)\, \nu_i(\xi),
			&  \xi\in \partial T_0,
			\\[3pt]
			\partial^n_{\xi_1}N(\xi)\big|_{\xi_1=0} = \partial^n_{\xi_1}N(\xi)\big|_{\xi_1=1} & \text{on} \ \Gamma, \ \ n\in\{0, 1\},
			\\[3pt]
			\langle N\rangle_{Y_0}= 0, & 
		\end{array}
		\right.
	\end{equation}
	where
	$  \{F_0,\,  F_1,\,  F_2 \} \subset L^2(\omega_0), \ \Phi \in L^2(\partial T_0),$
	$ \Phi^{+}_0 \in L^2(S^{+}),$ $ \Phi^{-}_0 \in L^2(S^{-}).$
	
	\begin{definition}
		A function $N\in H^1_\sharp(Y_0):=\{v\in H^1(Y_0)\colon \   v \ \text{ is 1-periodic in } \, \xi_1 \}$ is called a weak solution to problem \eqref{cell-problem} if for any function $\phi \in H^1_\sharp(\omega_0)$ the following identity holds:
		\begin{equation*}
			\int\limits_{Y_0}\Big(\sum_{i, j=1}^{2} d_{ij}\, \partial_{\xi_i}N \, \partial_{\xi_j}\phi \, + \,   \phi \, \overrightarrow{V}\cdot\nabla_{\xi}N\Big) d\xi =
			-\int\limits_{Y_0} F_0\, \phi \, d\xi + \sum_{i=1}^{2} \int\limits_{Y_0} F_i \, \partial_{\xi_i}\phi \, d\xi + \int\limits_{S^{\pm}} \Phi^{\pm} \, \phi \, dl_{\xi} +  \int\limits_{\partial T_0} \Psi \, \phi \, dl_{\xi}.
		\end{equation*}
	\end{definition}
	\begin{remark}
		Hereinafter the symbol  $\int_{S^{\pm}} \Phi^{\pm}\, dl_{\xi}$ in equations means the sum of two integrals
		$\int_{S^{+}} \Phi^{+}\, dl_{\xi} + \int_{S^{-}} \Phi^{-}\, dl_{\xi}.$
	\end{remark}
	
	\begin{proposition}\label{Prop-4-1}
		Problem \eqref{cell-problem} has a unique weak solution if and only if
		\begin{equation}\label{cell_solution_exists}
			\int_{Y_0} F_0(\xi) d\xi =  \int_{S^{\pm}} \Phi^{\pm}_0(\xi)  \, dl_\xi + \int_{\partial T_0} \Psi(\xi)  \, dl_\xi.
		\end{equation}
	\end{proposition}
	\begin{proof}
		Applying the Fredholm alternative to problem \eqref{cell-problem}, the solvability condition is given by the relation
		\begin{equation}\label{cell_solution_exists1}
			\int_{Y_0} F_0(\xi) \, \rho(\xi) \, d\xi = \sum_{i=1}^{2} \int\limits_{Y_0} F_i(\xi)\,\partial_{\xi_i}\rho(\xi)\, d\xi + \int_{S^{\pm}} \Phi^{\pm}_0(\xi)  \, \rho(\xi) \, dl_\xi + \int_{\partial T_0} \Psi(\xi)  \, \rho(\xi) \, dl_\xi.
		\end{equation}
		where,  $\rho \in H^1_\sharp(Y_0)$ represents any solution to the corresponding homogeneous adjoint problem
		\begin{equation}\label{adjoint-cell-problem2}
			\left\{
			\begin{array}{ll}
				L_{\xi\xi}(\rho(\xi))  + \nabla_{\xi}\big(\rho(\xi)\, \overrightarrow{V}(\xi)\big) = 0, &
				\xi\in Y_0,
				\\[3pt]
				\mathcal{B}_{\xi}(\rho(\xi))  + \rho(\xi)\, \overrightarrow{V}(\xi) \cdot \boldsymbol{\nu}_\xi = 0,
				&  \xi \in S^{\pm} \cup  \partial T_0,
				\\[3pt]
				\partial^n_{\xi_1}\rho(\xi)\big|_{\xi_1=0} = \partial^n_{\xi_1}\rho(\xi)\big|_{\xi_1=1} & \text{on} \ \Gamma, \ \ n\in\{0, 1\}.
			\end{array}
			\right.
		\end{equation}
		Based on the assumptions for the vector-function $\overrightarrow{V}$ (see $\bf{A2}$), this  is equivalent to solving the problem
		\begin{equation}\label{adjoint-cell-problem1}
			\left\{
			\begin{array}{ll}
				L_{\xi\xi}(\rho(\xi))  + \overrightarrow{V}(\xi) \cdot \nabla_{\xi}\rho(\xi) = 0, &
				\xi\in Y_0,
				\\[3pt]
				\mathcal{B}_{\xi}(\rho(\xi)) = 0,
				&  \xi \in S^{\pm} \cup  \partial T_0,
				\\[3pt]
				\partial^n_{\xi_1}\rho(\xi)\big|_{\xi_1=0} = \partial^n_{\xi_1}\rho(\xi)\big|_{\xi_1=1} & \text{on} \ \Gamma, \ \ n\in\{0, 1\}.
			\end{array}
			\right.
		\end{equation}
		It is easy to see that every solution to problem \eqref{adjoint-cell-problem1} must be a constant function. Therefore, the solvability condition \eqref{cell_solution_exists1} simplifies to \eqref{cell_solution_exists}, where $\rho =1.$
	\end{proof}
	
	By writing out the solvability conditions for the problems involving the coefficients $N_1$ and $N_2$, we obtain the following relations:
	\begin{equation*}
		\widehat{v}_1 = - \langle v_1 \rangle_{Y_0} - \frac{\widehat{\Phi}^+}{\partial_{x_1}w_0}  - \frac{\widehat{\Phi}^-}{\partial_{x_1}w_0} - \frac{\widehat{\Psi}}{\partial_{x_1}w_0},
	\qquad 	\widehat{d}_{11} = \Big\langle d_{11} + \sum_{i=1}^{2} d_{1 i}\, \partial_{\xi_i}N_1\Big\rangle_{Y_0},
	\end{equation*}
	where
	\begin{equation}\label{eq10}
		\widehat{\Phi}^\pm(\mathbf{u}^\pm_0, \mathbf{w}^f_0, x_1, t) := \frac{1}{\upharpoonleft\!\! Y_0 \!\!\upharpoonright_2} \int_{S^\pm} \Phi^\pm(\mathbf{u}^\pm_0, \mathbf{w}^f_0, x_1, \xi, t)\, dl_\xi,
	\end{equation}
	\begin{equation}\label{eq11}
		\widehat{\Psi}(\mathbf{w}^f_0, x_1, t) := \frac{1}{\upharpoonleft\!\! Y_0 \!\!\upharpoonright_2} \int_{\partial T_0} \Psi(\mathbf{w}^f_0, x_1, \xi, t)\, dl_\xi.
	\end{equation}
	
	Assuming that the coefficient $N_1$ has been determined, the right-hand side of \eqref{eq1} can now be rewritten as follows:
	\begin{equation}\label{eq12}
		\partial_t w_0 +  \langle v_1 \rangle_{Y_0} \, \partial_{x_1} w_0 +  \widehat{\Phi}^+
		+ \widehat{\Phi}^- + \widehat{\Psi} + \varepsilon\, N_1\, \partial^2_{t x_1}w_0 
		- \varepsilon\, \widehat{d}_{11} \, \partial^2_{x^2_1} w_0 + \mathcal{O}(\varepsilon^2).
	\end{equation}
	Therefore, if the component $w_{0,k}$ of the vector-function $\mathbf{w}^f_0$ satisfies the equation
	\begin{equation}\label{eq13}
		\partial_t w_{0,k} 
		+ \langle v_1 \rangle_{Y_0} \, \partial_{x_1} w_{0,k} +  \widehat{\Phi}^+_k(\mathbf{u}^+_0, \mathbf{w}^f_0, x_1, t)
		+ \widehat{\Phi}^-_k(\mathbf{u}^-_0, \mathbf{w}^f_0, x_1, t) + \widehat{\Psi}_k(\mathbf{w}^f_0, x_1, t) = 0 \quad \text{in} \ \mathcal{I}^f\times (0,T),
	\end{equation}
	then ansatz \eqref{regul} leaves a small residue of order $\mathcal{O}(\varepsilon)$ in \eqref{eq1}. 
	
	\begin{proposition}\label{Prop-4-2}
		Coefficient $\hat{\mathrm{v}} := \langle v_1 \rangle_{Y_0}$ is positive.
	\end{proposition}
	\begin{proof}
		Multiplying the differential equation of problem \eqref{potential0} by  $\xi_1,$  integrating over $Y_0,$ and then integrating  by part, we derive  
		$$
		\int_{Y_0} \partial_{\xi_1}p(\xi)\,d\xi = \int_{-1}^{1} \mathrm{v}_0(\xi_2)\, d\xi_2.
		$$
		Considering assumption $\mathbf{A2},$ it follows from the previous equality that
		\begin{equation}\label{avarage conv flow}
			\langle v_1 \rangle_{Y_0} = \frac{1}{\upharpoonleft\!\! Y_0 \!\!\upharpoonright_2}\int_{-1}^{1} \mathrm{v}_0(\xi_2)\, d\xi_2 > 0.
		\end{equation}
	\end{proof}
	
	Thus, the first terms $\{w_{0,k}\}_{k=1}^\mathcal{M}$ in the asymptotics of $\mathbf{u}^f_\varepsilon$ must be solutions to the mixed problem for a first-order semi-linear hyperbolic system  
	\begin{equation}\label{hyperbolic-system}
		\left\{
		\begin{array}{ll}
			\partial_t w_{0,k}(x_1,t)  + \hat{\mathrm{v}} \, \partial_{x_1} w_{0,k}(x_1,t) = \widehat{F}_k(\mathbf{u}^+_0, \mathbf{u}^-_0, \mathbf{w}^f_0, x_1, t), &
			(x_1,t) \in \mathcal{I}^f\times (0,T),
			\\[3pt]
			w_{0,k}(0,t)= 0, \ \ t\in [0,T], \qquad w_{0,k}(x_1,0)= 0, \ \ x_1\in [0, \ell],
			& k\in \{1,\ldots,\mathcal{M}\}.
		\end{array}
		\right.
	\end{equation}
	Here,
	\begin{equation}\label{hat F}
		\widehat{F}_k(\mathbf{u}^+_0, \mathbf{u}^-_0, \mathbf{w}^f_0, x_1, t) := -  \widehat{\Phi}^+_k(\mathbf{u}^+_0, \mathbf{w}^f_0, x_1, t)
		- \widehat{\Phi}^-_k(\mathbf{u}^-_0, \mathbf{w}^f_0, x_1, t) - \widehat{\Psi}_k(\mathbf{w}^f_0, x_1, t),
	\end{equation}
	where the functions $\widehat{\Phi}^\pm_k$ and  $\widehat{\Psi}_k$ are defined by formulas \eqref{eq10} and \eqref{eq11}, respectively,  
	and the boundary conditions and initial conditions are taken from the original problem \eqref{probl}.

	Now we adjust the boundary conditions in the problems to determine the coefficients $\{N_1^{(k)}\}_{k=1}^\mathcal{M}$ (see Remark~\ref{rem-4-1}). For each $k,$ the coefficient $N_1^{(k)} \in H^1_\sharp(Y_0)$ is the unique solution to the problem
	\begin{equation}\label{problem-N-1}
		\left\{
		\begin{array}{ll}
			L^k_{\xi\xi}(N_1^{(k)})  - \overrightarrow{V} \cdot \nabla_{\xi}N_1^{(k)} = v_1(\xi) - \hat{\mathrm{v}}
			- \sum_{i=1}^{2} \partial_{\xi_i} d^k_{1 i}(\xi), & \xi \in Y_0, 
			\\[3pt]
			\mathcal{B}^k_{\xi}(N_1^{(k)}) =  - \sum_{i=1}^{2}d^k_{1 i}(\xi)\, \nu_i(\xi),
			&  \xi \in S^{\pm} \cup \partial T_0,
			\\[3pt]
			\langle N_1^{(k)}\rangle_{Y_0}= 0, & 
		\end{array}
		\right.
	\end{equation}
	where $L^k_{\xi\xi}(N) := \sum_{i, j =1}^{2} \partial_{\xi_i}\big(d^k_{i j}\, \partial_{\xi_j}N\big),$
	$\mathcal{B}^k_\xi(N) := \sum_{i, j =1}^{2} d_{i j}\, \partial_{\xi_j}N\, \nu_i.$ Based on Proposition~\ref{Prop-4-1}, such a solution does indeed exist.

	The coefficient $N^{(k)}_2\in H^1_\sharp(Y_0)$ is uniquely determined as the solution to the problem 
	\begin{equation}\label{problem-N-2}
		\left\{
		\begin{array}{ll}
			L^k_{\xi\xi}(N_2^{(k)}) - \overrightarrow{V} \cdot \nabla_{\xi}N^{(k)}_{2} = \Big(\widehat{d}^{\,k}_{11} - d^k_{11}
			- \sum_{i=1}^{2}d^k_{1 i}\,\partial_{\xi_i}N^{(k)}_1\Big) - \sum_{i=1}^{2}\partial_{\xi_i}(d^k_{1 i}\, N^{(k)}_1)\big), &
			\xi\in Y_0,
			\\[3pt]
			\mathcal{B}^k_{\xi}(N_2^{(k)}) =  - \sum_{i=1}^{2}d^k_{1 i}\,N_1^{(k)}\,\nu_i,
			&  \xi \in S^{\pm} \cup \partial T_0,
			\\[3pt]
			\langle N_2^{(k)}\rangle_{Y_0}= 0, & 
		\end{array}
		\right.
	\end{equation}
	where 
	\begin{equation}\label{hom-coeff}
		\widehat{d}^{\,k}_{11} = \big\langle d^k_{11} + \sum_{i=1}^{2}d^k_{1 i}\,\partial_{\xi_i}N^{(k)}_1 \big\rangle_{Y_0}, \quad k\in \{1,\ldots,\mathcal{M}\}.
	\end{equation}
	
	\subsection{Analysis in the bulk domains $\Omega^+_\varepsilon$ and $\Omega^-_\varepsilon$}\label{Par-4-2} 
	
	It is clear that it is sufficient to do this in one of these regions.   
	To achieve this, we will follow the approach used in \cite{Mel-1999,Mel-Naz-1997}, where asymptotic approximations for solutions to boundary value problems in domains with highly oscillating boundaries were constructed.
	
	We are looking for two types of expansions: the  outer one is
	\begin{equation}\label{exp-1}
		\mathbf{u}^+_\varepsilon(x,t) = \mathbf{u}^+_0(x,t) + \varepsilon \mathbf{u}^+_1(x,t) + \ldots  \quad \text{in} \ \ \Omega^+_\varepsilon
	\end{equation}
	outside a small neighborhood of the oscillating boundary $S^+_\varepsilon,$ and  the inner one 
	\begin{equation}\label{outer-1}
		\mathbf{u}^+_\varepsilon(x,t) = \mathbf{u}^+_0(x_1,0,t) + \varepsilon \Big(Z^+_1\big(\tfrac{x_1}{\varepsilon}, 
		\tfrac{x_2- \varepsilon}{\varepsilon}\big)\, \partial_{x_1}\mathbf{u}^+_0(x_1,0,t) + Z^+_2\big(\tfrac{x_1}{\varepsilon}, 
		\tfrac{x_2- \varepsilon}{\varepsilon}\big)\, \partial_{x_2}\mathbf{u}^+_0(x_1,0,t)\Big) +  \ldots
	\end{equation}
	is valid near  $S^+_\varepsilon.$ Here coefficients $Z^+_1(\xi_1, \xi_2)$ and $Z^+_2(\xi_1, \xi_2),$ where $\xi_1 = \frac{x_1}{\varepsilon}$ and $\xi_2 = \frac{x_2 - \varepsilon}{\varepsilon},$ are $1$-periodic in $\xi_1$ and will be determined below.
	
	Using Taylor's formula for the vector-function $\mathbf{F}^+=(F^+_1,\ldots,F^+_\mathcal{M}),$ it is easy to write down relations for the leading term of  \eqref{exp-1}:
	\begin{equation}\label{relations+}
		\left\{\begin{array}{rcll}
			\partial_t \mathbf{u}^+_0 -   \mathbb{D}^+ \Delta \mathbf{u}^+_0
			& = & \mathbf{F}^+(\mathbf{u}^+_0, x,t) + \boldsymbol{\mathfrak{F}}^+(x,t) & (x,t) \in \Omega^+ \times (0, T),
			\\[2mm]
			\mathbf{u}^+_0  & = & 0 & \text{on} \ \ \big(\partial\Omega^+  \setminus \mathcal{I}^f \big) \times (0, T),
		\end{array}\right.
	\end{equation}
	where $\mathbb{D}^+$ is the diagonal matrix $\text{diag}(D^+_1,\ldots,D^+_\mathcal{M}),$ $\boldsymbol{\mathfrak{F}}^\pm= (\mathfrak{f}_1^\pm,\ldots,\mathfrak{f}_\mathcal{M}^\pm).$
	
	To derive  boundary relations for the components of  $\mathbf{u}^+_0 = (u^+_{0,1},\ldots,u^+_{0,\mathcal{M}})$ on the interval $\mathcal{I}^f,$  we substitute the inner ansatz \eqref{outer-1} into the corresponding differential equations of problem \eqref{probl} in $\Omega^+_\varepsilon$ and into the boundary conditions on $S^+_\varepsilon.$ We  perform this substitution for each component of the vector expansion \eqref{outer-1}, omitting the "$+$" and $k$ indices.   As we will see, the coefficients $Z^+_1$ and $Z^+_2$ are the same for all components. 
		Collecting terms at the same powers of $\varepsilon,$ we obtain at $\varepsilon^{-1}$ in  $\Omega^+_\varepsilon$ the following sum:
	\begin{equation}\label{sum-Z}
		- D \Big(\Delta_\xi Z^+_1(\xi) \, \partial_{x_1}{u}_0(x_1,0,t) + \Delta_\xi Z^+_2(\xi) \, \partial_{x_2}{u}_0(x_1,0,t)\Big),
	\end{equation}
	where $\Delta_\xi$ is the Laplace operator with respect to variables $\xi_1 = \frac{x_1}{\varepsilon}$ and $\xi_2 = \frac{x_2 - \varepsilon}{\varepsilon}.$ Since the right-hand side of  the corresponding differential equation is of order $\mathcal{O}(1),$ the sum \eqref{sum-Z} must be equal to zero, i.e., $\Delta_\xi Z^+_1(\xi) = \Delta_\xi Z^+_2(\xi) = 0$ in the domain $\Pi^+_0 :=\{\xi=(\xi_1,\xi_2)\colon \ \xi\in (0,1), \ \xi_2 > h_+(\xi_1) - 1\}$ (see Fig.~\ref{fig-3}). 
	\begin{figure}[htbp]
		
		\vspace{-0.3cm}
		\begin{center}
			\includegraphics[width=5cm]{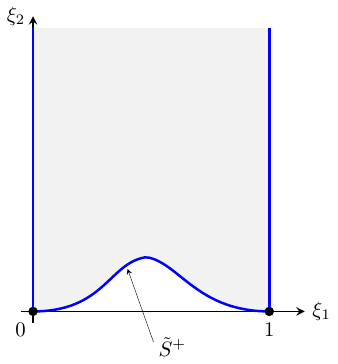}
		\end{center}
		
			\vspace{-0.3cm}
		\caption{Sketch of the domain $\Pi^+_0$ }\label{fig-3}
	\end{figure}

	Substituting \eqref{outer-1} and \eqref{regul} into the boundary condition $D \, \nabla u_{\varepsilon} \cdot \boldsymbol{\nu}_\varepsilon =  \Upsilon\big(\mathbf{u}^+_\varepsilon, \mathbf{u}^f_\varepsilon, \tfrac{x}{\varepsilon}, x_1, t \big)$ on $S^+_\varepsilon,$ 
	we get the following relation at order $\varepsilon^0:$
	\begin{multline}\label{bound-cond-1}
		- D \, \Big(\partial_{\boldsymbol{\nu}_\xi}Z^+_1(\xi) + \nu_1(\xi)\Big)\Big|_{\xi \in \tilde{S}^+} \, \partial_{x_1}{u}_0(x_1,0,t)
		- D \, \partial_{\boldsymbol{\nu}_\xi}Z^+_2(\xi)\Big|_{\xi \in \tilde{S}^+}\,  \partial_{x_2}{u}_0(x_1,0,t)
		\\
		=  \widetilde{\Upsilon}\big(\mathbf{u}^+_0(x_1,0,t), \mathbf{w}^f_0(x_1,t), x_1, t \big), \quad (x_1, t) \in \mathcal{I}^f \times (0, T),
	\end{multline}
	where $\tilde{S}^+:= \{\xi\colon\ \xi_2=h_+(\xi_1) -1, \ \xi_1 \in (0,1) \}$ is the lower part of $\partial\Pi^+_0$ (see Fig.~\ref{fig-3}),
	$\partial_{\boldsymbol{\nu}_\xi}$ is the derivative along the  outward unit normal $\boldsymbol{\nu}_\xi= (\nu_1(\xi), \nu_2(\xi))$ to the boundary of the domain $\Pi^+_0$ (hence the minus sign on the left-hand side), and
	\begin{equation}\label{eq-14}
		\widetilde{\Upsilon}\big(\mathbf{u}^+_0(x_1,0,t), \mathbf{w}^f_0(x_1,t), x_1, t \big) := 
		\frac{1}{\upharpoonleft\!\! {S}^+\!\!\upharpoonright_1}
		\int_{S^+}  \Upsilon\big(\mathbf{u}^+_0(x_1,0,t), \mathbf{w}^f_0(x_1,t), \xi, x_1, t \big) \,  dl_\xi .
	\end{equation}
	The right-hand side of \eqref{bound-cond-1} follows from \eqref{Lip-ineq} and the inequality
	\begin{multline}\label{eq-15}
		\bigg|\int\limits_{S^+_\varepsilon}\Big(\Upsilon\big(\mathbf{u}^+_0(x_1,0,t), \mathbf{w}^f_0(x_1,t), \tfrac{x}{\varepsilon}, x_1, t \big)
		\, \phi \, dl_x    
		\\
		- \upharpoonleft\!\! {S}^+\!\!\upharpoonright_1 \int_{0}^{\ell} \widetilde{\Upsilon}\big(\mathbf{u}^+_0(x_1,0,t), \mathbf{w}^f_0(x_1,t), x_1, t \big) \, \phi(x_1, 0) \, dx_1\bigg| \le c_1  \, \varepsilon^{\frac{1}{2}} \|\phi\|_{H^1(\Omega^+)} 
	\end{multline}
	for any function $\phi \in H^1(\Omega^+),$ which can be proved in the same way as, for example, the analogous inequality in Theorem~4.1 of \cite[\S 4, Chapt.~3]{Ole-Ios-Sha-1991}.
	
	To eliminate the dependence on the microvariables $\xi$ in \eqref{bound-cond-1}, we impose the following conditions: 
	\begin{equation}\label{eq-15+}
		\Big(\partial_{\boldsymbol{\nu}_\xi}Z^+_1(\xi) + \nu_1(\xi)\Big)\Big|_{\xi \in \tilde{S}^+} = b_1 \quad \text{and} \quad 
		\partial_{\boldsymbol{\nu}_\xi}Z^+_2(\xi)\Big|_{\xi \in \tilde{S}^+} = b_2,
	\end{equation}
	where the constants $b_1$ and $b_2$ are determined below. 
	 
	Thus, we have obtained relations that form boundary-value problems for $Z^+_1$ and $Z^+_2$ in  domain $\Pi^+_0.$

	\subsubsection{Model cell problem in $\Pi^+_0$}
	
	First, we consider following model boundary-value problem:
	\begin{equation}\label{cell-problem-Pi}
		\left\{
		\begin{array}{ll}
			\Delta_{\xi}Z(\xi)  = f(\xi) , & \xi\in\Pi^+_0,
			\\[3pt]
			\partial_{\boldsymbol{\nu}_\xi}Z(\xi) =  g(\xi),
			&  \xi \in \tilde{S}^+,
			\\[3pt]
			\partial^n_{\xi_1}Z(\xi)\big|_{\xi_1=0} = \partial^n_{\xi_1}Z(\xi)\big|_{\xi_1=1}, & \xi_2\in (0, +\infty), \ \ n\in\{0, 1\}.
		\end{array}
		\right.
	\end{equation}
	For this problem, we introduce an energy space of functions for which the Dirichlet integral is finite and constant functions belong to it.
	Let  $\widehat{C}^{\infty}_0\big(\overline{\Pi^+_0}\big)$ be a space of infinitely   differentiable functions in $\overline{\Pi^+_0}$ that satisfy the   periodical conditions of problem \eqref{cell-problem-Pi}  and  vanish at sufficiently large values of   $\xi_2$,
	i.e.,     
	$$
	\forall u\in
	\widehat{C}^{\infty}_0\big(\overline{\Pi^+_0}\big)\quad\exists\, R>0\quad\forall \xi \in \overline{\Pi^+_0} \quad \xi_2 \ge R\colon \  u(\xi)=0 .
	$$
	By $ \mathcal{H},$ we denote the completion of the space $\widehat{C}^{\infty}_0\big(\overline{\Pi^+_0}\big)$ with respect to  the norm
	$$
	\|u\|_{\mathcal{H}}=\left(\|\nabla_{\xi}u \|^2_{L^2(\Pi^+_0)}+\|\rho   u\|^2_{L^2(\Pi^+_0)}\right)^{1/2}\, ,
	$$
	where  $\rho(\xi_2)=(1+ \xi_2)^{-1} ,\ \xi_2 \in (0, +\infty).$  
	
	A  function $Z$ is called a weak solution to  problem \eqref{cell-problem-Pi}
	if, for all functions $u\in\mathcal{H},$  it holds that 
	\begin{equation}\label{eq-16}
		\int_{\Pi^+_0}\nabla_{\xi}Z\cdot\nabla_{\xi}u\,d\xi =    \int_{\tilde{S}^+} g\, u\, dl_\xi - \int_{\Pi^+_0}f \, u\,d\xi  .  
	\end{equation}
	
	In the same way as in \cite[\S 4.2]{Mel-1999}, we prove the following lemma.
	\begin{lemma}\label{lemma-4-1}
		Let $\rho^{-1} f\in L^2(\Pi^+), \ g \in L^2(\tilde{S}^+),$ and 
		\begin{equation}\label{eq-17}
			\int_{\tilde{S}^+} g\, dl_\xi = \int_{\Pi^+_0}f \, d\xi .
		\end{equation}
		
		Then there exists a unique solution $Z \in \mathcal{H}$ to problem \eqref{cell-problem-Pi}.
		In addition,
		\begin{itemize}
			\item  if the functions $f$ and $g$ are even or odd in $\xi_1$ with respect to $\frac12,$ then the solution inherits the corresponding symmetry;
			\item 
			if $\exp(\delta\, \xi_2)\, f \in L^2(\Pi^+)$ with some $\delta >0,$ then the solution has the differential asymptotics:
			\begin{equation}\label{eq-18}
				Z(\xi) = c_0 + \mathcal{O}\big(\exp(-\delta_0 \, \xi_2) \big) \quad \text{as} \ \ \xi_2 \to +\infty \ \ (\delta_0 > 0).  
			\end{equation}
		\end{itemize}
	\end{lemma}
	
	The coefficient $Z^+_1$ is defined as a solution to problem \eqref{cell-problem-Pi} with $f=0$ and $g(\xi) = - \nu_1(\xi).$ 
	Since the component $\nu_1$ of the normal is  odd  in $\xi_1$ with respect to $\frac{1}{2}$ -- a consequence of  the evenness of the function $h_+$ --  the solvability condition \eqref{eq-17} is satisfied. By Lemma~\ref{lemma-4-1}, it follows that  $Z^+_1$ is also odd  in $\xi_1$ with respect to $\frac{1}{2}.$ Consequently, the constants $b_1$ and $c_0$ in \eqref{eq-15+} and \eqref{eq-18}, respectively, vanish. 
	
	To derive a condition for $Z^+_2$ at infinity,  we match the outer expansion \eqref{exp-1} with the inner one \eqref{outer-1}. The asymptotics of $u_0$ as $x_2 \to 0$ is as follows: $u_0 = u_0|_{x_2=0} + \partial_{x_2}u_0|_{x_2=0}\, x_2 + \mathcal{O}(x_2^2),$ which, in terms of the scaled variable $ \xi_2,$ $\xi_2 = \frac{x_2 - \varepsilon}{\varepsilon},$ becomes
	$$
	u_0(x_1,0,t) + \varepsilon\, \partial_{x_2}u_0(x_1,0,t)\, \xi_2 +  \mathcal{O}(\varepsilon) +  \mathcal{O}(\varepsilon^2 \xi_2^2).
	$$
	Comparing this expansion with the asymptotic behavior of  \eqref{outer-1} as $\xi_2 \to +\infty,$ and using  the exponential decay of $Z^+_1,$  we deduce the following condition at infinity for the coefficient $Z^+_2:$ \ $Z^+_2(\xi) \sim \xi_2$ as $\xi_2 \to +\infty.$ Therefore,  $Z^+_2$ must satisfy  the following problem: 
	\begin{equation}\label{cell-problem-Z-2}
		\left\{
		\begin{array}{rcll}
			\Delta_{\xi}Z^+_2(\xi) & =& 0 , & \xi\in\Pi^+_0,
			\\[3pt]
			\partial_{\boldsymbol{\nu}_\xi}Z^+_2(\xi) &=&  b_2,
			&  \xi \in \tilde{S}^+,
			\\[3pt]
			Z^+_2(\xi) &\sim& \xi_2
			& \text{as} \  \xi_2 \to +\infty,
			\\[3pt]
			\partial^n_{\xi_1}Z^+_2(\xi)\big|_{\xi_1=0} &=& \partial^n_{\xi_1}Z^+_2(\xi)\big|_{\xi_1=1}, & \xi_2\in (0, +\infty), \ \ n\in\{0, 1\}.
		\end{array}
		\right.
	\end{equation}
	\begin{corollary}
		If 
		\begin{equation}\label{eq-19}
			b_2 = - \frac{1}{\upharpoonleft\!\! {S}^+\!\!\upharpoonright_1},
		\end{equation}
		then there is a unique solution to problem \eqref{cell-problem-Z-2}, which is even in $\xi_1$ with respect to $\frac12$ and has  the following  asymptotic behaviour at infinity:
		\begin{equation}\label{eq-20}
			Z^+_2(\xi) = \xi_2 + c_0 +  \mathcal{O}\big(\exp(-\delta_0 \, \xi_2) \big) \quad \text{as} \ \ \xi_2 \to +\infty \ \ (\delta_0 > 0).
		\end{equation}
	\end{corollary}
	\begin{proof}
		The solution $Z^+_2$ is sought as the sum $\xi_2\, \chi_+(\xi_2) + \widetilde{Z}_2,$ where $\widetilde{Z}_2 \in \mathcal{H},$ and  $\widetilde{Z}_2$  solves  problem  \eqref{cell-problem-Pi} with the right-hand sides $f(\xi_2)= -  \chi'_+(\xi_2) - \big(\xi_2\, \chi_+(\xi_2)\big)'$ and $g= b_2.$ Here, $\chi_{+}$ is a smooth cut-off function defined as 
		$$
		\chi_+(\xi_2) = \left\{
		\begin{array}{ll}
			1, & \hbox{if} \ \ \xi_2 \ge 3 \max_{[0,1]}h_+, 
			\\
			0, & \hbox{if }  \  \xi_2 \le 2 \max_{[0,1]}h_+.
		\end{array}
		\right.
		$$
		
		The solvability condition for this problem reads: $b_2 \upharpoonleft\!\! \tilde{S}^+\!\!\upharpoonright_1 = -1,$ which  is equivalent to
		\eqref{eq-19} due to the equality $ \upharpoonleft\!\! \tilde{S}^+\!\!\upharpoonright_1 = \upharpoonleft\!\! {S}^+\!\!\upharpoonright_1.$  
		Next it remains to apply Lemma~\ref{lemma-4-1}.
	\end{proof}
	
	Thus, we have uniquely determined the coefficients $Z^+_1$ and $Z^+_2,$ and relation \eqref{bound-cond-1} now takes the form
	\begin{equation}\label{bound-cond-hom}
		D \,  \partial_{x_2}{u}_0(x_1,0,t)
		=  \ \upharpoonleft\!\! {S}^+\!\!\upharpoonright_1 \widetilde{\Upsilon}\big(\mathbf{u}^+_0(x_1,0,t), \mathbf{w}^f_0(x_1,t), x_1, t \big), \quad (x_1, t) \in \mathcal{I}^f \times (0, T).
	\end{equation}
	\begin{remark}\label{rem-4-4}
		Because of the $1$-periodicity  in $\xi_1$ and the oddness of $Z^+_1$ with respect to $\frac{1}{2}$,  we have $Z^+_1\big|_{\xi_1=0} = 0.$
		Therefore, since $\mathbf{u}^+_0=0$ on $\partial\Omega^+  \setminus \mathcal{I}^f,$ the sum 
		\begin{equation}\label{eq-21}
			\mathbf{u}^+_0(x_1,0,t) + \varepsilon \Big(Z^+_1\big(\tfrac{x_1}{\varepsilon}, 
			\tfrac{x_2- \varepsilon}{\varepsilon}\big)\, \partial_{x_1}\mathbf{u}^+_0(x_1,0,t) + Z^+_2\big(\tfrac{x_1}{\varepsilon}, 
			\tfrac{x_2- \varepsilon}{\varepsilon}\big)\, \partial_{x_2}\mathbf{u}^+_0(x_1,0,t)\Big)
		\end{equation}
		vanishes on the vertical parts of the boundary of $\Omega^+,$ i.e., at $x_1 = 0$ or $x_1=\ell.$
	\end{remark}
	\begin{remark}
		If the function $h_+$ is not assumed to be even, one cannot conclude that the sum \eqref{eq-21} vanishes along the vertical segments of the boundary (see Remark~\ref{rem-4-4}). In this situation, the justification of the constructed asymptotics becomes significantly more complex: it requires the introduction 
		of additional boundary and corner layer corrections. Moreover, the condition on $\mathcal{I}^f$ takes the following form:
		$$
		D \int_{S^+}\nu_1(\xi)\, dl_\xi\,  \partial_{x_1}{u}_0(x_1,0,t) + 
		D \,  \partial_{x_2}{u}_0(x_1,0,t)
		=  \ \upharpoonleft\!\! {S}^+\!\!\upharpoonright_1 \widetilde{\Upsilon}\big(\mathbf{u}^+_0(x_1,0,t), \mathbf{w}^f_0(x_1,t), x_1, t \big).
		$$ 
	\end{remark}

	\section{The homogenized  problem}\label{Sect-4} 
	
	The systems \eqref{hyperbolic-system}, \eqref{relations+} with conditions \eqref{bound-cond-hom} for each components $u^+_{0,k}, \ k \in \{1,\ldots,M\},$ and the corresponding system in $\Omega^-$ form
	the \textit{homogenized problem}
	\begin{equation}\label{homo-problem}
		\left\{\begin{array}{rcll}
			\partial_t \mathbf{u}^\pm_0 -   \mathbb{D}^\pm \Delta \mathbf{u}^\pm_0
			& = & \mathbf{F}^\pm(\mathbf{u}^\pm_0, x, t) + \boldsymbol{\mathfrak{F}}^\pm(x,t) & (x,t) \in \Omega^{\pm, T} ,
			\\[2mm]
			\mathbf{u}^\pm_0  & = & \mathbf{0}, & \text{on} \ \ \partial\Omega^{\pm, T}  \setminus \mathcal{I}^{f,T},
			\\[2mm]
			\pm\, \mathbb{D}^\pm \partial_{x_2} \mathbf{u}^\pm_0&=& 
			\upharpoonleft\!\! {S}^\pm\!\!\upharpoonright_1\widetilde{\boldsymbol{\Upsilon}}^\pm\big(\mathbf{u}^\pm_0(x_1,0,t), \mathbf{w}^f_0(x_1,t), x_1, t \big), & (x_1, t) \in \mathcal{I}^{f,T},
			\\[2mm]
			\mathbf{u}^\pm_0 \big|_{t=0} & = & \mathbf{0}, & \text{in} \ \ \Omega^\pm  ,
			\\[2mm]
			\partial_t \mathbf{w}^f_0  + \hat{\mathrm{v}} \, \partial_{x_1} \mathbf{w}^f_0 & =& \widehat{ \mathbf {F}}(\mathbf{u}^+_0(x_1,0,t), \mathbf{u}^-_0(x_1,0,t), \mathbf{w}^f_0(x_1,t), x_1, t), &
			(x_1,t) \in \mathcal{I}^{f, T},
			\\[3pt]
			\mathbf{w}^f_0(0,t)= \mathbf{0}, \ \ t\in [0,T], & & \mathbf{w}^f_0(x_1,0)= \mathbf{0}, \ \ x_1\in [0, \ell],
		\end{array}\right.
	\end{equation}
	for problem \eqref{probl}. Here,  $\mathbb{D}^\pm = \text{diag}(D^\pm_1,\ldots,D^\pm_\mathcal{M}),$
	$\mathbf{F}^\pm=(F^\pm_1,\ldots,F^\pm_\mathcal{M}),$ $\widetilde{\boldsymbol{\Upsilon}}^\pm=(\widetilde{\Upsilon}^\pm_1,\ldots,\widetilde{\Upsilon}^\pm_\mathcal{M}),$ $\widehat{\mathbf{F}}=(\widehat{F}_1,\ldots,\widehat{F}_\mathcal{M}),$ where $\widehat{F}_k$ is defined by  formula \eqref{hat F},  and 
	\begin{equation}\label{eq-22}
		\widetilde{\Upsilon}^\pm_k\big(\mathbf{u}^\pm, \mathbf{w}, x_1, t \big) :=  \frac{1}{\upharpoonleft\!\! {S}^\pm\!\!\upharpoonright_1}
		\int_{S^\pm}  \Upsilon^\pm_k\big(\mathbf{u}^\pm, \mathbf{w}, \xi, x_1, t \big) \, dl_\xi  .
	\end{equation}
	
	Denote by $\boldsymbol{\mathfrak{H}}^{\pm, *} $ the dual space to the Sobolev vector space 
	$\boldsymbol{\mathfrak{H}}^{\pm} := \left\{ \mathbf{u}\in H^1(\Omega^\pm; \mathbb{R}^\mathcal{M})\colon \ \mathbf{u}\big|_{\partial\Omega^\pm \setminus \mathcal{I}^f } =0 \right\}.$
	
	\begin{definition}\label{Def-}
		We say that a vector-function $(\mathbf{u}^+, \mathbf{w}^f, \mathbf{u}^-),$ where
		$$
		\mathbf{u}^\pm \in L^2(0,T; \boldsymbol{\mathfrak{H}}^{\pm}), \quad \partial_t \mathbf{u}^\pm \in L^2(0,T; \boldsymbol{\mathfrak{H}}^{\pm, *}) \quad \text{and} \quad \mathbf{w}^f \in C([0, \ell] \times [0,T]; \mathbb{R}^\mathcal{M}),
		$$
		is a weak solution to the homogenized problem \eqref{homo-problem} provided
		\begin{multline}\label{eq23}
			\int_{\Omega^\pm} \partial_t \mathbf{u}^\pm \cdot \mathbf{v} \, dx + \mathbb{D}^\pm\int_{\Omega^\pm} \nabla_x  \mathbf{u}^\pm :  \nabla_x \mathbf{v} \, dx 
			\\
			= \int_{\Omega^\pm} \left(\mathbf{F}^\pm(\mathbf{u}^\pm, x, t) + \boldsymbol{\mathfrak{F}}^\pm\right) \cdot \mathbf{v}\, dx  \, - 
			\upharpoonleft\!\! {S}^\pm\!\!\upharpoonright_1 \int_{\mathcal{I}^f} \widetilde{\boldsymbol{\Upsilon}}^\pm\big(\mathbf{u}^\pm(x_1,0,t), \mathbf{w}^f(x_1,t), x_1, t \big) \cdot \mathbf{v}(x_1,0)\, dx_1
		\end{multline}
		for any function $\mathbf{v} \in \boldsymbol{\mathfrak{H}}^{\pm}$ and a.e. $t \in (0, T),$ and $\mathbf{u}^\pm|_{t=0}=\mathbf{0},$ and
		\begin{equation}\label{eq24}
			\mathbf{w}^f(x_1,t) = \mathbb{S}\{\mathbf{u}^+, \mathbf{u}^-,\mathbf{w}^f\}(x_1,t), \quad (x_1, t) \in \mathcal{I}^{f, T},
		\end{equation}
		where the components of the  vector-function $\mathbb{S}\{\mathbf{u}^+, \mathbf{u}^-,\mathbf{w}^f\}(x_1,t)= \big(\mathsf{S}_1(x_1,t),\ldots, \mathsf{S}_\mathcal{M}(x_1,t)\big)$ are defined as follows: 
		\begin{equation}\label{eq25}
			\mathsf{S}_k(x_1,t)=
			\left\{
			\begin{array}{ll}
				\displaystyle  \int\limits_{0}^{t} \widehat{F}_k \big(\mathbf{u}^+(y_1, 0, \tau), \mathbf{u}^-(y_1, 0, \tau),\mathbf{w}^f(y_1, \tau),  y_1, \tau\big)\big|_{y_1 = \hat{\mathrm{v}}(\tau -t) +x_1} 
				\, d\tau  , &  (x_1,t) \in \mathcal{D}_1 ,
				\\
				\displaystyle 
				\int\limits_{t - \frac{x_1}{\hat{\mathrm{v}}}}^{t}  \widehat{F}_k\big(\mathbf{u}^+(y_1, 0, 
				\tau), \mathbf{u}^-(y_1, 0, \tau), \mathbf{w}^f(y_1, \tau),  y_1, \tau\big)\big|_{y_1 = \hat{\mathrm{v}}(\tau -t) +x_1}  \, d\tau   , & (x_1,t) \in  \mathcal{D}_2,
			\end{array}
			\right.
		\end{equation}
		and $\mathcal{D}_1 :=\{(x_1, t)\colon x_1 \in (0, \ell), \ \ t \in (0, \frac{x_1}{\hat{\mathrm{v}}})\}
		,$
		$\mathcal{D}_2 :=  \mathcal{I}^{f, T} \setminus \overline{\mathcal{D}}_1.$ 
	\end{definition}
	
	\begin{remark}
		The integral equations \eqref{eq24} are obtained by the integration of hyperbolic differential equations in \eqref{homo-problem} along the corresponding characteristics taking into account the initial and boundary conditions for them. 
	\end{remark}
	
	\subsection{Recursive inequality}
	To justify the convergence of the approximations in the theorem concerning the existence of a solution to the homogenized problem, we require two lemmas.
	\begin{lemma}\label{lemma-4-2}
		Let $\{f_n(t)\}_{n\in \mathbb{N}_0}$ be a sequence of continuous and non-negative functions on $[0,T]$
		such that
		$$
		f_{n+1}(t) \le q\,f_n(t) + C \int_{0}^{t} f_n(\tau)\, d\tau,
		$$
		where $ 0 < q < 1,\ C > 0.$ Then the series $\sum_{n=0}^{\infty} f_n(t) $
		converges uniformly for $t\in[0,T].$ 
	\end{lemma}
	\begin{proof} 
		Let’s define a new majorant sequence:
		\begin{equation}\label{recur-1}
			g_0(t) := M_0 = \max_{t \in [0,T]} f_0(t), \quad
			g_{n+1}(t) = q\,  g_n(t) + C \int_0^t g_n(\tau)\, d\tau.
		\end{equation}
		This sequence satisfies the same recurrence as $\{f_n\},$ and it is easy to verify that
		$f_n(t) \le g_n(t)$ for $t\in [0,T].$
		
		Now let us find a solution of  the linear recurrence \eqref{recur-1}. To do this we define the bivariate generating function
		\[
		G(x,t)\;=\;\sum_{n=0}^{\infty}g_{n}(t)\,x^{n}.
		\]
		At \(t=0\), the integral vanishes and \(g_{n}(0)=q\,g_{n-1}(0)\) with \(g_{0}(0)=M_{0}\).  By induction
		$g_{n}(0)=M_{0}\,q^{n},$
		so
		\[
		G(x,0)
		=\sum_{n=0}^{\infty}M_{0}\,q^{n}\,x^{n}
		=\frac{M_{0}}{1-qx} \quad \text{for} \ \ x \in [0, \tfrac{1}{q}).
		\]
		
		Multiplying recurrence \eqref{recur-1} by $x^{n}$ and summing over $n\ge0,$ we get
		functional equation for $G$:
		\begin{equation}\label{recur-2}
			\frac{G(x,t)-M_{0}}{x}
			=q\,G(x,t)
			+ C\int_{0}^{t}G(x,\tau)\,d\tau
			\;\Longrightarrow\;
			(1-qx)\,G(x,t)
			= M_{0} + C\,x\int_{0}^{t}G(x,\tau)\,d\tau.
		\end{equation}
		Differentiating \eqref{recur-2} with respect to $t,$ we obtain 
		\[
		\frac{\partial G}{\partial t}
		= \frac{C\,x}{1-qx}\,G,
		\]
		whence 
		\[
		G(x,t)
		=G(x,0)\,
		\exp\Bigl(\frac{C\,x\,t}{1-qx}\Bigr)
		=\frac{M_{0}}{1-qx}
		\exp\Bigl(\frac{C\,x\,t}{1-qx}\Bigr).
		\]
		
		Using obvious expansions
		\[
		\frac{1}{1-qx}
		=\sum_{m=0}^{\infty}q^{m}x^{m},
		\quad
		\exp\Bigl(\frac{C\,x\,t}{1-qx}\Bigr)
		=\sum_{k=0}^{\infty}\frac{1}{k!}
		\Bigl(\frac{C\,t\,x}{1-qx}\Bigr)^{k}
		=\sum_{k=0}^{\infty}\frac{(C\,t)^{k}}{k!}\,x^{k}
		\sum_{j=0}^{\infty}\binom{k+j-1}{j}(q\,x)^{j},
		\]
		we deduce
		\[
		G(x,t)
		=M_{0}
		\sum_{m,k,j\ge0}
		q^{m}\,
		\frac{(C\,t)^{k}}{k!}\,
		\binom{k+j-1}{j}\,
		q^{\,j}\,
		x^{\,m+k+j}.
		\]
		The coefficient of $x^{n}$ is
		\[
		g_{n}(t)
		= M_{0}
		\sum_{\substack{m,k,j\ge0\\m+k+j=n}}
		q^{\,m+j}
		\frac{(C\,t)^{k}}{k!}
		\binom{k+j-1}{j}.
		\]
		Collapsing  the triple sum, we have
		\[
		g_{n}(t)
		= M_{0}
		\sum_{k=0}^{n}
		\binom{n}{k}
		q^{\,n-k}
		\frac{(C\,t)^{k}}{k!}.
		\]
		
		Taking into account that $q \in (0, 1),$ we verify with the  Weierstrass $M$-test that the series $\sum_{n=0}^{\infty} g_n(t) $
		converges uniformly for $t\in[0,T].$
	\end{proof}
	
	\begin{lemma}\label{lemma-4-3}
		Let $\{f_n(t)\}_{n\in \mathbb{N}_0}$ be a sequence of continuous and non-negative functions on $[0,T]$
		satisfying the recursive inequality
		\[
		f_{n+1}(t) \;\le\; C \int_0^t f_n(\tau)\, d\tau + \alpha_n,
		\]
		where $C > 0$ and the series of non-negative numbers $\sum_{n=0}^{\infty} \alpha_n$ converges.
		
		Then the series $\sum_{n=0}^{\infty} f_n(t) $
		converges uniformly for $t\in[0,T].$ 
	\end{lemma}
	\begin{proof} We define a majorant sequence $\{g_n(t)\}$ recursively as follows:
		\[
		g_0(t) := M_0 := \max_{t \in [0, T]} f_0(t), \quad
		g_{n+1}(t) := C \int_0^t g_n(\tau)\, d\tau.
		\]
		Since \( f_{n+1}(t) \le C \int_0^t f_n(\tau)\,d\tau + \alpha_n \le g_{n+1}(t) + \alpha_n \),
		we obtain by induction:
		\[
		f_n(t) \;\le\; g_n(t) + \sum_{k=0}^{n-1} \alpha_k \cdot \frac{(C t)^{n-1-k}}{(n-1-k)!},
		\quad \text{for all } t \in [0, T].
		\]

		To evaluate \( \max_{t \in [0,T]} f_n(t) \), we note that 
		\[
		g_n(t) = M_0 \cdot \frac{(C t)^n}{n!},
		\quad \text{and so} \quad \max_{t \in [0,T]} g_n(t) = M_0 \cdot \frac{(C T)^n}{n!}.
		\]
		Thus,	
		\[
		\max_{t \in [0,T]} f_n(t)
		\;\le\; M_0 \cdot \frac{(C T)^n}{n!}
		+ \sum_{k=0}^{n-1} \alpha_k \cdot \frac{(C T)^{n-1-k}}{(n-1-k)!}.
		\]
		Summing over $n,$ we obtain:
		\begin{align*}
			\sum_{n=0}^{\infty} \max_{t \in [0,T]} f_n(t)
			&\le M_0 \sum_{n=0}^{\infty} \frac{(C T)^n}{n!}
			+ \sum_{n=0}^{\infty} \sum_{k=0}^{n-1} \alpha_k \cdot \frac{(C T)^{n-1-k}}{(n-1-k)!} \\
			&= M_0\, e^{C T} + \sum_{k=0}^{\infty} \alpha_k \sum_{m=0}^{\infty} \frac{(C T)^m}{m!}
			= M_0\, e^{C T} + e^{C T} \sum_{k=0}^{\infty} \alpha_k.
		\end{align*}
		
		Since \( \sum \alpha_k < \infty \), the series \( \sum_{n=0}^{\infty} \max_{t \in [0,T]} f_n(t) \) converges.
	\end{proof}

	\subsection{Well-posedness of the homogenized problem}\label{par-exist}
	\begin{theorem}\label{Th-4-1}
		There exists a unique weak solution to problem \eqref{homo-problem}. 
	\end{theorem}

	\begin{proof}
		The proof of existence will be carried out using the iteration method.
		
		{\bf 1.} We propose the following iteration scheme. Consider the sequences of vector-functions
		\begin{equation}\label{sequences}
			\{\mathbf{v}^+_m\}_{m\in \mathbb N_0}, \quad \{\mathbf{v}^-_m\}_{m\in \mathbb N_0}, \quad \{\mathbf{z}^f_m\}_{m\in \mathbb N_0}, 
		\end{equation}   
		where $\mathbf{z}^f_0 = \mathbf{0},$ $\mathbf{v}^\pm_0 \in C^\infty\big(\overline{\Omega^{\pm, T}}; \mathbb{R}^\mathcal{M}\big)$ and 
		$
		\mathbf{v}^\pm_0\big|_{\partial\Omega^{\pm, T}\setminus \mathcal{I}^{f,T}} = \mathbf{0}, \quad 
		\mathbf{v}^\pm_0\big|_{t=0} = \mathbf{0}.
		$
		The other functions are defined as follows: if $\mathbf{v}^+_{n-1}$ and $\mathbf{v}^-_{n-1}$ are known as functions from the H\"older spaces
		$H^{2+\gamma,1 + \frac{1}{\gamma}}\big(\overline{\Omega^{+,T}}; \mathbb{R}^\mathcal{M}\big)$  and $H^{2+\gamma,1 + \frac{1}{\gamma}}\big(\overline{\Omega^{-,T}}; \mathbb{R}^\mathcal{M}\big),$ and $\mathbf{z}^f_{n-1}$ is known as function from the space
		$C^2(\overline{\mathcal{I}^{f, T}}; \mathbb{R}^\mathcal{M}),$
		respectively, then first we define $\mathbf{v}^+_{n}$ and $\mathbf{v}^-_{n}$ as solutions to the linear problems
		\begin{equation}\label{recur-prob-1}
			\left\{\begin{array}{rcll}
				\partial_t \mathbf{v}^\pm_n -   \mathbb{D}^\pm \Delta \mathbf{v}^\pm_n
				& = & \mathbf{F}^\pm(\mathbf{v}^\pm_{n-1}, x, t) + \boldsymbol{\mathfrak{F}}^\pm(x,t) & (x,t) \in \Omega^{\pm, T} ,
				\\[2mm]
				\mathbf{v}^\pm_n  & = & \mathbf{0}, & \text{on} \ \ \partial\Omega^{\pm, T}  \setminus \mathcal{I}^{f,T},
				\\[2mm]
				\pm\, \mathbb{D}^\pm \partial_{x_2} \mathbf{v}^\pm_n &=& 
				\upharpoonleft\!\! {S}^\pm\!\!\upharpoonright_1\widetilde{\boldsymbol{\Upsilon}}^\pm\big(\mathbf{v}^\pm_{n-1}(x_1,0,t), \mathbf{z}^f_{n-1}(x_1,t), x_1, t \big), & (x_1, t) \in \mathcal{I}^{f,T},
				\\[2mm]
				\mathbf{v}^\pm_{n} \big|_{t=0} & = & \mathbf{0}, & \text{in} \ \ \Omega^\pm ,
			\end{array}\right.
		\end{equation}
		depending on sign "$+$" or "$-$",  
		and then the function $\mathbf{z}^f_{n}$ as a solution to the linear first order hyperbolic system
		\begin{equation}\label{recur-prob-2}
			\left\{\begin{array}{rcll}
				\partial_t \mathbf{z}^f_n  + \hat{\mathrm{v}} \, \partial_{x_1} \mathbf{z}^f_n & =& \widehat{ \mathbf {F}}(\mathbf{v}^+_n(x_1,0,t), \mathbf{v}^-_n(x_1,0,t), \mathbf{z}^f_{n-1}(x_1,t), x_1, t), &
				(x_1,t) \in \mathcal{I}^{f, T},
				\\[3pt]
				\mathbf{z}^f_n(0,t)= \mathbf{0}, \ \ t\in [0,T], & & \mathbf{z}^f_n(x_1,0)= \mathbf{0}, \ \ x_1\in [0, \ell].
			\end{array}\right.
		\end{equation}
		
		Due to assumptions \textbf{(A3)} for  the vector-functions $\mathbf{F}^{\pm},$ $\boldsymbol{\mathfrak{F}}^\pm$ and $\boldsymbol{\Upsilon}^\pm,$ all necessary matching conditions for problem \eqref{recur-prob-1} are satisfied and the solution $\mathbf{v}^\pm_n$ belongs to the same space as $\mathbf{v}^\pm_{n-1}$ (see \cite[Chapt. IV, \S 5]{Lad_Sol_Ura_1968}) and each its component ${v}^\pm_{k, n}$
		can be represented in the doubly–indexed Fourier series
		\begin{equation}\label{rep-1}
			{v}^\pm_{k, n}(x_{1},x_{2},t) = \sum_{m=1}^{\infty}\sum_{p=0}^{\infty}
			A^{(\pm, k)}_{m,p}(t)\,\phi_{m,p}(x_{1},x_{2}), \quad (x,t) \in \overline{\Omega^{\pm, T}},
		\end{equation}
		with corresponding Fourier coefficients $\{A^{(\pm, k)}_{m,p}\}.$ Here 
		$
		\phi_{m,p}(x_{1},x_{2}) 
		= \sin\bigl(\alpha_{m}\,x_{1}\bigr)\, \cos\bigl(\mu_{p}\,x_{2}\bigr)
		$
		are separable modes and $\alpha_{m} = \frac{m\pi}{\ell},$  
		\ 
		$
		\mu^{\pm}_{p} = \frac{(2p+1)\,\pi}{2\,\mathfrak{h}^{\pm}}.
		$

		Taking into account the assumptions regarding the smoothness of functions $\boldsymbol{\Phi}^{\pm}$ and $\boldsymbol{\Psi},$ as well as their compact supports, the matching conditions at the point $(0,0)$ for the hyperbolic system \eqref{recur-prob-2} are satisfied (for more detail, see \cite[Appendix B]{Mel-Roh_JMAA-2024}). Consequently, problem \eqref{recur-prob-2} admits a unique $C^2$-smooth solution, and its explicit representation by formula \eqref{eq25} is possible.
		
		\smallskip
		\noindent
		{\bf 2.} Here, we prove the convergence of the sequences 
		$\{\mathbf{v}^\pm_m(x_1, 0, t)\}$ and $\{\mathbf{z}^f_m(x_1,t)\}$
		in the space $C(\overline{\mathcal{I}^{f, T_0}}; \mathbb{R}^\mathcal{M})$ for sufficiently small $T_0.$ Let us introduce the following notation:
		\begin{equation}\label{not-1}
			Z_n(t) := \max_{k;\, x_1\in [0,\ell];\, \tau\in [0, t]} \big|z^f_{k, n+1}(x_1, \tau) - z^f_{k, n}(x_1, \tau)\big|,
		\end{equation}
		\begin{align}\label{not-2}
			V_n(t) := & \ \max_{k; \,x_1\in [0,\ell]; \, \tau\in [0, t]} \big|v^+_{k, n+1}(x_1, 0, \tau) - v^+_{k, n}(x_1, 0, \tau)\big|
			\notag
			\\
			& \ +
			\max_{k;\, x_1\in [0,\ell]; \, \tau\in [0, t]} \big|v^-_{k, n+1}(x_1, 0, \tau) - v^-_{k, n}(x_1, 0, \tau)\big|.
		\end{align}

		Using the representation formula \eqref{eq25} for solutions $\mathbf{z}^f_{n+1}$ and $\mathbf{z}^f_n$ and the Lipschitz continuity of $\widehat{ \mathbf{F}}$ (see Remark~\ref{rem-2-3}), we derive the inequality 
		\begin{equation}\label{Z-n}
			Z_n(t) \le C_1 \bigg(\int_0^t V_n(\tau) \, d\tau + \int_0^t Z_{n-1}(\tau)\, d\tau\bigg).
		\end{equation}
		
		Now we estimate $V_n$ using the Fourier representation \eqref{rep-1} for the solutions $\mathbf{v}^\pm_{n+1}$ and $\mathbf{v}^\pm_{n}.$ For each $k\in \{1,\ldots,\mathcal{M}\},$ 
		\begin{multline}\label{dif-v-1}
			\big|v^+_{k, n+1}(x_1, 0, \tau) - v^+_{k, n}(x_1, 0, \tau)\big|
			+
			\big|v^-_{k, n+1}(x_1, 0, \tau) - v^-_{k, n}(x_1, 0, \tau)\big|
			\\	
			\le\; \frac{4}{\ell \,\mathfrak{h}^+}
			\sum_{m=1}^{\infty}\sum_{p=0}^{\infty}
			\int_{0}^{\tau}e^{-D^+_k \lambda^+_{m,p}\, (\tau-s)}\bigl|G^+_{m}(s)\bigr|\,ds + \frac{4}{\ell \,\mathfrak{h}^-}\sum_{m=1}^{\infty}\sum_{p=0}^{\infty}
			\int_{0}^{\tau}e^{-D^-_k\lambda^-_{m,p} \, (\tau-s)}\bigl|G^-_{m}(s)\bigr|\,ds.
		\end{multline}
		Here $\lambda^{\pm}_{m,p}=\alpha_{m}^{2}+(\mu^{\pm}_{p})^{2}$ and $G^{\pm}_{m}(s)$ is the sine-Fourier coefficient for the corresponding Neumann data.
		
		Using Cauchy-Schwarz and Parseval inequalities, we deduce
		\begin{multline}\label{dif-v-2}
			\sum_{m=1}^{\infty}\sum_{p=0}^{\infty}
			\int_{0}^{\tau}e^{-D^\pm_k \lambda^\pm_{m,p}\, (\tau-s)}\bigl|G^\pm_{m}(s)\bigr|\,ds \le 
			\int_{0}^{\tau} \sum_{m=1}^{\infty} \bigl|G^\pm_{m}(s)\bigr|\, 
			e^{-D^\pm_k \alpha^2_m\, (\tau-s)}\, \sum_{p=0}^{\infty} e^{-D^\pm_k (\mu^\pm_{p})^2\, (\tau-s)}\, ds	
			\\
			\le 
			\int_{0}^{\tau} \Bigl(\sum_{m=1}^\infty G^\pm_m(s)^2\Bigr)^{1/2}
			\Bigl(\sum_{m=1}^\infty e^{-2 D^\pm_k\alpha_m^2 (\tau - s)}\Bigr)^{1/2}
			\int_0^\infty e^{- D^\pm_k (\pi/\mathfrak{h}^\pm)^2 y^2 (\tau- s)}\,dy \, ds
			\\
			=
			\int_{0}^{\tau} \Bigl(\sum_{m=1}^\infty G^\pm_m(s)^2\Bigr)^{1/2}
			\Biggl(\frac{\ell}{2 \sqrt{2 \pi D^\pm_k \, (\tau -s)}} \Biggr)^{1/2}
			\frac{\mathfrak{h}^\pm}{2 \sqrt{ \pi D^\pm_k \, (\tau -s)}}\, ds
			\\
			\le
			\sqrt{\ell} \upharpoonleft\!\! {S}^\pm\!\!\upharpoonright_1 \max_{x_1\in [0,\ell], \, s \in [0, \tau]} \big|\widetilde{\Upsilon}_k^\pm\big(\mathbf{v}^\pm_{n}(x_1,0,s), \mathbf{z}^f_{n}(x_1,s), x_1, s \big) -
			\widetilde{\Upsilon}_k^\pm\big(\mathbf{v}^\pm_{n-1}(x_1,0,s), \mathbf{z}^f_{n-1}(x_1,s), x_1, s \big)\big| 
			\\
			\times C_1 \, \int_{0}^{\tau} \frac{1}{(\tau -s)^{3/4}} \, ds 	
		\end{multline}
		
		Taking into account \eqref{Lip-ineq}, \eqref{dif-v-1} and \eqref{dif-v-2}, we get that
		\begin{equation}\label{V-n}
			V_n(t) \le C_2 \, t^{\frac{1}{4}} \Big(V_{n-1}(t) + Z_{n-1}(t)\Big),	
		\end{equation} 
		where the constant $C_2$ depends only on $\ell,$ $\mathfrak{h}^\pm,$ $\mathbb{D}^\pm$ and $\upharpoonleft\!\! {S}^\pm\!\!\upharpoonright_1.$
		
		It follows from \eqref{Z-n} and \eqref{V-n} that
		\begin{equation}\label{W-n}
			W_n(t) := Z_n(t) + V_n(t) \le C_2 \, t^{\frac{1}{4}} \, W_{n-1}(t) + C_3 \, T \, \int_{0}^{t} W_{n-1}(\tau)\, d\tau. 	
		\end{equation}
		
		We choose $T_0$ so small that $C_2\, (T_0)^{\frac{1}{4}} < 1.$ This allows us to apply Lemma~\ref{lemma-4-2} and 
		Cauchy criterion for uniform convergence of a sequence. Consequently, the sequences
		$\{\mathbf{v}^\pm_n(x_1, 0, t)\}_{n\in \mathbb N_0}$ and $\{\mathbf{z}^f_n(x_1,t)\}_{n\in \mathbb N_0}$ converge uniformly on $[0, \ell] \times [0, T_0]$ to continuous vector-functions $\mathbf{u}^\pm_0(x_1,t)$ and $\mathbf{w}^f_0(x_1, t),$ respectively.
		
		\smallskip
		\noindent
		{\bf 3.} Here we prove the convergence of $\{\mathbf{v}^\pm_n(x, t)\}$ 
		in the space 
		$$
		X_{T_0} := C([0, T_0]; L^2(\Omega^\pm; \mathbb{R}^\mathcal{M})) \cap L^2(0, T_0; H^1(\Omega^\pm; \mathbb{R}^\mathcal{M}))
		$$
		with the norm
		$$
		\|\mathbf{v}\| = \max_{t\in [0,T_0]} \|\mathbf{v}(t)\|_{L^2(\Omega^\pm; \mathbb{R}^\mathcal{M})} +
		\|\, |\nabla_x \mathbf{v}|\, \|_{L^2(\Omega^{\pm,T_0})}.  
		$$
		
		Using  the Lipschitz continuity of $\mathbf{F}^\pm$ and the inequality 
		$$
		\int_{\mathcal{I}^f} \phi^2(x_1,0)\, dx_1  \le 2 \Bigg(\delta \int_{\Omega^\pm}|\nabla_x\phi|^2 \, dx + \frac{1}{\delta} \int_{\Omega^\pm}\phi^2 \, dx \Bigg) \quad \forall\, \phi \in H^1(\Omega^\pm),
		$$
		with any positive $\delta,$  we derive from  \eqref{recur-prob-1} the following inequality for the difference
		\begin{equation}\label{deffer-V_n}
			\mathbf{V}^\pm_n(x,t):= \mathbf{v}^\pm_{n+1}(x,t) - \mathbf{v}^\pm_n(x,t) 
		\end{equation}
		for any $t \in (0, T]$: 
		\begin{multline}\label{app-1}
			\int_{\Omega^\pm} |\mathbf{V}^\pm_n(x,t)|^2\, dx + \int_{\Omega^{\pm, t}} |\nabla_x \mathbf{V}^\pm_n|^2 \, dx\, d\tau 
			\\
			\le C_1 \int_{\Omega^{\pm, t}} |\mathbf{V}^\pm_n(x,\tau)|^2 \, dx\, d\tau + C_2 \int_{\Omega^{\pm, t}} |\mathbf{V}^\pm_{n-1}(x,\tau)|^2 \, dx\, d\tau + C_3 \, W^2_{n-1}(t) \, t.  
		\end{multline}
		
		Gronwall’s lemma applied to the inequality
		$$
		\int_{\Omega^\pm} |\mathbf{V}^\pm_n(x,t)|^2\, dx \le C_1 \int_{0}^{t}\int_{\Omega^\pm} |\mathbf{V}^\pm_n|^2 \, dx\, d\tau 
		+  C_2 \int_{\Omega^{\pm, t}} |\mathbf{V}^\pm_{n-1}(x,\tau)|^2 \, dx\, d\tau + C_3 \, W^2_{n-1}(t) \, t
		$$
		gives the estimate
		\begin{equation}\label{Gron-1}
			\int_{\Omega^\pm} |\mathbf{V}^\pm_n(x,t)|^2\, dx \le \exp(C_1 T) \, C_4 \Big(\int_{0}^{t}\int_{\Omega^\pm} |\mathbf{V}^\pm_{n-1}(x,\tau)|^2 \, dx\, d\tau + W^2_{n-1}(T_0) \, T \Big) \quad \text{for any } \ \ t \in [0,T_0]. 
		\end{equation}
		
		Now, it remains to apply Lemma~\ref{lemma-4-3} to \eqref{Gron-1} with  
		$$
		f_n(\tau) = \int_{\Omega^\pm} |\mathbf{V}^\pm_{n-1}(x,\tau)|^2 \, dx \quad \text{and} \quad  \alpha_n = 
		\exp(C_1 T) \, C_4 \, T \, W^2_{n-1}(T_0).
		$$ 
		It should be noted that, based on the results of item 2, $\lim_{n\to 0}W_{n-1}(T_0)= 0.$ 
		Thus, the series $\sum_{n=0}^{\infty} f_n(t) $ converges uniformly for $t\in[0,T_0],$ and from \eqref{app-1} it follows that
		$$
		\sum_{n=0}^{\infty} \int_{\Omega^{\pm, T_0}} |\nabla_x \mathbf{V}^\pm_n|^2 \, dx\, d\tau < +\infty .
		$$

		Now using the Cauchy criterion for the convergence of a sequence and the fact that $\{\mathbf{v}^\pm_n(x_1, 0, t)\}_{n\in \mathbb N_0}$ uniformly on $[0, \ell] \times [0, T_0]$ converges to $\mathbf{u}^\pm_0(x_1, t),$ we conclude that 
		$\{\mathbf{v}^\pm_n(x, t)\}$ converges to a vector-function $\mathbf{u}^\pm_0$
		in the space $X_{T_0},$  whose trace on $\mathcal{I}^f$ coincides with
		the continuous function $\mathbf{u}^\pm_0(x_1, t).$ Obviously, that for a.e. $t\in [0,T_0]$ the function  $\mathbf{u}^\pm_0$ has zero-trace on $\partial\Omega^\pm \setminus \mathcal{I}^f.$ We can also regard that
		\begin{equation}\label{conv4}
			\mathbf{v}^\pm_n    \rightarrow  \mathbf{u}^\pm_0 \quad \text{almost everywhere in} \ \ \Omega_\varepsilon^{\pm, T_0}.
		\end{equation}

		\smallskip
		\noindent
		{\bf 4.} Using the convergence results obtained in the point 2, we can pass to the limit as $n\to \infty$ in the representation 
		formula $\mathbf{z}^f_{n+1}(x_1,t) = \mathbb{S}\{\mathbf{v}^+_{n+1}, \mathbf{v}^-_{n+1},\mathbf{z}^f_n\}(x_1,t)$ 
		and obtain \eqref{eq24} for $(x_1, t) \in \mathcal{I}^{f, T_0}.$
		
		Now we should pass to the limit in the integral identity 
		\begin{multline}\label{in-1}
			\int_{\Omega^{\pm}} \mathbf{v}^{\pm}_n(x, T_0) \cdot \partial_t \boldsymbol{\phi}(x, T_0) \, dx	- \int_{\Omega^{\pm, T_0}} \mathbf{v}^{\pm}_n \cdot \partial_t \boldsymbol{\phi} \, dxdt + \mathbb{D}^\pm \int_{\Omega^{\pm, T_0}} \nabla_x  \mathbf{v}^{\pm}_n :  \nabla_x\boldsymbol{\phi}\,dxdt 
			\\
			= \int_{\Omega^{\pm, T_0}} \mathbf{F}^\pm(\mathbf{v}^\pm_{n-1},x, t) \cdot \boldsymbol{\phi}\, dxdt  \, - 
			\upharpoonleft\!\! {S}^\pm\!\!\upharpoonright_1 \int_{\mathcal{I}^{f, T_0}} \widetilde{\boldsymbol{\Upsilon}}^\pm\big(\mathbf{v}^\pm_{n-1}(x_1,0,t), \mathbf{z}^{f}_{n-1}(x_1,t), x_1, t \big) \cdot \boldsymbol{\phi}(x_1,0,t)\, dx_1 dt
		\end{multline}
		for any function $\boldsymbol{\phi} \in H^1\big(0,T_0; H^1(\Omega^\pm; \mathbb{R}^\mathcal{M})\big),$ which corresponds to problem \eqref{recur-prob-1}.
		
		Based on the convergence results above and assumption $\mathbf{A3},$ it is easy to find the limits of all integrals in \eqref{in-1}.
		As a result, we get that 
		\begin{multline}\label{in-2}
			\int_{\Omega^{\pm}} \mathbf{u}^{\pm}_0(x, T_0) \cdot \partial_t \boldsymbol{\phi}(x, T_0) \, dx	- \int_{\Omega^{\pm, T_0}} \mathbf{u}^{\pm}_0 \cdot \partial_t \boldsymbol{\phi} \, dxdt + \mathbb{D}^\pm \int_{\Omega^{\pm, T_0}} \nabla_x  \mathbf{u}^{\pm}_0 :  \nabla_x\boldsymbol{\phi}\,dxdt 
			\\
			= \int_{\Omega^{\pm, T_0}} \mathbf{F}^\pm(\mathbf{u}^\pm_{0},x, t) \cdot \boldsymbol{\phi}\, dxdt  \, - 
			\upharpoonleft\!\! {S}^\pm\!\!\upharpoonright_1 \int_{\mathcal{I}^{f, T_0}} \widetilde{\boldsymbol{\Upsilon}}^\pm\big(\mathbf{u}^\pm_{0}(x_1,0,t), \mathbf{w}^{f}(x_1,t), x_1, t \big) \cdot \boldsymbol{\phi}(x_1,0,t)\, dx_1 dt .
		\end{multline}
		Due to the equivalent definitions of weak solutions (see, e.g. \cite[Chapt. III]{Showalter}), $\mathbf{u}^{\pm}_0$ satisfies identity \eqref{eq23} for  $t\in (0,T_0).$
		Thus, the vector-function $(\mathbf{u}^+_0, \mathbf{w}^f, \mathbf{u}^-_0)$ is a weak solution to problem \eqref{homo-problem} existing on the time interval $[0, T_0].$ 
		
		\smallskip
		\noindent
		{\bf 5.}
		Since $\mathbf{u}^\pm_0(t) \in H^1(\Omega^\pm; \mathbb{R}^\mathcal{M})$ with zero-trace on $\partial\Omega^\pm \setminus \mathcal{I}^f$ for almost every $t \in [0, T_0],$ we may, upon redefining $T_0$ if necessary, assume that $\mathbf{u}^\pm_0(T_0) \in H^1(\Omega^\pm; \mathbb{R}^\mathcal{M}).$ This allows us to repeat the previous arguments and extend the solution to the time interval $[T_0, 2 T_0].$ To do this,
		we consider the iteration sequences \eqref{sequences}, where we set  
		$\mathbf{z}^f_0 = \mathbf{w}^f(x_1, T_0),$ $\mathbf{v}^\pm_0 = \mathbf{u}^\pm_0(T_0),$
		and define the remaining terms as solution to systems \eqref{recur-prob-1} and \eqref{recur-prob-2}, subject to the initial conditions 
		$$
		\mathbf{v}^\pm_{n} \big|_{t=T_0}  =  \mathbf{u}^\pm_0(T_0) \ \  \text{in} \ \ \Omega^\pm, \qquad
		\mathbf{z}^f_n(x_1,T_0)= \mathbf{w}^f(x_1, T_0), \ \ x_1\in [0, \ell],
		$$
		respectively. 
		For the corresponding differences defined in \eqref{not-1}, \eqref{not-2} and \eqref{deffer-V_n}, 
		we derive the following estimates, analogous to those obtained previously:
		\begin{gather*}
			W_n(t) \le C_2 \, (t - T_0)^{\frac{1}{4}} \, W_{n-1}(t) + C_3 \, T \, \int_{T_0}^{t} W_{n-1}(\tau)\, d\tau,
			\\ 
			\int_{\Omega^\pm} |\mathbf{V}^\pm_n(x,t)|^2\, dx \le \exp(C_1 T) \, C_4 \Big(\int_{T_0}^{t}\int_{\Omega^\pm} |\mathbf{V}^\pm_{n-1}(x,\tau)|^2 \, dx\, d\tau + W^2_{n-1}(2 T_0) \, T \Big) 
		\end{gather*}
		for all $t \in [T_0, 2 T_0].$ These estimates imply the convergence of the iteration sequences.
		
		By repeating this extension process a finite number of times, we ultimately construct a weak solution to the homogenized problem \eqref{homo-problem} on the entire interval $[0, T].$  
		
		\smallskip
		\noindent
		{\bf 6.} To demonstrate uniqueness, suppose there exist two weak solutions  $(\mathbf{u}^+, \mathbf{w}^f, \mathbf{u}^-)$
		and $(\tilde{\mathbf{u}}^+, \tilde{\mathbf{w}}^f, \tilde{\mathbf{u}}^-)$ of \eqref{homo-problem}. 
		Define  the functions
		\begin{equation*}
			Z(t) := \max_{k;\, x_1\in [0,\ell];\, \tau\in [0, t]} \big|w^f_{k}(x_1, \tau) - \tilde{w}^f_{k}(x_1, \tau)\big|,
		\end{equation*}
		$$
		V(t) :=  \max_{k; \,x_1\in [0,\ell]; \, \tau\in [0, t]} \big|u^+_{k}(x_1, 0, \tau) - \tilde{u}^+_{k}(x_1, 0, \tau)\big|
		+
		\max_{k;\, x_1\in [0,\ell]; \, \tau\in [0, t]} \big|u^-_{k}(x_1, 0, \tau) - \tilde{u}^-_{k}(x_1, 0, \tau)\big|
		$$
		for all $t\in [0,T].$
		Arguing similarly to item 2, we derive the estimate
		$$
		Z(t) + V(t) \le C_2 \, t^{\frac{1}{4}} \, \left(Z(t) + V(t)\right) + C_1 \, \int_{0}^{t} \left(Z(\tau) + V(\tau)\right)\, d\tau.
		$$
		Choosing $T_0 > 0$ such that $C_2\, (T_0)^{\frac{1}{4}} < 1,$ we obtain 
		$$
		Z(t) + V(t) \le C_3 \, \int_{0}^{t} \left(Z(\tau) + V(\tau)\right)\, d\tau,
		$$
		from which, by Gronwall's lemma, it follows that $Z(t) = V(t) = 0$ for $t\in [0,T_0].$
		
		Now consider the difference
		$
		\mathbf{V}^\pm(x,t):= \mathbf{u}^\pm(x,t) - \tilde{\mathbf{u}}^\pm(x,t). 
		$
		Proceeding as in item 3, we derive the estimate
		\begin{equation*}
			\int_{\Omega^\pm} |\mathbf{V}^\pm(x,t)|^2\, dx \le C_4 \, \int_{0}^{t}\int_{\Omega^\pm} |\mathbf{V}^\pm(x,\tau)|^2 \, dx\, d\tau  \quad \text{for any } \ \ t \in [0,T_0]. 
		\end{equation*}
		Applying Gronwall's lemma again yields   
		$\mathbf{V}^\pm(x,t) =0$ for  $t \in [0, T_0].$
		
		Repeating the above argument on successive intervals, we conclude that the functions $Z(t),$ $V(t),$ $\mathbf{V}^\pm(x,t)$ vanish on each interval $[nT_0, (n+1)T_0]$ for $n \in \mathbb{N},$ until the entire interval $[0,T]$ is covered.
		Hence, 
		$$
		(\mathbf{u}^+, \mathbf{w}^f, \mathbf{u}^-)
		= (\tilde{\mathbf{u}}^+, \tilde{\mathbf{w}}^f, \tilde{\mathbf{u}}^-) \quad \text{for all} \ \ t\in [0,T],
		$$ 
		which proves uniqueness. 
	\end{proof}
	
	\begin{lemma}\label{Th-4-4}
		The weak solution to problem \eqref{homo-problem} is classical, and 
		$$
		\mathbf{u}^\pm_0 \in H^{3+\alpha,\,2+\alpha/2}\big(\overline{\Omega^\pm}\times[0,T]; \mathbb{R}^\mathcal{M} \big),
		\quad  \mathbf{w}^f_0 \in C^{3}\big([0,\ell]\times[0,T]; \mathbb{R}^\mathcal{M}\big).
		$$  
	\end{lemma}

	\begin{proof} {\bf 1.} In point 3 of Theorem~\ref{Th-4-1} we proved that $\mathbf{u}_0^\pm$ is continuous on $\overline{\mathcal{I}^{f,T}}.$
		Since $\mathbf{u}_0^\pm|_{\partial\Omega^{\pm, T}  \setminus \mathcal{I}^{f,T}} =0,$ the function $\mathbf{u}_0^\pm$ is bounded on $\partial\Omega^{\pm,T}$. Hence Theorem 2.1 of \cite[Chapt.~V]{Lad_Sol_Ura_1968} gives $\mathbf{u}_0^\pm\in L^\infty(\Omega^{\pm,T})$, and Theorem 1.1 of the same chapter yields
		$\mathbf{u}_0^\pm\in H^{\alpha,\alpha/2}(\Omega^{\pm,T})
		$
		for some $\alpha\in(0,1)$.
		
		We  regard $\mathbf{u}_0^\pm$ as a weak solution to the linear problem
		\begin{equation}\label{par-problem}
			\left\{\begin{array}{rcll}
				\partial_t \mathbf{u}^\pm_0 -   \mathbb{D}^\pm \Delta \mathbf{u}^\pm_0
				& = & \mathbf{B}^\pm(x, t),  & (x,t) \in \Omega^{\pm, T} ,
				\\[2mm]
				\mathbf{u}^\pm_0  & = & \mathbf{0}, & \text{on} \ \ \partial\Omega^{\pm, T}  \setminus \mathcal{I}^{f,T},
				\\[2mm]
				\pm\, \mathbb{D}^\pm \partial_{x_2} \mathbf{u}^\pm_0&=& \mathbf{G}^\pm(x_1, t), & (x_1, t) \in \mathcal{I}^{f,T},
				\\[2mm]
				\mathbf{u}^\pm_0 \big|_{t=0} & = & \mathbf{0}, & \text{in} \ \ \Omega^\pm  ,
			\end{array}\right.
		\end{equation}
		where
		$\mathbf{B}^\pm(x, t) := \mathbf{F}^\pm(\mathbf{u}^\pm_0, x, t) + \boldsymbol{\mathfrak{F}}^\pm(x,t)$ and
		$\mathbf{G}^\pm(x_1, t) := \upharpoonleft\!\! {S}^\pm\!\!\upharpoonright_1\widetilde{\boldsymbol{\Upsilon}}^\pm\big(\mathbf{u}^\pm_0(x_1,0,t), \mathbf{w}^f_0(x_1,t), x_1, t \big).$
		
		By assumption ${\bf A3},$ the vector-function $\mathbf{B}^\pm \in L^\infty(\Omega^{\pm, T}; \mathbb{R}^M)$ and has compact support in $\Omega\times(0,T).$ Due to the continuity of $\mathbf{u}_0^\pm$ and $\mathbf{w}_0^f$ on $\overline{\mathcal{I}^{f,T}},$
		$\mathbf{G}^\pm$ is continuous on $\overline{\mathcal{I}^{f, T}}.$ Moreover, by assumption ${\bf A3},$ it has compact support in $(0,\ell)$ and $\mathbf{G}^\pm(\cdot,t)\equiv\mathbf{0}$ for $t\in[0,\delta].$ These properties prevent the formation of corner singularities at Dirichlet–Neumann junctions, and therefore standard boundary-value theory gives
		$
		\mathbf{u}_0^\pm\in W^{2,1}_2(\Omega^{\pm,T};\mathbb{R}^\mathcal{M}),
		$
		cf. \cite{solonnikov1971boundary,grisvard1985elliptic}.
		By the trace theorem for parabolic Sobolev spaces (see \cite{Lad_Sol_Ura_1968,lions1968non}), the trace of $\mathbf{u}_0^\pm$ on the Neumann boundary belongs to the anisotropic Sobolev space
		\[
		\mathbf{u}_0^\pm(\cdot,0,\cdot)\in W^{3/2,\,3/4}_2\big((0,\ell)\times(0,T);\mathbb{R}^\mathcal{M}\big).
		\]
		Using the parabolic Sobolev–Hölder embedding on bounded domains, we have 
		\[
		W^{3/2,\,3/4}_2\big((0,\ell)\times(0,T)\big)\hookrightarrow H^{\alpha,\alpha/2}\big([0,\ell]\times[0,T]\big) \quad \text{for every} \ \ \alpha\in(0,1),
		\]
		and consequently
		\begin{equation}\label{Hord-1}
			\mathbf{u}_0^\pm(x_1,0,t)\in H^{\alpha,\alpha/2}\big([0,\ell]\times[0,T];\mathbb{R}^\mathcal{M}\big).
		\end{equation}

		\smallskip
		\noindent
		\textbf{2.} \emph{H\"older continuity of $\mathbf{w}^f_0$.} We consider the function $\mathbf{w}^f_0$ as a continuous solution to the mixed semi-linear problem
		\begin{equation}\label{mixed-problem}
			\left\{\begin{array}{l}
				\partial_t \mathbf{w}^f_0  + \hat{\mathrm{v}} \, \partial_{x_1} \mathbf{w}^f_0  = \mathbf{Q}\big(\mathbf{w}^f_0(x_1,t), x_1, t\big), \quad 
				(x_1,t) \in \mathcal{I}^{f, T},
				\\[3pt]
				\mathbf{w}^f_0(0,t)= \mathbf{0}, \ \ t\in [0,T], \qquad \mathbf{w}^f_0(x_1,0)= \mathbf{0}, \ \ x_1\in [0, \ell],
			\end{array}\right.
		\end{equation}
		where 
		$
		\mathbf{Q}(\mathbf{s}, x_1, t) := \widehat{\mathbf {F}}\big(\mathbf{u}^+_0(x_1,0,t), \mathbf{u}^-_0(x_1,0,t), \mathbf{s}, x_1, t\big).
		$
		We show that $\mathbf{w}^f_0\in H^{\alpha,\alpha/2}\big([0,\ell]\times[0,T];\mathbb{R}^\mathcal{M}\big)$, with a bound depending only on the data in \eqref{mixed-problem} and the structural constants in the assumptions.
		
		At first, we show that  for each fixed $\mathbf{s}$, $(x_1,t)\mapsto \mathbf{Q}(\mathbf{s},x_1,t)$ is in $H^{\alpha,\alpha/2}\big([0,\ell]\times[0,T]; \mathbb{R}^\mathcal{M}\big)$ with a bound uniform in $\mathbf{s}.$
		Taking into account the definition of $\widehat{\mathbf {F}}$ (see \eqref{hat F}), 
		it suffices to verify the mixed Hölder regularity, e.g., for
		$\Phi_k^+(\mathbf{u}^+_0(x_1,0,t),\mathbf{s},\xi,x_1,t).$
		Using the triangle inequality and assumptions ${\bf A3}$ (see also Remark~\ref{rem-2-3}), we obtain for any $(x_1,t),(x_1',t')\in[0,\ell]\times[0,T]$, 
		\begin{multline*}
			\big| \Phi_k^+(\mathbf{u}^+_0(x'_1,0,t'),\mathbf{s},\xi,x'_1,t') - \Phi_k^+(\mathbf{u}^+_0(x_1,0,t),\mathbf{s},\xi,x_1,t)\big|
			\\
			\le
			\big| \Phi_k^+(\mathbf{u}^+_0(x'_1,0,t'),\mathbf{s},\xi,x'_1,t') - \Phi_k^+(\mathbf{u}^+_0(x_1,0,t),\mathbf{s},\xi,x'_1,t')\big|
			\\
			+
			\big| \Phi_k^+(\mathbf{u}^+_0(x_1,0,t),\mathbf{s},\xi,x'_1,t') - \Phi_k^+(\mathbf{u}^+_0(x_1,0,t),\mathbf{s},\xi,x_1,t)\big|
			\\
			\le C_1 |\mathbf{u}^+_0(x'_1,0,t') - \mathbf{u}^+_0(x_1,0,t)| + C_2 |x'_1 -x_1| + C_3 |t' - t| \le C_0 \big(|x'_1 -x_1|^\alpha + |t' - t|^{\alpha/2}\big),
		\end{multline*}
		with a constant $C_0$ independent of $(x_1,t),(x_1',t')$ and of $\mathbf{s}$.
		
		Now we prove that the solution $\mathbf{w}^f_0 = (w_{0,1},\ldots,w_{0,\mathcal{M}})$ to problem \eqref{mixed-problem} is H\"older continuous. We do this for the $k$-th component $w_{0,k},$ which is a continuous solution to the problem
		\[
		\left\{\begin{array}{l}
			\partial_t w_{0,k}(x_1,t) + \hat{\mathrm{v}} \,\partial_{x_1} w_{0,k}(x_1,t) = Q_k\big(\mathbf{w}^f_0(x_1,t),x_1,t\big), \quad (x_1,t)\in(0,\ell)\times(0,T),\\[3pt]
			w_{0,k}(0,t)=0,  \ \  t\in[0,T],\qquad w_{0,k}(x_1,0)=0, \ \  x_1\in[0,\ell].
		\end{array}\right.
		\]
		
		For $(x_1,t)\in(0,\ell)\times(0,T)$, define
		\[
		y_1(\tau;x_1,t):=x_1+\hat{\mathrm{v}}\,(\tau-t),\qquad 
		\mathbf{z}(\tau;x_1,t):=\mathbf{w}^f_0\big(y_1(\tau;x_1,t),\tau\big),\quad 
		z_k(\tau;x_1,t):=w_{0,k}\big(y_1(\tau;x_1,t),\tau\big).
		\]
		Then (see \eqref{eq25})
		\[
		w_{0,k}(x_1,t)=\int_{0}^{t} Q_k\!\left(\mathbf{z}(\tau;x,t),\,y_1(\tau;x_1,t),\,\tau\right)\,d\tau, \quad (x_1,t)\in \mathcal{D}_1,
		\]
		\[
		w_{0,k}(x_1,t)=\int_{t-\frac{x_1}{\hat{\mathrm{v}}}}^{t} Q_k\!\left(\mathbf{z}(\tau;x,t),\,y_1(\tau;x_1,t),\,\tau\right)\,d\tau, \quad (x_1,t)\in \mathcal{D}_2.
		\]
		It should be noted that the two integral formulas coincide on the characteristic interface 
		$\Sigma :=
		\{(x_1,t)\in[0,\ell]\times[0,T]:\, t= x_1/\hat{\mathrm{v}} \},
		$
		which separates $\mathcal{D}_1$ and $\mathcal{D}_2$.

		\subsubsection*{Spatial Hölder estimate on $\mathcal{D}_1$ at fixed time}
		Fix $t\in[0,T]$ and $x_1,x_1'\in[0,\ell]$ such that $t\le x_1/\hat{\mathrm{v}}$ and $t\le x_1'/\hat{\mathrm{v}};$ so $(x_1,t),(x_1',t)\in\mathcal{D}_1$. Then
		\[
		w_{0,k}(x_1,t)- w_{0,k}(x_1',t)=\int_0^t \Big(Q_k\big(\mathbf{z}(\tau;x_1,t),y_1(\tau;x_1,t),\tau\big) - Q_k\big(\mathbf{z}(\tau;x_1',t),y_1(\tau;x_1',t),\tau\big)\Big)\,d\tau.
		\]
		Using the spatial $\alpha$-Hölder continuity of $Q_k$ (uniform in $\mathbf{s}$) and its Lipschitz continuity in $\mathbf{s}$, we get
		\begin{equation}\label{Hor-2}
			|w_{0,k}(x_1,t)- w_{0,k}(x_1',t)|\le C_1 \int_0^t |\mathbf{z}(\tau;x_1,t)- \mathbf{z}(\tau;x_1',t)|\,d\tau + t\,C_0\,|x_1-x_1'|^\alpha.
		\end{equation}
		To bound the integral term, note that for every $\tau\in[0,t]$,
		\[
		z_k(\tau;x_1,t)-z_k(\tau;x_1',t)
		=\int_0^\tau \Big(Q_k\big(\mathbf{z}(\eta;x_1,t),y_1(\eta;x_1,t),\eta\big)
		- Q_k\big(\mathbf{z}(\eta;x_1',t),y_1(\eta;x_1',t),\eta\big)\Big)\,d\eta,
		\]
		whence, by the same splitting and the Lipschitz/Hölder bounds on $Q_k$,
		$$
		|\mathbf{z}(\tau;x_1,t)- \mathbf{z}(\tau;x_1',t)| \le \tilde{C}_1 \int_0^\tau |\mathbf{z}(\eta;x_1,t)- \mathbf{z}(\eta;x_1',t)|\,d\eta + \tau\, \tilde{C}_0\,|x_1-x_1'|^\alpha.
		$$
		By Grönwall’s inequality,
		\begin{equation}\label{Hor-3}
			|\mathbf{z}(\tau;x_1,t)- \mathbf{z}(\tau;x_1',t)| 
			\le \frac{\tilde{C}_0}{\tilde{C}_1}\big(e^{\tilde{C}_1 \tau}-1\big)\,|x_1-x_1'|^\alpha .
		\end{equation}
		Substituting \eqref{Hor-3} in \eqref{Hor-2} yields  
		$$
		|w_{0,k}(x_1,t)- w_{0,k}(x_1',t)|\le C_2 \, |x_1-x_1'|^\alpha,
		$$
		with the constant $C_2$ independent of $x_1,x_1'$ and $t$.
		
		\subsubsection*{Spatial Hölder estimate on $\mathcal{D}_2$ at fixed time}
		Fix $t\in[0,T]$ and $x_1,x_1'\in[0,\ell]$ with $t> x_1/\hat{\mathrm{v}}$ and $t> x_1'/\hat{\mathrm{v}}$. Then
		\[
		\begin{aligned}
			w_{0,k}(x_1,t)-w_{0,k}(x_1',t) & = \int_{t-\frac{x_1}{\hat{\mathrm{v}}}}^t Q_k\!\left(\mathbf{z}(\tau;x_1,t),y_1(\tau;x_1,t),\tau\right)\,d\tau
			-\int_{t-\frac{x_1'}{\hat{\mathrm{v}}}}^t Q_k\!\left(\mathbf{z}(\tau;x_1',t),y_1(\tau;x_1',t),\tau\right)\,d\tau
			\\
			&=\int_{t-\frac{x_1'}{\hat{\mathrm{v}}}}^t \!\Big(Q_k(\mathbf{z}(\tau;x_1,t),y_1(\tau;x_1,t),\tau)-Q_k(\mathbf{z}(\tau;x_1',t),y_1(\tau;x_1',t),\tau)\Big)\,d\tau
			\\
			&\quad + \int_{t-\frac{x_1}{\hat{\mathrm{v}}}}^{t-\frac{x_1'}{\hat{\mathrm{v}}}} Q_k(\mathbf{z}(\tau;x_1,t),y_1(\tau;x_1,t),\tau)\,d\tau.
		\end{aligned}
		\]
		The first term is handled exactly as in $\mathcal{D}_1$, giving the bound $ C_3\,|x_1-x_1'|^\alpha.$
		For the second term, use boundedness of $Q_k$ and the length of the interval:
		\[
		\left|\int_{t-\frac{x_1}{\hat{\mathrm{v}}}}^{t-\frac{x_1'}{\hat{\mathrm{v}}}} Q_k(\mathbf{z}(\tau;x_1,t),y_1(\tau;x_1,t),\tau)\,d\tau\right|
		\le \|Q_k\|_{L^\infty}\,\frac{|x_1-x_1'|}{\hat{\mathrm{v}}}
		\le C_4\,|x_1-x_1'|^\alpha.
		\]
		As a result, we obtain
		\[
		|w_{0,k}(x_1,t)-w_{0,k}(x_1',t)|\le C_5\,|x_1-x_1'|^\alpha,
		\]
		with $C_5$ independent of $x_1,x_1',t$.
		
		If $x_1\in \mathcal{D}_1$ and $x'_1\in \mathcal{D}_2,$ the interpolating via the interface point $x^*_1 = \hat{\mathrm{v}}\, t $ on $\Sigma$ gives 
		\[
		|w_{0,k}(x'_1,t)-w_{0,k}(x_1,t)|
		\le |w_{0,k}(x'_1,t)-w_{0,k}(x^*_1,t)| + |w_{0,k}(x^*_1,t)-w_{0,k}(x_1,t)|.
		\]
		Each term is controlled by the estimates of two previous cases, giving the same bound $C\,|x'_1- x_1|^{\alpha}$.
		
		The H\"older estimate in time for a fixed $x_1$ is obtained by applying the same technique first in $\mathcal{D}_1$ and then in $\mathcal{D}_2$, taking into account the $\frac{\alpha}{2}$-H\"older continuity of the function $\mathbf{Q}$.
		
		\medskip
		
		Combining the spatial estimates on $\mathcal{D}_1$ and $\mathcal{D}_2$ with the temporal estimate at fixed $x_1$, we conclude that for each $k=1,\dots,M$,
		\[
		|w_{0,k}(x_1,t)-w_{0,k}(x_1',t')|
		\le C\big(|x_1-x_1'|^\alpha + |t'-t|^{\alpha/2}\big),
		\]
		for all $(x_1,t),(x_1',t')\in[0,\ell]\times[0,T]$, with $C$ depending only on the data of \eqref{mixed-problem} and the structural constants from {\bf A3}. Therefore,
		$\displaystyle 
		\mathbf{w}^f_0\in H^{\alpha,\alpha/2}\big([0,\ell]\times[0,T];\mathbb{R}^\mathcal{M}\big).
		$
		
		\smallskip
		
		\noindent
		\textbf{3.} \emph{H\"older regularity of $\mathbf{u}^\pm_0$.} Returning to the linear problem \eqref{par-problem} and taking into account results obtained in the previous points, we see that   
		$$
		\mathbf{B}^\pm \in H^{\alpha,\alpha/2}\big(\Omega^{\pm, T}; \mathbb{R}^\mathcal{M}\big) \quad \text{and} \quad \mathbf{G}^\pm \in  H^{\alpha,\alpha/2}\big([0,\ell]\times[0,T]; \mathbb{R}^\mathcal{M}\big). 
		$$
		
		Then, by standard parabolic Schauder theory (see, e.g., \cite[Ch.~IV]{Lad_Sol_Ura_1968}, \cite[Ch.~IV]{Lieberman}), $\mathbf{u}_0^\pm$ has interior $H^{2+\alpha,1+\alpha/2}$-regularity in $\Omega^\pm\times(0,T)$.
		
		The only obstruction to global $H^{2+\alpha,1+\alpha/2}$ up to $\overline{\Omega^\pm}\times[0,T]$ is the presence of mixed (Dirichlet–Neumann) corner points $(0,0)$ and $(\ell,0)$ on the interface $\mathcal{I}^{f,T}$. To remove this, we reflect across the vertical Dirichlet edges $x_1=0$ and $x_1=\ell$.
		
		Consider the corner $(0,0).$ In the extended domain $\widetilde{\Omega}^{\pm, T} = (-\ell, \ell)\times (0,\mathfrak{h}^\pm) \times (0,T)$ we define the odd extension $\widetilde{\mathbf{u}}^\pm_0,$ $\widetilde{\mathbf{B}}^\pm$ and $\widetilde{\mathbf{G}}^\pm$
		in $x_1$ (componentwise), i.e.,
		\[
		\widetilde{u}(x_1,x_2,t)=\begin{cases}
			u(x_1,x_2,t), & (x_1,x_2,t)\in [0, \ell]\times [0,\mathfrak{h}^\pm] \times [0,T],
			\\
			-\,u(-x_1,x_2,t), & (x_1,x_2,t) \in [-\ell, 0]\times [0,\mathfrak{h}^\pm] \times [0,T].
		\end{cases} 
		\]
		A direct computation shows 
		$\partial_t \widetilde{\mathbf{u}}^\pm_0 - \mathbb{D}^\pm \Delta\widetilde{\mathbf{u}}^\pm_0
		=  \widetilde{\mathbf{B}}^\pm$ in $\widetilde{\Omega}^{\pm, T}$, $\widetilde{\mathbf{u}}^\pm_0 = \mathbf{0}$ on 
		$\partial\widetilde{\Omega}^{\pm,T} \setminus \big((-\ell,\ell) \times(0,T)\big),$ and
		$\pm\, \mathbb{D}^\pm \partial_{x_2}\widetilde{\mathbf{u}}^\pm_0 = \widetilde{\mathbf{G}}^\pm$ on  $(-\ell,\ell)\times (0,T)$. 
		
		Since $\mathbf{B}^\pm$ and $\mathbf{G}^\pm$ are Hölder and compactly supported in the corresponding regions, the odd extensions $\widetilde{\mathbf{B}}^\pm$ and $\widetilde{\mathbf{G}}^\pm$ 
		preserve $H^{\alpha,\alpha/2}$ on $\widetilde{\Omega}^{\pm,T}$ and $(-\ell,\ell)\times[0,T]$, respectively. Moreover, the assumption $\mathbf{G}^\pm(\cdot,t)\equiv \mathbf{0}$ on $[0,\delta]$ implies $\widetilde{\mathbf{G}}^\pm(\cdot,t)\equiv \mathbf{0}$ on $[0,\delta]$, giving initial-time compatibility for the Neumann flux on the reflected interval.
		
		In the reflected configuration, the corner $(0,0)$ is converted into a \emph{smooth Neumann boundary point} for the extended problem along $x_2=0$ with $x_1\in(-\ell,\ell)$, and data vanish in a neighborhood of $(0,0)$. 
		Hence classical parabolic Schauder boundary estimates (Neumann case) yield
		\[
		\widetilde{\mathbf{u}}^\pm_0 \in H^{2+\alpha,\,1+\alpha/2}\quad \text{in a neighborhood of }\ (0,0)\times[0,T].
		\]
		Restricting back to $x_1\ge 0$ gives the same regularity to ${\mathbf{u}}^\pm_0$ near $(0,0)$.
		
		Repeating the same odd-reflection construction across $x_1=\ell$ (after translating coordinates), we obtain  the same regularity to  ${\mathbf{u}}^\pm_0$ in a neighborhood of $(\ell,0).$ 
		
		Away from the corners, the boundary is flat and of pure type. 
		On Dirichlet patches, the homogeneous condition $\mathbf{u}^\pm_0=0$ together with the Hölder regularity of $\mathbf{B}^\pm$ yields $H^{2+\alpha,1+\alpha/2}$ up to the boundary. 
		On the Neumann segment $\{x_2=0\}$ away from the endpoints, $\mathbf{G}^\pm \in H^{\alpha,\alpha/2}([0,\ell]\times[0,T]; \mathbb{R}^\mathcal{M})$ and the vanishing $\mathbf{G}^\pm\equiv \mathbf{0}$ on $[0,\delta]$ ensure compatibility with the initial condition, and thus $H^{2+\alpha,1+\alpha/2}$ regularity up to $[0,\ell]\times\{x_2=0\}\times [0,T]$. 
		These properties of $\mathbf{B}^\pm$ and $\mathbf{G}^\pm$ guarantee time–Hölder continuity at the initial slice. 
		Consequently,
		$
		\mathbf{u}^\pm_0 \in H^{2+\alpha,\,1+\alpha/2}\big(\overline{\Omega^\pm}\times[0,T]; \mathbb{R}^\mathcal{M}\big).
		$
		
		\smallskip
		
		\noindent
		\textbf{4.} \emph{Increasing the regularity of $\mathbf{u}^\pm_0$ and $\mathbf{w}^f_0.$} 
		We return to the mixed semi-linear problem \eqref{mixed-problem}. For each fixed $\mathbf{s}$
		the map $(x_1,t)\mapsto \mathbf{Q}(\mathbf{s},x_1,t)$ belongs to 
		$
		H^{1+\alpha,\,1+\alpha/2}\big([0,\ell]\times[0,T];\mathbb{R}^\mathcal{M}\big),
		$
		with a bound uniform in $\mathbf{s}$. Hence, by Theorem 2 of \cite{Myshkis_1960}, the solution $\mathbf{w}^f_0$ to problem \eqref{mixed-problem} is classical; in particular
		$
		\mathbf{w}^f_0\in C^1\big([0,\ell]\times[0,T];\mathbb{R}^\mathcal{M}\big),
		$
		and, arguing as in point {\bf 2}, $\mathbf{w}^f_0\in H^{1+\alpha,\,1+\alpha/2}\big([0,\ell]\times[0,T];\mathbb{R}^\mathcal{M}\big)$.
		
		Combining this regularity of $\mathbf{w}^f_0$ with the previously established regularity of $\mathbf{u}^\pm_0$ and assumptions {\bf A3}, the data of the linear problems \eqref{par-problem} satisfy
		$$
		\mathbf{B}^\pm\in H^{1+\alpha,\,1+\alpha/2}\big(\overline{\Omega^\pm}\times[0,T];\mathbb{R}^\mathcal{M}\big), \qquad
		\mathbf{G}^\pm\in H^{1+\alpha,\,1+\alpha/2}\big([0,\ell]\times[0,T];\mathbb{R}^\mathcal{M}\big).
		$$
		Arguing as in point {\bf 3} (interior and boundary Schauder estimates together with the reflection near mixed corners), we obtain the improved regularity
		$
		\mathbf{u}^\pm_0\in H^{3+\alpha,\,2+\alpha/2}\big(\overline{\Omega^\pm}\times[0,T];\mathbb{R}^\mathcal{M}\big).
		$
		
		The higher regularity of $\mathbf{u}^\pm_0$ implies that, for each fixed \(\mathbf{s}\),  
		$\mathbf{Q}(\mathbf{s},\cdot,\cdot) \in 
		H^{2+\alpha,\,2+\alpha/2}\big([0,\ell]\times[0,T];\mathbb{R}^\mathcal{M}\big),
		$
		uniformly in $\mathbf{s}$. Due to assumptions {\bf A3} the function $\mathbf{Q}$ has $C^3$-regularity in $\mathbf{s}.$ By the same argument used in Remark 3.2 of \cite{Mel-Roh_JMAA-2024} (see also the classical regularity results in \cite{Friedrichs1948}), this improved regularity of the right-hand side upgrades \(\mathbf{w}_0^f\) to
		$
		\mathbf{w}_0^f\in C^{2}\big([0,\ell]\times[0,T];\mathbb{R}^\mathcal{M}\big).
		$
		Finally, repeating these arguments again completes the successive regularity improvements and leads to the claimed regularity of $\mathbf{w}_0^f.$
	\end{proof}	
	
	
	\subsection{Boundary-layer asymptotics}\label{BLPs}
	The regular part \eqref{regul} of the asymptotic expansion does not satisfy the boundary condition on the right side $\Gamma^T_{\ell,\varepsilon}$ of $\partial\Omega^f_\varepsilon$, leaving residuals of order $\mathcal{O}(1)$ on that part of the boundary. To compensate for these residuals and enforce the boundary condition on $\Gamma^T_{\ell,\varepsilon}$ in \eqref{probl}, we introduce the boundary-layer ansatzes
	\begin{equation}\label{prim+}
		\Pi_{0,k}\left(\frac{x_1 - \ell}{\varepsilon}, \frac{x_2}{\varepsilon}, t\right) + \varepsilon \, \Pi_{1,k}\left(\frac{x_1 - \ell}{\varepsilon}, \frac{x_2}{\varepsilon}, t\right)
		+ \ldots, \quad k \in \{1,\ldots,\mathcal{M}\},
	\end{equation}
	which are localized in a neighborhood of $\Gamma^T_{\ell,\varepsilon}$.
	
	According to assumption ${\bf A3}$, the boundary conditions on the oscillating and perforated parts  of $\partial\Omega^f_\varepsilon$ vanish in the narrow right-hand region $\Omega^f_{\varepsilon, \ell -\delta_1} := \Omega^f_\varepsilon \cap \{x \colon x_1 \in (\ell -\delta_1, \ell)\},$ $\delta_1 >0.$ In a neighborhood of $\Omega^f_{\varepsilon, \ell -\delta_1}$, we introduce the local coordinates $\zeta = (\zeta_1, \zeta_2)$ defined by
	$
	\zeta_1 = \frac{x_1 - \ell}{\varepsilon}, \quad \zeta_2 = \frac{x_2}{\varepsilon}.
	$
	Passing formally to the limit $\varepsilon \to 0$, we obtain a $1$-periodic perforated semi‑infinite strip, perforated in the horizontal direction and having $1$-periodically varying upper and lower boundaries
	(see Fig.~\ref{fig-4})
	\[
	\mathfrak{C}  := \left\{\zeta \colon \zeta_1 \in (-\infty, 0), \quad \zeta_2 \in \left(-h_-(\zeta_1), \, h_+(\zeta_1)\right)\right\} \setminus \bigcup_{k \in \mathbb{Z}} \left(k \, \vec{\boldsymbol{e}}_1 + {T}_0\right).
	\]
	\begin{figure}[htbp]
		\begin{center}
			\includegraphics[width=11cm]{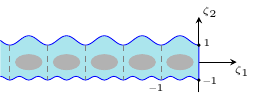}
		\end{center}
		
		\vspace{-0.3cm}
		\caption{Semi-strip $\mathfrak{C}$ with perforations and upper/lower boundaries all  varying with period$1$}\label{fig-4}
	\end{figure}
	
	Let $\Gamma_\ell$ denote the interval 
	$
	\left\{\zeta \colon \zeta_1 = 0, \  \zeta_2 \in (-1,1)\right\},
	$
	which is the vertical part of the boundary $\partial\mathfrak{C}$ on the right.
	
	Substituting the boundary-layer ansatz for each $k$ into the corresponding differential equation and boundary conditions of problem \eqref{probl} within $\Omega^f_{\varepsilon, \ell -\delta_1}$, taking into account the $1$-periodicity of the coefficients and Remark~\ref{rem-2-4},  and collecting terms of equal powers of $\varepsilon$, we arrive at the following problem for the leading-order coefficient $\Pi_{0,k}$:
	\begin{equation}\label{prim+probl+0}
		\left\{\begin{array}{rcll}
			\mathcal{L}_{\zeta}\big(\Pi_{0,k}(\zeta,t)\big)
			& =    & 0,
			& \quad \zeta\in \mathfrak{C},
			\\[2mm]
			-\big(\mathbb{D}(\zeta)\nabla_{\zeta}\Pi_{0,k}(\zeta,t)\big) \cdot \boldsymbol{\nu}(\zeta) & =
			& 0,
			& \quad \zeta\in \partial\mathfrak{C} \setminus \Gamma_\ell,
			\\[2mm]
			\Pi_{0, k}(0, {\zeta}_2,t) & =
			& \Phi_{0, k}(t),
			& \quad {\zeta}_2\in \Gamma_\ell,
			\\[2mm]
			\Pi_{0, k}(\zeta,t) & \to
			& 0,
			& \quad \zeta_1\to -\infty,
		\end{array}\right.
	\end{equation}
	where $\Phi_0(t) := q^{\ell}_k(t) - w_{0,k}(\ell,t),$ the variable $t\in [0,T]$ is considered as a parameter, and
	$$
	\mathcal{L}_{\zeta}U := -\nabla_{\zeta}\!\cdot\!\big(\mathbb{D}(\zeta)\nabla_{\zeta}U(\zeta)\big)+ \nabla_{\zeta}\!\cdot\!\big(\overrightarrow{V}(\zeta)\, U(\zeta)\big).
	$$
	
The behavior of solutions to elliptic problems at infinity in unbounded cylindrical domains has been extensively studied; see, for example, \cite{Kon-Ole_1983,Koz-Maz-Ros_97,Land-Pan-1977,Nazarov-1982,Naz99,Ole-Ios-1981,Ole_book_1996,Piat-1984}. Most works treat rectilinear cylinders, where the product geometry allows the use of Fourier/Floquet techniques, explicit Green’s kernels, and operator-valued multipliers. These methods may fail or become substantially more technical in curvilinear or perforated settings. In particular, \cite{Nazarov-1982} employs a Gelfand-type transform to analyze problems (elliptic in the Douglis–Nirenberg sense) in periodic cylinders with oscillating boundaries, while \cite{Ole-Ios-1981} establishes stabilization results for divergence-form equations in periodically perforated cylinders. Related behavior in partially perforated layers is addressed in \cite{Melnyk-2026}. Problems involving first-order terms require additional care and have been treated in \cite{Nazarov-1982,Naz99,Ole_book_1996,Piat-1984,Piat-2009,Mel-Kle-2022,Mel-Roh_Non-Diff-2024}.

In what follows, under assumptions \(\mathbf{A1}\) and \(\mathbf{A2}\), and using the Hopf–Oleinik boundary point lemma \cite{Hopf-1952,Oleinik-1952},  the maximum principle, and Schauder estimates, 
we provide a short proof of the exponential decay to zero of the solution to problem \eqref{prim+probl+0}, together with its first and second derivatives, as \(\zeta_1 \to -\infty\).

	To study problem \eqref{prim+probl+0} we omit the indices $0, k$, assume that $\Phi(t) \equiv 1$ (otherwise, we introduce a new function $\widetilde{\Pi}_{N} = \Pi_{N}/\Phi(t)$), and introduce the following notation:
	\begin{gather*}
		\mathfrak{C}_{m,m + N} := \mathfrak{C} \cap \{\zeta\colon \ \zeta_1\in (-m-N, \, -m)\}, \quad m \in \mathbb{N}_0, \ \ N \in \mathbb{N},	
		\\
		\Gamma_{-N} = \left\{\zeta \colon \zeta_1 = -N, \ \zeta_2 \in (-1,1)\right\}.
	\end{gather*}
	It is easy to see that $\mathfrak{C}_{0,1} = Y_0 - \vec{\boldsymbol{e}}_1$ and  $\Gamma_0 = \Gamma_\ell.$ 
	
	\begin{lemma}\label{lemma-4-5}
		Problem \eqref{prim+probl+0} admits a unique solution \(\Pi\) in the H\"older space
		\(H^{2+\mu}\big(\overline{\mathfrak{C}}\big)\), for some \(\mu \in (0,1),\) and this solution
		decays exponentially at infinity together with its first and second derivatives:
		\begin{equation}\label{bv-1}
			|\Pi(\zeta)| + |\nabla\Pi(\zeta)| + |D^2\Pi(\zeta)| = \mathcal{O}\big(\exp(\delta_0 \, \zeta_1)\big)
			\quad \text{as} \quad \zeta_1 \to -\infty,
		\end{equation}
		for some constant \(\delta_0 > 0\).
	\end{lemma}

	\begin{proof}{\bf 1.} For each $N \in \mathbb{N}$ we consider the problem 
		\begin{equation}\label{prim+probl+N}
			\left\{\begin{array}{rcll}
				\mathcal{L}_{\zeta}\big(\Pi_{N}(\zeta,t)\big)
				& =    & 0,
				& \quad \zeta\in \mathfrak{C}_{0,N},
				\\[2mm]
				-\big(\mathbb{D}(\zeta)\nabla_{\zeta}\Pi_{N}(\zeta,t)\big) \cdot \boldsymbol{\nu}(\zeta) & =
				& 0,
				& \quad \zeta\in \partial\mathfrak{C}_{0,N} \setminus \big(\Gamma_0 \cup \Gamma_{-N}\big),
				\\[2mm]
				\Pi_{N}\big|_{\zeta\in \Gamma_0} = 1, & & &\Pi_{N}\big|_{\zeta\in \Gamma_{-N}} = 0.  
			\end{array}\right.
		\end{equation}
		By virtue of assumption ${\bf A2},$ the corresponding homogeneous problem has only trivial solution. Therefore, by the Fredholm alternative, problem \eqref{prim+probl+N} has a unique weak solution. 
		
		Taking into account the assumptions on the functions $h_\pm$ and conditions ${\bf A1}$ and ${\bf A2}$, and proceeding similarly to point 3 in  the proof of Lemma~\ref{Th-4-4}, we can transform the corners on $\partial\mathfrak{C}_{0,N}$ into smooth Dirichlet points by applying even extensions in $\zeta_2$ through the corresponding horizontal sides.  Therefore, by Schauder theory, we obtain that $\Pi_N$ belongs to the H\"older space $H^{2+\mu}\big(\overline{\mathfrak{C}_{0,N}}\big).$ 
		
		By the maximum principle, the solution $\Pi_N$ cannot attain its minimum or maximum inside of $\mathfrak{C}_{0,N}.$ 
		In addition, the extremum cannot be attained on $\partial\mathfrak{C}_{0,N} \setminus \big(\Gamma_0 \cup \Gamma_{-N}\big),$ as the Hopf–Oleinik boundary point lemma combined with the uniform positive definiteness of the matrix $\mathbb{D}$ ensures that the conormal derivative $\big(\mathbb{D}\nabla_{\zeta}\Pi_{N}\big) \cdot \boldsymbol{\nu}$ is nonzero at such points.
		Consequently, we conclude that $0 \le \Pi_N(\zeta) \le 1$ for all $\zeta \in \overline{\mathfrak{C}_{0,N}}.$

		\smallskip
		
		{\bf 2.} Consider a smooth cutoff function $\chi_{\delta_1}(\zeta_1),$ defined for  $\zeta_1 \le 0,$ such that $\chi_{\delta_1}(0) =1$ and
		$\chi_{\delta_1}(\zeta_1) = 0$ for $\zeta_1 \le -\delta_1,$ where the positive number $\delta_1$ is chosen so that 
		$\mathfrak{C}_{0,N} \cap \big((-\delta_1,0)\times (-1, 1)\big)= (-\delta_1,0)\times (-1, 1).$  This is possible due to the assumptions on the functions $h_-$ and$h_+,$ and the structure of the periodicity cell $Y_0.$ Define $\widetilde{\Pi}_N := \Pi - \chi_{\delta_1};$ then $\widetilde{\Pi}_N$  is a weak solution to the problem
		\begin{equation}\label{prim+probl+N+}
			\left\{\begin{array}{rcll}
				\mathcal{L}_{\zeta}\big(\widetilde{\Pi}_{N}\big)
				& =    & \widetilde{F},
				& \quad \zeta\in \mathfrak{C}_{0,N},
				\\[2mm]
				-\big(\mathbb{D}(\zeta)\nabla_{\zeta}\widetilde{\Pi}_{N}\big) \cdot \boldsymbol{\nu}(\zeta) & =
				& 0,
				& \quad \zeta\in \partial\mathfrak{C}_{0,N} \setminus \big(\Gamma_0 \cup \Gamma_{-N}\big),
				\\[2mm]
				\widetilde{\Pi}_{N}\big|_{\zeta\in \Gamma_0} = 0, & & &\widetilde{\Pi}_{N}\big|_{\zeta\in \Gamma_{-N}} = 0,
			\end{array}\right.
		\end{equation}
		where $\widetilde{F} := -\mathcal{L}_{\zeta}\big(\chi_{\delta_1}(\zeta_1)\big).$ Multiplying the differential equation in \eqref{prim+probl+N+} by $\widetilde{\Pi}_N$ and integrating over $\mathfrak{C}_{0,N}$ by parts, we deduce
		the equality
		\begin{equation}\label{bv-2}
			\int_{\mathfrak{C}_{0,N}}\big(\mathbb{D}(\zeta)\nabla_{\zeta}\widetilde{\Pi}_N\big) \cdot \nabla_{\zeta}\widetilde{\Pi}_N\, d\zeta
			= \int_{\mathfrak{C}_{0,N}}\widetilde{F} \, \widetilde{\Pi}_N\, d\zeta,
		\end{equation}
		To justify integration by parts, we use assumption ${\bf A2},$ which yields
		\[
		\int_{\mathfrak C_{0,N}}\nabla_{\zeta}\cdot (\overrightarrow V\, \widetilde{\Pi}_N)\, \widetilde{\Pi}_N\,d\zeta
		=\tfrac12\int_{\mathfrak C_{0,N}}\overrightarrow V\cdot\nabla_{\zeta}(\widetilde{\Pi}_N^2)\,d\zeta
		=\tfrac12\int_{\partial\mathfrak C_{0,N}\setminus (\Gamma_0 \cup \Gamma_{-N})}(\overrightarrow V\cdot\nu)\,\widetilde{\Pi}_N^2\,ds =0.
		\]
		Taking into account the uniform positive definiteness of the matrix $\mathbb{D}$ and the estimate $|\widetilde{\Pi}_N| \le 2,$  we derive  from \eqref{bv-2} the inequality
		\begin{equation*}
			\int_{\mathfrak{C}_{0,N}}|\nabla_{\zeta}\widetilde{\Pi}_N|^2\, d\zeta
			\le C_0 \int_{(-\delta_1,0)\times (-1, 1)}|\widetilde{F}| \, d\zeta =: C_1.
		\end{equation*}
		Therefore,
		\begin{equation}\label{bv-3}
			\int_{\mathfrak{C}_{0,N}}|\nabla_{\zeta}{\Pi}_N|^2\, d\zeta
			\le  C_2,
		\end{equation}
		where the constant $C_2$ is independent of $N.$
		
		Thus, for any $M\in \mathbb{N},$ the sequence $\{\Pi_N\}_{N\in \mathbb{N}}$ (extended by zero for $\zeta_1 \le - N$) is bounded in 
		the Sobolev space $W_2^1(\mathfrak{C}_{0,M}).$ Note that $L^2$-norm of $\Pi_N$ depends on $M.$ Hence, up to a subsequence,
		\begin{equation}\label{bv-4}
			\Pi_N \longrightarrow \Pi \quad \text{weakly in} \ \ W_2^1(\mathfrak{C}_{0,M}) \quad \text{as} \ \ N \to +\infty.
		\end{equation}
		Using this convergence, one can show that for any test function $\varphi \in W^1_2(\mathfrak{C}_{0,M})$ with zero traces on $\Gamma_0$ and $\Gamma_{-M},$ the following identity holds: 
		\begin{equation}\label{bv-5}
			\int_{\mathfrak{C}_{0,M}}\big(\mathbb{D}(\zeta)\nabla_{\zeta}\Pi(\zeta)\big) \cdot \nabla_{\zeta}\varphi(\zeta)\, d\zeta +
			\int_{\mathfrak{C}_{0,M}}\big(\overrightarrow{V}(\zeta) \cdot \nabla_{\zeta}\Pi(\zeta)\big) \, \varphi(\zeta)\, d\zeta =0,
		\end{equation}
		i.e., $\Pi$ is a weak solution to the problem 
		\begin{equation}\label{bv-6}
			\left\{\begin{array}{rcll}
				\mathcal{L}_{\zeta}\big(\Pi\big)
				& =    & 0,
				& \quad \zeta\in \mathfrak{C},
				\\[2mm]
				-\big(\mathbb{D}(\zeta)\nabla_{\zeta}\Pi\big) \cdot \boldsymbol{\nu}(\zeta) & =
				& 0,
				& \quad \zeta\in \partial\mathfrak{C} \setminus \Gamma_0,
				\\[2mm]
				\Pi(0, {\zeta}_2) & =
				& 1,
				& \quad {\zeta}_2\in \Gamma_0.
			\end{array}\right.
		\end{equation}
		Just as in the previous point, we justify that $\Pi \in H^{2+\mu}_{loc}(\overline{\mathfrak{C}}).$  In addition, it follows from \eqref{bv-3} and \eqref{bv-4} that $\nabla_{\zeta}\Pi \in L^2(\mathfrak{C})$ and 
		\begin{equation}\label{bv-7}
			\int_{\mathfrak{C}}|\nabla_{\zeta}{\Pi}|^2\, d\zeta	\le  C_2 .
		\end{equation}

		By the maximum principle, the sequence $\{\Pi_N\}_{N\in \mathbb{N}}$ is monotonically increasing for each $\zeta \in \overline{\mathfrak{C}},$ bounded above by $1$ and hence converges pointwise to $\Pi.$ By Dini’s theorem, this convergence is uniform on any compact subset of~$\overline{\mathfrak{C}}.$ Thus, $0 \le \Pi(\zeta) \le 1$ for all $\zeta\in \overline{\mathfrak{C}}.$
		
		For any compact $K \subset \overline{\mathfrak{C}}$ and for all $N \ge N_0,$ where $N_0$ is chosen such that $K \subset \overline{\mathfrak{C}_{0, N_0}},$ standard Schauder estimates (see. e.g. \cite{Lad_Ura_1968}) give
		\begin{equation}\label{uniform-bound}
			\|\Pi_N\|_{H^{2+\mu}(K)} \le C_1 \|\Pi_N\|_{C^0(K)} \le C_2,
		\end{equation}
		with the constant $C_2 := C_2(K)$  depending only on $K$ and the ellipticity and H\"older norms of coefficients. By Arzela-Ascoli Theorem we extract a subsequence (still denoted $\Pi_N$) such that 
		\begin{equation}\label{uniform-conv}
			\Pi_N  \to \Pi \quad\text{in} \ \ C^{2}(K) \quad \text{as} \ \ N \to +\infty.
		\end{equation}

		\smallskip
		
		{\bf 3.} Suppose there exists a point $\zeta_0 \in \mathfrak{C}$ such that $\Pi(\zeta_0)=1.$ Then, by the maximum principal, it follows that $\Pi \equiv 1$ in $\mathfrak{C}$ (clearly, there is no point in $\mathfrak{C}$ at which $\Pi$ achieves a local maximum less than $1).$
		Thus, the sequence $\{\Pi_N\}_{N\in \mathbb{N}}$ converges to $1$ in $C^{2}(K)$ on every  compact subset $ K \subset \overline{\mathfrak{C}} .$    In this step, we demonstrate that such a scenario leads to a contradiction.
		
		Considering   assumption ${\bf A2}$ and the fact that  $\mathrm{v}_0$ is a positive function on $ [-1, 1],$
		multiplying the differential equation in problem \eqref{prim+probl+N} by $\Pi_N$ and integrating over $\mathfrak{C}_{0,N}$ by parts,
		we obtain
		$$
		\int_{\mathfrak{C}_{0,N}}\big(\mathbb{D} \nabla_{\zeta}\Pi_N\big) \cdot \nabla_{\zeta}\Pi_N(\zeta)\, d\zeta - \int_{\Gamma_0}d_{11}(0,\zeta_2) \, \partial_{\zeta_1}\Pi_N(0,\zeta_2)\, d\zeta_2 + \frac12\int_{\Gamma_0} \mathrm{v}_0(\zeta_2)\, d\zeta_2 =0,
		$$	
		from which the inequality follows
		\begin{equation}\label{bv-8}
			0 < \frac12\int_{\Gamma_0} \mathrm{v}_0(\zeta_2)\, d\zeta_2  \le \int_{\Gamma_0}d_{11}(0,\zeta_2) \, \partial_{\zeta_1}\Pi_N(0,\zeta_2)\, d\zeta_2.	
		\end{equation}
		Passing to the limit in \eqref{bv-8}, we arrive at a contradiction: $\int_{\Gamma_0} \mathrm{v}_0(\zeta_2)\, d\zeta_2 =0.$ Thus, the assumption $\Pi (\zeta _0)=1$ for some $\zeta_0\in \mathfrak{C}$ is false, and hence $\Pi$  cannot be identically equal to $1$ in $\mathfrak{C}.$
		
		By the Hopf-Oleinik boundary point lemma, the solution $\Pi$ cannot attain the value $1$ at any point $\zeta_0 \in \partial\mathfrak{C} \setminus \Gamma_0.$ 
		Similarly, we prove that there are no points in $\mathfrak{C}$ and on $\partial\mathfrak{C} \setminus \Gamma_0$ at which $\Pi$ vanishes.
		Hence,
		\begin{equation}\label{bv-9}
			0 < \Pi(\zeta) < 1, \quad \zeta \in \mathfrak{C} \cup \big(\partial\mathfrak{C} \setminus \Gamma_0\big).
		\end{equation}
		
		\smallskip 
		
		{\bf 4.}  Here, we show that $\Pi$  decays exponentially to zero at infinity. Let us introduce the following notation: 
		\[
		a_m=\max_{\zeta_2\in[-1,1]}\Pi(-m,\zeta_2), \quad m \in \mathbb{N}; \quad \gamma: =a_1.
		\]
		Based on the maximum principle and on \eqref{bv-9}, we have monotonicity $a_m \ge a_{m+1} > 0$ for every $m \in \mathbb{N}$ and $\gamma < 1.$ 
		
		For each $N \in \mathbb{N}$ we define the function
		\begin{equation}\label{bv-10}
			U_N^{(1)}(\zeta) := \frac{\Pi_{N+1}(\zeta_1 -1, \zeta_2)}{\max_{\zeta_2\in [-1,1]}\Pi_{N+1}(-1,\zeta_2)}, \quad \zeta \in \overline{\mathfrak{C}_{0,N}}.
		\end{equation}
		Obviously, that the denominator is positive, not grater than $\gamma,$  and converges to $\gamma$ as
		$N\to\infty$. 
		By construction $U_N^{(1)}$ solves $\mathcal L_\zeta U_N^{(1)}=0$ in $\mathfrak C_{0,N}$, satisfies the homogeneous conormal condition on the lateral boundary, vanishes on $\Gamma_{-N}$ and $U_N^{(1)}|_{\Gamma_0}\le 1.$
		
		Comparing  $U_N^{(1)}$ with the solution $\Pi_N$ and using the maximum principle, we get $U_N^{(1)} \le \Pi_N$ in $\overline{\mathfrak{C}_{0,N}}.$ Then it follows from \eqref{bv-10} that $\Pi_{N+1}(-2, \zeta_2) \le \Pi_N(-1, \zeta_2) \, \gamma .$ Letting $N\to\infty$ and using uniform convergence of $\Pi_{N+1}$ to $\Pi$ on the fixed slices,
		we obtain 
		$
		a_2 \le a_1 \, \gamma = \gamma^2.
		$ 
		
		Iterating the same construction, we define
		\[
		U_N^{(k)}(\zeta):=\frac{\Pi_{N+k}(\zeta_1-1,\zeta_2)}
		{\max_{\zeta_2\in [-1,1]}\Pi_{N+k}(-1,\zeta_2)},
		\]
		compare $U_N^{(k)}$ with $\Pi_{N+k-1}$ by the maximum principle on the appropriate truncated domain, pass to the maximum on slices, and then let $N\to\infty$. The inductive step yields
		\[
		a_{m+1} \le a_{m} \, \gamma \le\gamma^{m+1}
		\quad\text{for every }m\in\mathbb N.
		\]
		
		By the maximum principle together with the Hopf–Oleinik lemma, the supremum of $\Pi$ on $\mathfrak{C}_{m, m+1}$ is attained on the horizontal slice $\Gamma_{m},$ hence
		\begin{equation}\label{bv-11}
			\sup_{\mathfrak{C}_{m, m+1}} \Pi \le a_m \le \gamma^m.	
		\end{equation}
		
		Set $\delta_0 :=-\ln\gamma>0.$ Then for any $\zeta \in \overline{\mathfrak{C}}$ with $\zeta_1\le - 2$ we choose $m=\lfloor -\zeta_1\rfloor\ge 2$ so that $\zeta\in \overline{\mathfrak{C}_{m, m+1}}$. Then
		\[
		\Pi(\zeta)\le \sup_{\mathfrak{C}_{m, m+1}}\Pi \le \gamma^m \le C \,  e^{\delta_0 \zeta_1}, \quad \text{where} \ \ C :=\gamma^{-1} .
		\]

		\smallskip
		\noindent\textbf{5.} Now we prove that the first and second derivatives of $\Pi$
		decay also exponentially as $\zeta_1\to-\infty$. Fix an integer $m\ge2.$ Due to $1$-periodicity structure of $\mathfrak{C}$ and  the $C^{2,\alpha}$-smoothness of the boundaries $\partial\mathfrak{C}\setminus \Gamma_0$,  the closure $\overline{\mathfrak{C}_{m, m+1}}$
		can be covered by a finite family (independent of $m$) of interior balls $B_r$ and boundary half‑balls $B_r^+$
		of a fixed radius $r>0$. 
		
		Standard interior Schauder estimates for the homogeneous equation
		$\mathcal L_\zeta\Pi=0$ provide positive constants $C_{\mathrm{int}}$ and $C_{\mathrm{int}}'$ (depending only
		on ellipticity and the $C^\alpha$ norms of the coefficients) such that for every interior ball $B_r$
		\[
		\|\nabla\Pi\|_{L^\infty(B_{r/2})} \le C_{\mathrm{int}}\,r^{-1}\|\Pi\|_{L^\infty(B_r)},
		\qquad
		\|D^2\Pi\|_{L^\infty(B_{r/2})} \le C_{\mathrm{int}}'\,r^{-2}\|\Pi\|_{L^\infty(B_r)}.
		\]
		Boundary Schauder estimates  give analogous bounds
		on each boundary half‑ball $B_r^+$ with constants $C_{\mathrm{bd}},C_{\mathrm{bd}}'$.
		Combining the finitely many local interior and boundary Schauder estimates on the cover of
		\(\mathfrak C_{m,m+1}\) by balls \(B_{r/2}\) and half‑balls \(B_{r/2}^+\) yields uniform constants  $C'_1,C'_2>0$ (independent of $m$) such that
		\[
		\sup_{\mathfrak{C}_{m, m+1}}|\nabla\Pi| \le \ C'_1 \sup_{\mathfrak{C}_{m-1, m+2}}\Pi \ \stackrel{\eqref{bv-11}}{\le}  \ C'_1\, \gamma^{m-1} ,
		\qquad
		\sup_{\mathfrak{C}_{m, m+1}}|D^2\Pi| \le \ C'_2 \sup_{\mathfrak{C}_{m-1, m+2}}\Pi \ \stackrel{\eqref{bv-11}}{\le}  \ C'_2 \gamma^{m-1}.
		\]
		
		Then, arguing similarly to item 4, we conclude that there exist positive constants
		$C_1,C_2,\delta_0$ such that
		\[
		|\nabla\Pi(\zeta)| \le C_1 e^{\delta_0 \zeta_1},\qquad
		|D^2\Pi(\zeta)| \le C_2 e^{\delta_0 \zeta_1} \qquad\text{for }\zeta_1\le -2.
		\]
	\end{proof}
	
	\begin{remark}
		From the proof, it follows that the conclusion of Lemma \ref{lemma-4-5} extends to the inhomogeneous problem \(\mathcal L_\zeta\Pi=f\), provided the right‑hand side satisfies
		\(f(\zeta)=\mathcal O\big(e^{\delta_f\zeta_1}\big)\) as \(\zeta_1\to-\infty\) for some \(\delta_f>0\),
		$f\in C^{1}(\overline{\mathfrak C}),$ and  $\partial_{\zeta_2}f =0$ at $\zeta_2 = \pm 1$ and $\zeta_1 \in (-\delta_1,0).$
	\end{remark}
	
	Based on this remark, we define the next term $\Pi_{1,k}$ in the boundary-layer asymptotics \eqref{prim+} as a solution to 
	the problem
	\begin{equation}\label{prim+probl+1}
		\left\{\begin{array}{rcll}
			\mathcal{L}_{\zeta}\big(\Pi_{1,k}(\zeta,t)\big)
			& =    & - \partial_t \Pi_{0,k}(\zeta,t),
			& \quad \zeta\in \mathfrak{C},
			\\[2mm]
			-\big(\mathbb{D}(\zeta)\nabla_{\zeta}\Pi_{1,k}(\zeta,t)\big) \cdot \boldsymbol{\nu}(\zeta) & =
			& 0,
			& \quad \zeta\in \partial\mathfrak{C} \setminus \Gamma_\ell,
			\\[2mm]
			\Pi_{1, k}(0, {\zeta}_2,t) & =
			& 0,
			& \quad {\zeta}_2\in \Gamma_\ell,
			\\[2mm]
			\Pi_{1, k}(\zeta,t) & \to
			& 0,
			& \quad \zeta_1\to -\infty,
		\end{array}\right.
	\end{equation}
	possessing the same properties as those established in Lemma~\ref{lemma-4-5} for $\Pi_{0,k}.$
	
	
	\section{Asymptotic Approximation and Asymptotic Estimates }\label{Sect-5}
	
	Using the solution $(\mathbf{u}^+_0, \mathbf{w}^f_0, \mathbf{u}^-_0)$ to the homogenized problem \eqref{homo-problem}, the boundary-layer solutions $Z^\pm_1,$ $Z^\pm_2$ (see \S \ref{Par-4-2}),  the solutions $\{N^{(k)}_1\}_{k=1}^{\mathcal{M}}$ of problems \eqref{problem-N-1},  and the solutions $\{\Pi_{0, k}\}_{k=1}^{\mathcal{M}}$ and $\{\Pi_{1, k}\}_{k=1}^{\mathcal{M}}$ of problems  \eqref{prim+probl+0} and \eqref{prim+probl+1}, respectively,  we construct the following approximation vector-function
	\begin{equation}\label{app-functions}
		\mathbf{R}_\varepsilon =
		\left\{\begin{array}{ll}
			\mathbf{R}_\varepsilon^\pm := \mathbf{u}^\pm_0(x,t) + \varepsilon \, \chi_0(x_2) \, \boldsymbol{\mathcal{N}}^\pm\big(\tfrac{x}{\varepsilon},x_1,t\big), & \quad (x,t) \in 
			\Omega^{\pm, T}_\varepsilon,
			\\[4pt]
			\mathbf{R}_\varepsilon^f := \mathbf{w}^f_0(x_1,t) + \chi_{1}\big(\frac{x_1 -\ell}{\varepsilon^\gamma}\big) \, \boldsymbol{\Pi}^f\big(\tfrac{x_1 - \ell}{\varepsilon}, \tfrac{x_2}{\varepsilon}, t\big) + \varepsilon \, \boldsymbol{\mathcal{N}}^f\big(\tfrac{x}{\varepsilon},x_1,t\big),& \quad (x,t) \in 
			\Omega^{f, T}_\varepsilon.
		\end{array}\right.
	\end{equation}
	Here $\chi_i \in C^\infty(\mathbb{R})$ ($i \in \{0,1\}$) satisfies $0 \leq \chi_i \leq 1$, with $\chi_i(s) = 1$ for $|s| \leq \theta_i/2$ and $\chi_i(s) = 0$ for $|s| \geq \theta_i$, where $\theta_0$ and $\theta_1$ are small fixed positive numbers; the parameter $\theta_1$ is chosen so that $\chi_1$ vanishes on the support of $\boldsymbol{\Phi}^\pm$ and $\boldsymbol{\Psi}^\pm$;  $\gamma$ is a fixed number from $(\tfrac{2}{3}, 1)$;
	\[
	\boldsymbol{\mathcal{N}}^\pm\!\left(\tfrac{x}{\varepsilon},x_1,t\right) :=
	Z^\pm_1\!\left(\tfrac{x_1}{\varepsilon}, \tfrac{x_2 \mp \varepsilon}{\varepsilon}\right)\, \partial_{x_1}\mathbf{u}^\pm_0(x_1,0,t)
	+ \left(Z^\pm_2\!\left(\tfrac{x_1}{\varepsilon}, \tfrac{x_2 \mp \varepsilon}{\varepsilon}\right) - \tfrac{x_2 \mp \varepsilon}{\varepsilon}\right)\, \partial_{x_2}\mathbf{u}^\pm_0(x_1,0,t);
	\]
	the $k$-th component of $\boldsymbol{\mathcal{N}}^f$ is
	$N^{(k)}_1\big(\tfrac{x}{\varepsilon}\big) \, \partial_{x_1}w_{0,k}(x_1,t),$ and $\big(\boldsymbol{\Pi}^f\big)_k = \Pi_{0, k} + \varepsilon\, \Pi_{1, k}.$

	Substituting $\mathbf{R}_\varepsilon$ into problem \eqref{probl} instead of the solution, we 
	calculate the resulting discrepancies. Obviously, $\mathbf{R}_\varepsilon\big|_{t=0}= \mathbf{0}$ in $\Omega_{\varepsilon},$ and, 
	due to Remark~\ref{rem-4-4}, $\mathbf{R}_\varepsilon^\pm  = \mathbf{0}$ on $\partial\Omega^{\pm, T}_\varepsilon \setminus S^{\pm, T}_\varepsilon.$

	\textit{Discrepancies in the differential equations in} $\Omega^{+, T}_\varepsilon.$ In $\Omega^{-, T}_\varepsilon$ the discrepancies are calculated in a similar way.  Direct computation yields the following expression:
	\begin{align}\label{res-1}
		\partial_t\mathbf{R}_\varepsilon^+ & - \mathbb{D}^+ \Delta \mathbf{R}_\varepsilon^+   - \boldsymbol{\mathfrak{F}}^\pm  =  \mathbf{F}^+(\mathbf{u}_0^+, x, t)  + \varepsilon\, \chi_0(x_2) \, \partial_t\boldsymbol{\mathcal{N}}^+(\xi,x_1,t) \notag
		\\
		&  -   \varepsilon \chi_0(x_2) \,  \partial_{x_1} \big(\big(\partial_{x_1}\boldsymbol{\mathcal{N}}^+(\xi,x_1,t)\bigr)\big|_{\xi_1=\frac{x_1}{\varepsilon}, \, \xi_2=\frac{x_2 - \varepsilon}{\varepsilon}}\big) 
		\ - \ \varepsilon \,\partial_{x_2}\bigl( \chi^\prime_0(x_2)\, \boldsymbol{\mathcal{N}}^+(\tfrac{x}{\varepsilon}, x_1, t) \bigr) \notag
		\\
		&
		- \chi^\prime_0(x_2)\, \partial_{\xi_2}\boldsymbol{\mathcal{N}}^+(\xi,x_1,t)\big|_{\xi_1=\frac{x_1}{\varepsilon}, \, \xi_2=\frac{x_2 - \varepsilon}{\varepsilon}}
		-   \chi_0(x_2) \, \partial^2_{x_1\xi_1}\boldsymbol{\mathcal{N}}^+(\xi,x_1,t)\big|_{\xi_1=\frac{x_1}{\varepsilon}, \, \xi_2=\frac{x_2 - \varepsilon}{\varepsilon}}.
	\end{align}
	
	\textit{Discrepancies on the oscillating boundaries} $S^{+, T}_\varepsilon.$ Observing that
	$$
	(\mathbb{D}^+\nabla\mathbf{R}_\varepsilon^+)\cdot\boldsymbol{\nu}_\varepsilon = 
	\big(D_1^+\,\nabla R^+_{\varepsilon, 1}\cdot\boldsymbol{\nu}_\varepsilon,\dots,D_{\mathcal{M}}^+\,\nabla R^+_{\varepsilon,\mathcal{M}}\cdot\boldsymbol{\nu}_\varepsilon\big)^\top,
	$$
	taking into account the Neumann conditions for $Z^\pm_1,$ $Z^\pm_2$ (see \S \ref{Par-4-2}) and recalling that $\boldsymbol{\nu}_\varepsilon$ is the outward unit normal to the boundary of $\Omega^f_\varepsilon,$ we find 
	\begin{multline}\label{res-2}
		(\mathbb{D}^+\nabla\mathbf{R}_\varepsilon^+)\cdot\boldsymbol{\nu}_\varepsilon  = (\mathbb{D}^+(\nabla\mathbf{u}_0^+(x,t) -\nabla\mathbf{u}_0^+(x_1,0,t))\cdot\boldsymbol{\nu}_\varepsilon 
		+  \widetilde{\boldsymbol{\Upsilon}}^+\big(\mathbf{u}^+_0(x_1,0,t), \mathbf{w}^f_0(x_1,t), x_1, t \big)
		\\ +
		\varepsilon\, \big(\big(\partial_{x_1}\boldsymbol{\mathcal{N}}^+(\xi,x_1,t)\bigr)\big|_{\xi_1=\frac{x_1}{\varepsilon}, \, \xi_2=\frac{x_2 - \varepsilon}{\varepsilon}}\big) \, \nu_1(\tfrac{x}{\varepsilon})
		\quad \text{on} \ \
		S^{+, T}_\varepsilon, 
	\end{multline}
	where the components of the vector-function $\widetilde{\boldsymbol{\Upsilon}}^+$ are defined by formula \eqref{eq-22}.

	\textit{Discrepancies in the differential equations in fracture domain} $\Omega^{f}_\varepsilon.$ We calculate discrepancies for each component of 
	$\mathbf{R}_\varepsilon^f,$ first for 
	$$
	\widetilde{\mathcal{U}}^{(k)}_\varepsilon(x,t) :=  w_{0,k}(x_1,t) + \varepsilon\, N^{(k)}_1\big(\tfrac{x}{\varepsilon}\big) \, \partial_{x_1}w_{0,k}(x_1,t),
	$$
	and then for 
	$$
	\widetilde{\mathcal{P}}^{(k)}_\varepsilon(x,t) := \chi_{1}\big(\tfrac{x_1 -\ell}{\varepsilon^\gamma}\big) \, \Big(\Pi_{0,k}\big(\tfrac{x_1 - \ell}{\varepsilon}, \tfrac{x_2}{\varepsilon}, t\big) + \varepsilon\, \Pi_{1,k}\big(\tfrac{x_1 - \ell}{\varepsilon}, \tfrac{x_2}{\varepsilon}, t\big)\Big).
	$$ 
	Using \eqref{eq1}, \eqref{hyperbolic-system}, \eqref{problem-N-1} and \eqref{problem-N-2}, we obtain 
	\begin{multline}\label{eq-5-4}
		\partial_t\widetilde{\mathcal{U}}^{(k)}_\varepsilon -  \varepsilon\, \nabla_x \cdot\big( \mathbb{D}(\tfrac{x}{\varepsilon}) \nabla_x \widetilde{\mathcal{U}}^{(k)}_\varepsilon \big) +
		\overrightarrow{V_\varepsilon}(x) \cdot \nabla_x \widetilde{\mathcal{U}}^{(k)}_\varepsilon = \widehat{F}_k(\mathbf{u}^+_0, \mathbf{u}^-_0, \mathbf{w}^f_0, x_1, t) 
		+ \varepsilon\, N^{(k)}_1\, \partial^2_{t x_1}w_{0,k}
		\\
		- \varepsilon \Big(\widehat{d}^{\,k}_{11} + \overrightarrow{V} \cdot \nabla_{\xi}N^{(k)}_{2}\Big)\Big|_{\xi=\frac{x}{\varepsilon}}\,  \partial^2_{x^2_1}w_{0,k} +  
		\varepsilon^2\Big(d_{11}(\xi)\, N^{(k)}_{1}(\xi) + \sum_{j=1}^{2}
		d_{1j}(\xi)\, \partial_{\xi_j}N_2^{(k)}(\xi)\Big)\Big|_{\xi=\frac{x}{\varepsilon}} \partial^3_{x^3_1}w_{0,k}
		\\
		+ \varepsilon^2 \sum_{i=1}^{2} \partial_{x_i}\Big(\sum_{j=1}^{2}
		d_{ij}(\xi)\, \partial_{\xi_j}N_2^{(k)}(\xi)\big|_{\xi=\frac{x}{\varepsilon}} \, \partial^2_{x^2_1}w_{0,k} \Big) .
	\end{multline}
	
	Taking into account that $\Pi_{0,k}$ is a solution to problem \eqref{prim+probl+0}, we find 
	\begin{multline*}
		\partial_t \widetilde{\mathcal{P}}^{(k)}_\varepsilon
		-  \varepsilon\, \nabla_x \cdot\big( \mathbb{D}(\tfrac{x}{\varepsilon}) \nabla_x \widetilde{\mathcal{P}}^{(k)}_\varepsilon \big) +
		\overrightarrow{V_\varepsilon}(x) \cdot \nabla_x \widetilde{\mathcal{P}}^{(k)}_\varepsilon  =  \varepsilon\, \chi_{1}\big(\tfrac{x_1 -\ell}{\varepsilon^\gamma}\big)\, \partial_t\Pi_{1,k}\big(\tfrac{x_1- \ell}{\varepsilon}, \tfrac{x_2}{\varepsilon}, t\big) 
		\\
		- \varepsilon^{1-2\gamma}\,  \chi^{\prime\prime}_{1}\big(\tfrac{x_1 -\ell}{\varepsilon^\gamma}\big) \, d_{11} \, \Pi_{0,k}  
		- \varepsilon^{-\gamma} \bigg(\sum_{i=1}^{2} \big(d_{1 i}\, \partial_{\xi_i}\Pi_{0,k} +
		\partial_{\xi_i}(d_{1 i} \, \Pi_{0,k})\big) - v_{1}(\xi)\, \Pi_{0,k}\bigg)\Big|_{\xi_1=\frac{x_1 - \ell}{\varepsilon}, \ \xi_2=\frac{x_2}{\varepsilon}}\, \chi^{\prime}_{1}\big(\tfrac{x_1 -\ell}{\varepsilon^\gamma}\big)
		\\
		- \varepsilon^{2-2\gamma}\,  \chi^{\prime\prime}_{1}\big(\tfrac{x_1 -\ell}{\varepsilon^\gamma}\big) \, d_{11} \, \Pi_{1,k}  
		- \varepsilon^{1-\gamma} \bigg(\sum_{i=1}^{2} \big(d_{1 i}\, \partial_{\xi_i}\Pi_{1,k} +
		\partial_{\xi_i}(d_{1 i} \, \Pi_{1,k})\big) - v_{1}(\xi)\, \Pi_{1,k}\bigg)\Big|_{\xi_1=\frac{x_1 - \ell}{\varepsilon}, \ \xi_2=\frac{x_2}{\varepsilon}}\, \chi^{\prime}_{1}\big(\tfrac{x_1 -\ell}{\varepsilon^\gamma}\big).
	\end{multline*}

	\textit{Discrepancies on the oscillating and perforated boundaries.} Considering the boundary conditions for $N^{(k)}_1$ and $N^{(k)}_2$(see \eqref{problem-N-1} and \eqref{problem-N-2}), we get
	\begin{equation}\label{eq-5-6}
		-  \mathbb{D}_k(\tfrac{x}{\varepsilon})\nabla_x \widetilde{\mathcal{U}}^{(k)}_\varepsilon  \cdot \boldsymbol{\nu}_\varepsilon  =  \varepsilon \, \mathcal{B}^k_\xi(N^{(k)}_2)\Big|_{\xi=\frac{x}{\varepsilon}} \partial^2_{x_1^2}w_{0,k}  \quad \text{on} \ \ S^{\pm, T}_\varepsilon \cup
		G^{T}_\varepsilon .
	\end{equation}
	
	Now we note that 
	\begin{equation*}
		\widetilde{\mathcal{U}}^{(k)}_\varepsilon\Big|_{\Gamma^T_{0,\varepsilon}} = \varepsilon\, N^{(k)}_1\big(0,\tfrac{x_2}{\varepsilon}\big) \, \partial_{x_1}w_{0,k}(0,t) \quad \text{and} \quad  \Big(\widetilde{\mathcal{U}}^{(k)}_\varepsilon + \widetilde{\mathcal{P}}^{(k)}_\varepsilon \Big)\Big|_{\Gamma^T_{\ell,\varepsilon}} = q^\ell_k(t) + \varepsilon\, N^{(k)}_1\big(0,\tfrac{x_2}{\varepsilon}\big) \, \partial_{x_1}w_{0,k}(\ell,t).
	\end{equation*}
	To neutralize the residuals of order $\mathcal{O}(\varepsilon),$ we introduce the following  cutoff function $\chi_\varepsilon^{(0)} \in C^\infty([0,+\infty))$ defined by
	$$
	0\le \chi_\varepsilon^{(0)}(x_1) \le 1, \quad \chi_\varepsilon^{(0)}(x_1) = 1 \ \ \text{for} \ \ x_1\in [0, \varepsilon  \tfrac{\theta_2}{2}], \quad \text{and}  \quad \chi_\varepsilon^{(0)}(x_1) = 0 \ \ \text{for} \ \ x_1 \ge \varepsilon \, \theta_2 .
	$$
	Clearly, $|(\chi_\varepsilon^{(0)})^\prime| \le C_0 \varepsilon^{-1}.$ We also define a new cutoff function $\chi_\varepsilon^{(\ell)}(x_1) := \chi_\varepsilon^{(0)}(\ell - x_1).$ The constant $\theta_2$ is chosen such that 
	$$
	Y_0 \cap \big((0, \theta_2) \times (-1, 1)\big) = (0, \theta_2) \times (-1, 1)\quad \text{ and} \quad Y_0 \cap \big((1- \theta_2, 1) \times (-1, 1)\big) = (1- \theta_2, 1) \times (-1, 1),
	$$ 
	which is possible due to the assumptions on the structure of $Y_0$ (see Sect.~\ref{Sec-2}). 
	Then for each $k\in \{1,\ldots,\mathcal{M}\},$ the function
	\begin{equation}\label{function-L}
		\lambda_\varepsilon^{(k)}(x,t) := \varepsilon \,\big(\chi_\varepsilon^{(0)}(x_1) + \chi_\varepsilon^{(\ell)}(x_1)\big)\, N^{(k)}_1\big(\tfrac{x}{\varepsilon}\big) \, \partial_{x_1}w_{0,k}(x_1,t), \quad (x,t) \in \Omega_\varepsilon^{f, T},
	\end{equation}
	belongs to $L^2\big(0,T; H^1(\Omega_\varepsilon^f)\big).$ Thanks to the boundedness of $N^{(k)}_1(\xi),$  $\nabla_{\xi}N^{(k)}_1(\xi),$
	and $\partial^2_{x^2_1}w_{0,k},$ it is easy to show that
	\begin{equation}\label{function-L-bound}
		\|\lambda_\varepsilon^{(k)}\|_{L^2(\Omega_\varepsilon^{f, T}))} \le C_1 \, \varepsilon^2 \quad \text{and} \quad 
		\|\lambda_\varepsilon^{(k)}\|_{L^2(0,T; H^1(\Omega_\varepsilon^f))} \le C_2 \, \varepsilon .
	\end{equation}
	
	Using these functions, we adjust the approximation in $\Omega^{f,T}_\varepsilon,$ namely $\widetilde{\mathbf{R}}_\varepsilon^f :=  \mathbf{R}_\varepsilon^f - \mathbf{\Lambda}^f_\varepsilon,$ where the $k$-th component of $\mathbf{\Lambda}^f_\varepsilon$ is  
	$\lambda_\varepsilon^{(k)}.$ Thus, 
	\begin{equation}
		\widetilde{\mathbf{R}}_\varepsilon^f\Big|_{\Gamma^T_{0,\varepsilon}} = 0 \quad \text{and} \quad                                 \widetilde{\mathbf{R}}_\varepsilon^f\Big|_{\Gamma^T_{\ell,\varepsilon}} = \mathbf{q}^\ell(t).	
	\end{equation}
	
	\begin{remark}
		Hereinafter, constants in all inequalities are positive and independent of the parameter $\varepsilon.$ In general, constants with the same indices in different inequalities are different.
	\end{remark}
	
	\begin{theorem}\label{main theorem}
		Assume that all hypotheses stated in Section~\ref{Sec-2} are satisfied. Then there exist positive constants 
		$\tilde{C}_1$, $\tilde{C}_2$, and $\varepsilon_0$ such that, for all 
		$\varepsilon \in (0,\varepsilon_0)$, the following estimates hold: 
		\begin{equation}\label{main-est-1}
			\max_{t\in [0,T]} \|\mathbf{u}^\pm_\varepsilon(\cdot,t) - \mathbf{R}^\pm_\varepsilon(\cdot,t)\|_{L^2(\Omega_\varepsilon^{\pm})} \, + \, 
			\||\nabla \mathbf{u}^\pm_\varepsilon - \nabla\mathbf{R}^\pm_\varepsilon|\|_{L^2(\Omega_\varepsilon^{\pm,T})} \le \tilde{C}_1 \, \varepsilon^{1/2},
		\end{equation}	
		\begin{equation}\label{main-est-2}
			\max_{t\in [0,T]} \|\mathbf{u}^f_\varepsilon(\cdot,t) - \mathbf{R}^f_\varepsilon(\cdot,t)\|_{L^2(\Omega_\varepsilon^{f})} \, + \, 
			\sqrt{\varepsilon} \, \||\nabla \mathbf{u}^f_\varepsilon - \nabla\mathbf{R}^f_\varepsilon|\|_{L^2(\Omega_\varepsilon^{f,T})} \le \tilde{C}_2 \, \varepsilon,
		\end{equation}
		where $\mathbf{u}_\varepsilon = \big(\mathbf{u}^+_\varepsilon, \mathbf{u}^+_\varepsilon, \mathbf{u}^f_\varepsilon \big)$ is the weak solution to problem \eqref{probl}, and  $\mathbf{R}^\pm_\varepsilon$,  $\mathbf{R}^f_\varepsilon$ are the approximation functions defined in 
		\eqref{app-functions}.
	\end{theorem}
	
	\begin{proof}
		\textbf{1.} \emph{Estimates in $\Omega_\varepsilon^{\pm, T}$.}
		We present the proof only in $\Omega_\varepsilon^{+, T}$; the argument in $\Omega_\varepsilon^{-,T}$ is identical.  Subtracting from \eqref{res-1} the corresponding differential equations of problem \eqref{probl},  multiplying the resulting identity by $\mathbf{W}^+_\varepsilon := \mathbf{R}^+_\varepsilon - \mathbf{u}^+_\varepsilon,$ and 
		integrating by parts, we get
		\begin{multline}\label{est-5-10}
			\int_{\Omega_\varepsilon^+} \big|\mathbf{W}^+_\varepsilon(x,\tau)\big|^2 \, dx\  + \int_{\Omega_\varepsilon^{+, \tau}} \big|\nabla \mathbf{W}^+_\varepsilon\big|^2 \, dxdt \le C_1\Bigg(\int_{\Omega_\varepsilon^{+, \tau}} |\mathbf{F}^+(\mathbf{u}_0^+, x, t) - \mathbf{F}^+(\mathbf{u}_\varepsilon^+, x, t)| \,
			|\mathbf{W}^+_\varepsilon|\, dxdt 
			\\
			+ \varepsilon \int_{\Omega_\varepsilon^{+, \tau}} |\chi^\prime_0(x_2)|\, |\boldsymbol{\mathcal{N}}^+(\tfrac{x}{\varepsilon}, x_1, t)|\,
			|\partial_{x_2}\mathbf{W}^+_\varepsilon| \, dxdt
			+ \varepsilon \int_{\Omega_\varepsilon^{+, \tau}} \chi_0(x_2) |\partial_t\boldsymbol{\mathcal{N}}^+(\tfrac{x}{\varepsilon},x_1,t)|
			\, |\mathbf{W}^+_\varepsilon|\, dxdt 
			\\
			+
			\int_{\Omega_\varepsilon^{+, \tau}} |\chi^\prime_0(x_2)| \, \big| \partial_{\xi_2}\boldsymbol{\mathcal{N}}^+(\xi,x_1,t)\big|_{\xi_1=\frac{x_1}{\varepsilon}, \, \xi_2=\frac{x_2 - \varepsilon}{\varepsilon}}\big|
			\, |\mathbf{W}^+_\varepsilon|\, dxdt
			\\
			+ \int_{\Omega_\varepsilon^{+, \tau}} \chi_0(x_2) \, \big|\partial^2_{x_1\xi_1}\boldsymbol{\mathcal{N}}^+(\xi,x_1,t)\big|_{\xi_1=\frac{x_1}{\varepsilon}, \, \xi_2=\frac{x_2 - \varepsilon}{\varepsilon}}\big| \, |\mathbf{W}^+_\varepsilon|\, dxdt
			\\
			+	\int_{S_\varepsilon^{+, \tau}}|\nabla\mathbf{u}_0^+(x,t) -\nabla\mathbf{u}_0^+(x_1,0,t)| \, |\mathbf{W}^+_\varepsilon| \, dl_xdt 
			\\	
			+ \bigg|\int_{S_\varepsilon^{+, \tau}} \big(\widetilde{\boldsymbol{\Upsilon}}^+(\mathbf{u}^+_0(x_1,0,t), \mathbf{w}^f_0(x_1,t), x_1, t) -  \boldsymbol{\Upsilon}^+(\mathbf{u}_\varepsilon^+, \mathbf{u}_\varepsilon^f, \tfrac{x}{\varepsilon}, x_1, t)\big) \cdot \mathbf{W}^+_\varepsilon \, dl_xdt\bigg| \Bigg).
		\end{multline}
		Using Cauchy’s inequality with any $\delta > 0$ $(ab \le \delta a^2 + b^2/4\delta),$ assumptions $\mathbf{A3},$ the properties of the solutions
		$Z^\pm_1,$ $Z^\pm_2$ (see \S \ref{Par-4-2}), and the solution to the homogenized problem (see Lemma~\ref{Th-4-4}),  we deduce from \eqref{est-5-10} the inequality    
		\begin{equation}\label{main-omega+}
			\int\limits_{\Omega_\varepsilon^+} \big|\mathbf{W}^+_\varepsilon(x,\tau)\big|^2  dx\,  + \int\limits_{\Omega_\varepsilon^{+, \tau}} \big|\nabla \mathbf{W}^+_\varepsilon\big|^2  dxdt \, \le \, C_1\int\limits_{\Omega_\varepsilon^{+, \tau}} \big|\mathbf{W}^+_\varepsilon\big|^2  dxdt \,
			+ C_2 \int\limits_{S_\varepsilon^{+, \tau}} \big|\mathbf{R}^f_\varepsilon - \mathbf{u}^f_\varepsilon\big|^2  dl_xdt \, + C_3 \, \varepsilon. 
		\end{equation}
		
		We now illustrate how to estimate several integrals in \eqref{est-5-10}. Since  $\partial_{\xi_2}\boldsymbol{\mathcal{N}}^+(\xi,x_1,t)$ and $\partial^2_{x_1\xi_1}\boldsymbol{\mathcal{N}}^+(\xi,x_1,t)$  vanish exponentially as $\xi_2 \to +\infty,$ we have  
		$$
		\int\limits_{\Omega_\varepsilon^{+, \tau}} |\chi^\prime_0(x_2)| \, \big| \partial_{\xi_2}\boldsymbol{\mathcal{N}}^+(\xi,x_1,t)\big|_{\xi_1=\frac{x_1}{\varepsilon}, \, \xi_2=\frac{x_2 - \varepsilon}{\varepsilon}}\big|
		\, |\mathbf{R}^+_\varepsilon - \mathbf{u}^+_\varepsilon|\, dxdt \le 
		\int\limits_{\Omega_\varepsilon^{+, \tau}} |\mathbf{R}^+_\varepsilon - \mathbf{u}^+_\varepsilon|^2\, dxdt + \mathcal{O}\big(\exp(- \tfrac{\delta_0\, \theta_0}{\varepsilon})\big). 
		$$
		By Lemma~3.1 of \cite{Mel-Naz-1997}, for any $\rho\in(0,1)$ and $\delta>0$, 
		\begin{multline*}
			\int_{\Omega_\varepsilon^{+, \tau}} \chi_0(x_2) \, \big|\partial^2_{x_1\xi_1}\boldsymbol{\mathcal{N}}^+(\xi,x_1,t)\big|_{\xi_1=\frac{x_1}{\varepsilon}, \, \xi_2=\frac{x_2 - \varepsilon}{\varepsilon}}\big| \, |\mathbf{W}^+_\varepsilon|\, dxdt \, \le \, C(\rho) \varepsilon^{1-\rho}\|\nabla \mathbf{W}^+_\varepsilon\|_{L^2(\Omega_\varepsilon^{+, \tau})}
			\\
			\le \delta \, \int_{\Omega_\varepsilon^{+, \tau}} \big|\nabla \mathbf{W}^+_\varepsilon\big|^2 \, dxdt + C(\rho,\delta)\, \varepsilon^{2(1-\rho)}.
		\end{multline*}
		Choosing $\delta$ sufficiently small, the gradient term is absorbed into the left-hand side of \eqref{est-5-10}.
		
		The last integral in \eqref{est-5-10} is bounded above by the sum
		$$
		\bigg|\int_{S_\varepsilon^{+, \tau}} \big(\widetilde{\boldsymbol{\Upsilon}}^+(\mathbf{u}^+_0(x_1,0,t), \mathbf{w}^f_0(x_1,t), x_1, t) -  \boldsymbol{\Upsilon}^+(\mathbf{u}^+_0(x_1,0,t), \mathbf{w}^f_0(x_1,t), \tfrac{x}{\varepsilon}, x_1, t)\big) \cdot \mathbf{W}^+_\varepsilon \, dl_xdt\bigg|
		$$
		$$ 
		+ \int_{S_\varepsilon^{+, \tau}} \big|\boldsymbol{\Upsilon}^+(\mathbf{u}^+_0(x_1,0,t), \mathbf{w}^f_0(x_1,t), \tfrac{x}{\varepsilon}, x_1, t) - \boldsymbol{\Upsilon}^+(\mathbf{R}^+_\varepsilon, \mathbf{R}^f_\varepsilon, \tfrac{x}{\varepsilon}, x_1, t)\big| \, \big|\mathbf{W}^+_\varepsilon\big| \, dl_xdt
		$$
		$$
		+ \int_{S_\varepsilon^{+, \tau}} \big|\boldsymbol{\Upsilon}^+(\mathbf{R}^+_\varepsilon, \mathbf{R}^f_\varepsilon, \tfrac{x}{\varepsilon}, x_1, t) - \boldsymbol{\Upsilon}^+(\mathbf{u}_\varepsilon^+, \mathbf{u}_\varepsilon^f, \tfrac{x}{\varepsilon}, x_1, t)\big| \, \big|\mathbf{W}^+_\varepsilon\big| \, dl_xdt =: I_1(\varepsilon) + I_2(\varepsilon) + I_3(\varepsilon) .
		$$
		Taking into account that
		$$
		\int_{S^{+}} \big(\widetilde{\boldsymbol{\Upsilon}}^+(\mathbf{u}^+_0(x_1,0,t), \mathbf{w}^f_0(x_1,t), x_1, t) -  \boldsymbol{\Upsilon}^+(\mathbf{u}^+_0(x_1,0,t), \mathbf{w}^f_0(x_1,t), \tfrac{x}{\varepsilon}, x_1, t)\big)  \, dl_\xi = \mathbf{0}
		$$
		and following the proof of inequality \eqref{eq-15}, we derive
		$$
		I_1(\varepsilon) \le C_1 \varepsilon^\frac{1}{2} \|\nabla\mathbf{W}^+_\varepsilon\|_{L^2(\Omega_\varepsilon^{+, \tau})} \,\le \, \delta \, \int_{\Omega_\varepsilon^{+, \tau}} \big|\nabla\mathbf{W}^+_\varepsilon\big|^2 \, dxdt + C(\delta)\, \varepsilon .
		$$
		Thanks to \eqref{Lip-ineq}, we have that $I_2(\varepsilon) \le C_2 \varepsilon \int_{S_\varepsilon^{+, \tau}} |\mathbf{W}^+_\varepsilon| \, dl_xdt \le C_3 \varepsilon^2 + \int_{S_\varepsilon^{+, \tau}} |\mathbf{W}^+_\varepsilon|^2 \, dl_xdt.$ Then applying Cauchy's inequality for the trace operator yields 
		$$
		I_2(\varepsilon) \le 
		C_3 \varepsilon^2 + \delta \, \int_{\Omega_\varepsilon^{+, \tau}} \big|\nabla\mathbf{W}^+_\varepsilon\big|^2 \, dxdt + C(\delta)\, 
		\int_{\Omega_\varepsilon^{+, \tau}} |\mathbf{W}^+_\varepsilon|^2 \, dxdt.
		$$ 
		Similarly,
		$$
		I_3(\varepsilon) \le 
		\delta \, \int_{\Omega_\varepsilon^{+, \tau}} \big|\nabla\mathbf{W}^+_\varepsilon\big|^2 \, dxdt + C(\delta)\, 
		\int_{\Omega_\varepsilon^{+, \tau}} |\mathbf{W}^+_\varepsilon|^2 \, dxdt + C_2 \int_{S_\varepsilon^{+, \tau}} |\mathbf{R}^f_\varepsilon - \mathbf{u}^f_\varepsilon|^2 \, dl_xdt.
		$$ 
		\smallskip 
		
		\noindent
		\textbf{2.} \emph{Estimates in $\Omega_\varepsilon^{f, T}$.}
		Denote $\mathbf{W}^f_\varepsilon := \widetilde{\mathbf{R}}^f_\varepsilon - \mathbf{u}^f_\varepsilon.$ Clearly, 
		$\mathbf{W}^f_\varepsilon\big|_{\Gamma^T_{0,\varepsilon}} =  \mathbf{W}^f_\varepsilon \big|_{\Gamma^T_{\ell,\varepsilon}} = 0.$
		
		According to assumption $\mathbf{A2},$ each component of $\mathbf{W}^f_\varepsilon$ satisfies
		\begin{equation}\label{eq-5-12}
			\int_{\Omega_\varepsilon^{f, \tau}} \nabla \cdot \big(\big(\mathbf{W}^f_\varepsilon\big)_k \, \overrightarrow{V_\varepsilon}\big)\,
			\big(\mathbf{W}^f_\varepsilon\big)_k \, dxdt = \frac12 \int_{\Omega_\varepsilon^{f, \tau}} \nabla \cdot \big( \big(\mathbf{W}^f_\varepsilon\big)^2_k \, \overrightarrow{V_\varepsilon}\big) \, dxdt =
			\frac12 \int_{\Omega_\varepsilon^{f, \tau}} \big(\mathbf{W}^f_\varepsilon\big)^2_k \,  \big(\overrightarrow{V_\varepsilon}\cdot \boldsymbol{\nu}_\varepsilon \big) \, dxdt = 0,
		\end{equation} 	
		where the last equality uses \(\overrightarrow{V_\varepsilon}\cdot\boldsymbol{\nu}_\varepsilon=0\) on the relevant part of the spatial boundary together with the homogeneous trace conditions for \(\mathbf{W}^f_\varepsilon\) on $\Gamma^T_{0,\varepsilon}$ and $\Gamma^T_{\ell,\varepsilon}.$

		Using assumptions $\mathbf{A3},$ relations \eqref{eq-5-4} -- \eqref{eq-5-6},  estimates \eqref{function-L-bound}, identity \eqref{eq-5-12}, Lemmas~\ref{Th-4-4} and \ref{lemma-4-5}, the trace inequality
		\begin{equation}\label{trace-1}
			\varepsilon \int_{S^\pm_\varepsilon \cup G_\varepsilon} u^2 \, dl_x \le C\Big(\int_{\Omega_\varepsilon^f} u^2\, dx + \varepsilon^2 \int_{\Omega_\varepsilon^f} |\nabla u|^2\, dx\Big) \quad \forall \, u\in H^1(\Omega_\varepsilon^f)	
		\end{equation}
		proved in \cite{Melnyk-Popov-2009},
		and
		the same calculation technique as above, we obtain for every \(\tau\in(0,T]\) the inequality   
		\begin{multline}\label{eq-5-14}
			\int_{\Omega_\varepsilon^f} \big|\mathbf{W}^f_\varepsilon(x,\tau)\big|^2 \, dx\,  + \varepsilon \int_{\Omega_\varepsilon^{f, \tau}} \big|\nabla \mathbf{W}^f_\varepsilon\big|^2 \, dxdt \, \le \, C_1\int_{\Omega_\varepsilon^{f, \tau}} \big|\mathbf{W}^f_\varepsilon\big|^2  dxdt \, + \, C_2 \varepsilon \int_{S_\varepsilon^{\pm, \tau}} |\mathbf{W}^\pm_\varepsilon|^2 \, dl_xdt \, +\,   C_3 \, \varepsilon^3
			\\	
			+ C_4 \Bigg|\varepsilon \int_{S_\varepsilon^{\pm, \tau}} \mathbf{\Phi}^\pm\big(\mathbf{u}^\pm_0, \mathbf{w}^f_0, \tfrac{x}{\varepsilon}, x_1, t\big) \cdot \mathbf{W}^f_\varepsilon \, dl_xdt \,+\,
			\varepsilon \int_{G_\varepsilon^{\tau}} \mathbf{\Psi}\big(\mathbf{w}^f_0, \tfrac{x}{\varepsilon}, x_1, t\big) \cdot \mathbf{W}^f_\varepsilon \, dl_xdt	
			\\
			+ \int_{\Omega_\varepsilon^{f, \tau}} \widehat{\mathbf{F}}(\mathbf{u}^+_0, \mathbf{u}^-_0, \mathbf{w}^f_0, x_1, t) \cdot \mathbf{W}^f_\varepsilon \, dxdt\Bigg|. 
		\end{multline}
		
		To estimate the absolute value of the sum of the boundary and volume integrals on the right-hand side of \eqref{eq-5-14}, we use the definition of $\widehat{\mathbf{F}}$ (see \eqref{hat F}) and Lemma 2.1 \cite{Melnyk-Popov-2009}. As a result, we find that this sum is bounded by 
		$C_5 \, \varepsilon^{3/2} \, \||\nabla\mathbf{W}^f_\varepsilon|\|_{L^2(\Omega_\varepsilon^{f, \tau})}.$ Applying Cauchy's inequality, we further obtain that it is less than
		$$
		\varepsilon \, \delta \, \||\nabla\mathbf{W}^f_\varepsilon|\|^2_{L^2(\Omega_\varepsilon^{f, \tau})} + C_6(\delta) \, \varepsilon^2 \quad \text{for all} \ \ \delta >0.
		$$
		By choosing the appropriate \(\delta\) and including the gradient term into the left-hand side of \eqref{eq-5-14}, we obtain
		\begin{equation}\label{eq-5-15}  
			\int\limits_{\Omega_\varepsilon^f} \big|\mathbf{W}^f_\varepsilon(x,\tau)\big|^2 \, dx\,  + \varepsilon \int\limits_{\Omega_\varepsilon^{f, \tau}} \big|\nabla \mathbf{W}^f_\varepsilon\big|^2 \, dxdt \, \le \, C_7\int\limits_{\Omega_\varepsilon^{f, \tau}} \big|\mathbf{W}^f_\varepsilon\big|^2  dxdt \, + \,  C_8 \, \varepsilon  \int\limits_{S_\varepsilon^{\pm, \tau}} |\mathbf{W}^\pm_\varepsilon|^2 \, dl_xdt \,+\, C_9 \, \varepsilon^2 .
		\end{equation}
		Inequality \eqref{eq-5-15} is the corrected form of \eqref{eq-5-14}.
		
		\smallskip

		\noindent
		\textbf{3.}
		\emph{Estimates in the whole domain $\Omega^{T}$.} 
		Using \eqref{function-L-bound}, we deduce from \eqref{main-omega+}  the following inequality: 
		\begin{equation}\label{omega+}
			\int\limits_{\Omega_\varepsilon^+} \big|\mathbf{W}^+_\varepsilon(x,\tau)\big|^2  dx\,  + \int\limits_{\Omega_\varepsilon^{+, \tau}} \big|\nabla \mathbf{W}^+_\varepsilon\big|^2  dxdt \, \le \, C_1\int\limits_{\Omega_\varepsilon^{+, \tau}} \big|\mathbf{W}^+_\varepsilon\big|^2  dxdt \,
			+ C_2 \int\limits_{S_\varepsilon^{+, \tau}} \big|\mathbf{W}^f_\varepsilon\big|^2  dl_xdt \, + C_3 \, \varepsilon,
		\end{equation}
		and the same inequality holds for $\mathbf{W}^-_\varepsilon.$
		
		Applying  Gronwall's inequality in time to \eqref{eq-5-15} yields, for every \(\tau\in(0,T]\), 
		\begin{equation}\label{main-1}
			\int\limits_{\Omega_\varepsilon^f} \big|\mathbf{W}^f_\varepsilon(x,\tau)\big|^2 \, dx\,  + \varepsilon \int\limits_{\Omega_\varepsilon^{f, \tau}} \big|\nabla \mathbf{W}^f_\varepsilon\big|^2 \, dxdt \, \le \, C_{10} \, \varepsilon \, \bigg(\int\limits_{S_\varepsilon^{+, \tau}} |\mathbf{W}^+_\varepsilon|^2 \, dl_xdt \,+\, \int\limits_{S_\varepsilon^{-, \tau}} |\mathbf{W}^-_\varepsilon|^2 \, dl_xdt \,+\,  \varepsilon\bigg).
		\end{equation}  
		Combining \eqref{main-1} with the trace estimate \eqref{trace-1} gives
		\begin{equation}\label{main-2}
			\int\limits_{S_\varepsilon^{+, \tau}} \big|\mathbf{W}^f_\varepsilon\big|^2  dl_xdt	\le C_{11} 
			\bigg(\int\limits_{S_\varepsilon^{+, \tau}} |\mathbf{W}^+_\varepsilon|^2 \, dl_xdt \,+\, \int\limits_{S_\varepsilon^{-, \tau}} |\mathbf{W}^-_\varepsilon|^2 \, dl_xdt \,+\,  \varepsilon\bigg),
		\end{equation}
		and the same inequality holds for $\int_{S_\varepsilon^{-, \tau}} \big|\mathbf{W}^f_\varepsilon\big|^2  dl_xdt.$ 
		
		Substituting \eqref{main-2} into \eqref{omega+} and using the trace inequality with an arbitrary small parameter \(\delta>0\) (to absorb gradient terms into the left-hand side) we obtain
		\begin{equation}\label{main-3}
			\int\limits_{\Omega_\varepsilon^{\pm}} \big|\mathbf{W}^{\pm}_\varepsilon(x,\tau)\big|^2  dx\,  + \int\limits_{\Omega_\varepsilon^{\pm, \tau}} \big|\nabla \mathbf{W}^{\pm}_\varepsilon\big|^2  dxdt \, \le \, C_1\int\limits_{\Omega_\varepsilon^{\pm, \tau}} \big|\mathbf{W}^\pm_\varepsilon\big|^2  dxdt \,
			+ C_2 \, \varepsilon,
		\end{equation}
		where the left- and right-hand sides are understood as the sums of the corresponding integrals over \(\Omega_\varepsilon^+\) and \(\Omega_\varepsilon^-\) (and their time cylinders).
		
		By Gronwall's inequality (in time) and the standard energy argument applied to \eqref{main-3} we deduce the uniform estimate
		\begin{equation}\label{main-4}
			\max_{t\in [0,T]} \|\mathbf{W}^{\pm}_\varepsilon(\cdot,t)\|_{L^2(\Omega_\varepsilon^{\pm})} \, + \, 
			\||\nabla\mathbf{W}^{\pm}_\varepsilon|\|_{L^2(\Omega_\varepsilon^{\pm,T})} \le C_1 \, \varepsilon^{1/2}.
		\end{equation}
		
		From \eqref{main-4} and the trace inequality we further obtain
		\begin{equation}\label{main-5}
			\int_{S_\varepsilon^{\pm, \tau}} |\mathbf{W}^\pm_\varepsilon|^2 \, dl_xdt \le C_2 \, \varepsilon.
		\end{equation}
		
		Finally, substituting \eqref{main-5} into \eqref{main-1} yields the uniform estimate for the fracture part:
		\begin{equation}\label{main-6}
			\max_{t\in [0,T]} \|\mathbf{W}^{f}_\varepsilon(\cdot,t)\|_{L^2(\Omega_\varepsilon^{f})} \, + \, 
			\sqrt{\varepsilon} \, \||\nabla\mathbf{W}^{f}_\varepsilon|\|_{L^2(\Omega_\varepsilon^{f,T})} \le C_2 \, \varepsilon.
		\end{equation}
		Combining \eqref{function-L-bound} with \eqref{main-6}, we obtain estimate \eqref{main-est-2}.
	\end{proof}
	
	\begin{corollary}
		As a consequence of \eqref{main-est-1} and \eqref{main-est-2}, we have
		\begin{gather}
			\max_{t\in [0,T]} \|\mathbf{u}^\pm_\varepsilon(\cdot,t) - \mathbf{u}^\pm_0(\cdot,t)\|_{L^2(\Omega_\varepsilon^{\pm})}  \le \tilde{C}_1 \, \varepsilon^{1/2}, \notag 
			\\ \label{main-est-3}
			\max_{t\in [0,T]} \|\mathbf{u}^f_\varepsilon(\cdot,t) - \mathbf{w}^f_0(\cdot,t)\|_{L^2(\Omega_\varepsilon^{f})} \le \tilde{C}_2 \, \varepsilon,
		\end{gather}
		where $(\mathbf{u}^+_0, \mathbf{w}^f_0, \mathbf{u}^-_0)$ is the solution to the homogenized problem \eqref{homo-problem}. 
	\end{corollary}
	\begin{corollary}
		If $\Omega_\varepsilon^{f}$ is a cylindrical domain, i.e., $\Omega_\varepsilon^{f} = Q_\varepsilon$ (see \eqref{cylindr})  (in this case
		$T_0 = \emptyset,$ $G_\varepsilon = \emptyset,$
		and the corresponding boundary conditions in problems~\eqref{probl}, \eqref{potential0}, 
		\eqref{problem-N-1}, and \eqref{problem-N-2} are absent), then estimate \eqref{main-est-3} implies the inequality
		\begin{equation*}
			\max_{t\in [0,T]} \big\|E_\varepsilon[\mathbf{u}^f_\varepsilon](\cdot,t) - \mathbf{w}^f_0(\cdot,t)\big\|_{L^2(0, \ell)} \le \tilde{C}_3 \, \sqrt{\varepsilon},
		\end{equation*}
		where
		$$
		E_\varepsilon[\mathbf{u}^f_\varepsilon](x_1, t) := \frac{1}{\varepsilon \big(h_+(\frac{x_1}{\varepsilon}) + h_-(\frac{x_1}{\varepsilon})\big)}
		\int_{-\varepsilon h_-(\frac{x_1}{\varepsilon})}^{\varepsilon h_+(\frac{x_1}{\varepsilon})} \mathbf{u}^f_\varepsilon(x_1, x_2,t) \, dx_2.
		$$
	\end{corollary}
	
	\section{Conclusions}\label{Sect-6}
	
\textbf{1.} In this work we have analyzed the nonlinear reactive transport problem \eqref{probl} in a layered medium containing a thin, perforated, and geometrically oscillatory fracture. The fracture exhibits $\varepsilon$-periodically  varying aperture and $\varepsilon$-periodically distributed perforations, while transport within it is convection dominated with Péclet number of order $\varepsilon^{-1}.$ This interplay of geometric complexity, nonlinear boundary interactions, and strong advection yields, in the limit $\varepsilon \rightarrow 0,$ the homogenized problem \eqref{homo-problem}, which differs qualitatively  from previously studied thin-layer models. For this limit system we prove existence, uniqueness, and $C^3$-regularity of solutions.

In addition, we have constructed the approximation \eqref{app-functions} for the solution to problem \eqref{probl}, which captures in greater detail the effective fracture dynamics generated by strong advection and the microscale geometry. 
This approximation consists of the solution to the homogenized problem \eqref{homo-problem}; the boundary-layer correctors in the bulk domains, which, through the boundary-layer solutions $Z_1^{\pm }$ and $Z_2^{\pm },$ account for the boundary oscillations (see \S\ref{Par-4-2}); the functions $\{ N_1^{(k)}\} _{k=1}^{\mathcal{M}}$ solving the cell problems \eqref{problem-N-1} and reflecting the geometric complexity of the fracture together with the nonlinear boundary interactions; and the boundary-layer solutions $\{\Pi _{0,k}\}_{k=1}^{\mathcal{M}}$ and $\{\Pi_{1,k}\}_{k=1}^{\mathcal{M}}$ of problems \eqref{prim+probl+0} and \eqref{prim+probl+1}, respectively, which compensate for the remaining residuals and enforce the boundary condition on $\Gamma _{\ell ,\varepsilon }^T$ in \eqref{probl}.

Our main result is the derivation of quantitative error estimates for the difference between the solution of the original problem and the constructed approximation in the energy norm. These estimates are of order $\mathcal{O}(\varepsilon )$ in the bulk domains and $\mathcal{O}(\varepsilon^{1/2})$ in the fracture (Theorem~\ref{main theorem}). They provide a rigorous justification of the homogenized model and quantify the contribution of the microstructure to the macroscopic behavior.

It is important to emphasize the significance of the approximation, which provides detailed information about the structure of the solution for small values of the parameter $\varepsilon$  in a neighborhood of the fracture, thereby enabling more accurate modeling of complex physical phenomena. Convergence results alone show only what happens to the solution in the limit $\varepsilon =0$ and may be insufficient, since the limit may lose important information present in the original problem. For example, the approximation \eqref{app-functions} indicates that both the solution of the original problem \eqref{probl} and its gradient exhibit a distinctive, rapidly oscillating behavior along the interfaces of the fracture in the bulk domains, and this feature cannot be seen in the solution to the homogenized problem.

The analytical framework developed here can be adapted to other classes of thin-layer problems:
\begin{itemize}
	\item
	with Péclet number of order $\mathcal{O}(1)$; in this situation the homogenized subsystem on the flat interface becomes a parabolic system
	\[
	\partial_t \mathbf{w}^f_0  
	+ \widehat{\mathbb{D}}\, \partial^2_{x_1} \mathbf{w}^f_0
	+ \hat{\mathrm{v}}\, \partial_{x_1} \mathbf{w}^f_0
	= \widehat{\mathbf{F}}\big(\mathbf{u}^+_0(x_1,0,t),\, \mathbf{u}^-_0(x_1,0,t),\, \mathbf{w}^f_0(x_1,t),\, x_1, t\big),
	\]
	where $\widehat{\mathbb{D}} := \mathrm{diag}\big(\widehat{d}^{\,1}_{11},\ldots,\widehat{d}^{\,\mathcal{M}}_{11}\big)$ and each coefficient $\widehat{d}^{\,k}_{11}$ is defined in \eqref{hom-coeff};
	\item
	in three-dimensional configurations where a bulk domain is separated by a thin perforated plate with rapidly varying thickness, and where the perforated region may form a connected subset.
\end{itemize}

\smallskip 

\textbf{2.}
Our results demonstrate that the parameter choice  \(\alpha=\beta=1\) in the intensity factors of the nonlinear boundary interactions in the fracture represents a  critical threshold (see Remark~\ref{remark-2-6}). In this case, the boundary interactions are incorporated in the homogenized hyperbolic subsystem through the vector function $\widehat{\mathbf{F}}=(\widehat{F}_1,\ldots,\widehat{F}_\mathcal{M})$ defined in \eqref{hat F}. If 
$\alpha >1$ and $\beta > 1,$ then, as follows from formal calculations in Section~\ref{Sect-3}, the homogenized problem splits into two independent problems in the bulk domains $\Omega^+$ and $\Omega^-$ with the following boundary conditions on the flat interface:
$$
\pm\, \mathbb{D}^\pm \partial_{x_2} \mathbf{u}^\pm_0 =  
\upharpoonleft\!\! {S}^\pm\!\!\upharpoonright_1\widetilde{\boldsymbol{\Upsilon}}^\pm\big(\mathbf{u}^\pm_0(x_1,0,t), \mathbf{0}, x_1, t\big).
$$
The main challenge in this regime is to detect the influence of reactive transport in the fracture in the higher-order terms of the asymptotics, which will depend on the parameters $\alpha$ and $\beta.$ This constitutes a natural direction for future research, aimed at understanding different scaling regimes for the boundary interactions.

Another promising direction is the development of numerical methods for the homogenized problem \eqref{homo-problem}. Its structure, which combines nonlinear diffusion–reaction equations in the bulk with a first‑order semilinear hyperbolic subsystem on the interface and nonlinear transmission conditions, lies outside the scope of standard finite element or finite volume schemes for interface problems. Dedicated discretization strategies are therefore required to handle the mixed hyperbolic–parabolic character and the nonlinear coupling.

The regularity results established for the solution to the homogenized system provide a solid analytical foundation for such developments. As in classical finite element analysis, higher regularity is expected to ensure optimal convergence rates, facilitate the construction of stable interface fluxes, allow accurate reconstruction of characteristic curves for the hyperbolic subsystem, and enable the use of high‑order schemes and adaptive refinement strategies that are essential for resolving the interplay between bulk diffusion and interface transport.

	\section*{Acknowledgments}
	The first author gratefully acknowledges the support of  the MSCA4Ukraine grant, which made it possible to carry out this research at the University of Stuttgart. This project has received funding through the MSCA4Ukraine project, which is funded by the 
	European Union. Views and opinions expressed are however those of the author(s) only and do not 
	necessarily reflect those of the European Union, the European Research Executive Agency or the 
	MSCA4Ukraine Consortium. Neither the European Union nor the European Research Executive 
	Agency, nor the MSCA4Ukraine Consortium as a whole nor any individual member institutions of 
	the MSCA4Ukraine Consortium can be held responsible for them.
	
	All authors acknowledge  (partial) funding from the Deutsche Forschungsgemeinschaft (DFG, German Research Foundation) – Project Number 327154368 – SFB 1313.


\end{document}